\documentclass[11pt,A4rpaper]{amsart}
\usepackage{url,amsmath,amsthm, amssymb}
\usepackage[margin=1in]{geometry}
\usepackage[pdfencoding=auto, psdextra, colorlinks=true, linkcolor = blue, urlcolor=blue, citecolor=blue,pagebackref,breaklinks]{hyperref}
\usepackage{bookmark}
\usepackage[all]{mathtools}
\usepackage{makecell}
\usepackage{mathrsfs}
\usepackage{mathdots}
\usepackage{amsxtra,amscd}
\usepackage{amsrefs}
\usepackage{hyperref}
\usepackage{enumerate, enumitem}
\usepackage{arydshln}
\usepackage{multirow}
\usepackage{bbm} 
\usepackage[all,arc]{xy}

 \usepackage{amssymb,amsrefs}
\usepackage{amsmath}
\usepackage{amscd,latexsym}
\usepackage{bookmark}
 \usepackage{amssymb,amsfonts,bm}
\usepackage[all,arc]{xy}
\usepackage{enumerate}
\usepackage{mathrsfs}
\usepackage{amscd}
\usepackage{tikz}
\usepackage{tikz-cd}
\usepackage{amsmath}
\usepackage{latexsym}
\usepackage{mathdots}
\usepackage{amsrefs}
\usepackage{graphicx}
\usepackage{mathtools}
\usepackage{leftidx}
\usepackage{tensor}
\usepackage{dsfont}
\usepackage{tikz-cd} 
\usepackage[new]{old-arrows}
\usepackage{xcolor}
\usepackage{bbm}
\usepackage{enumitem}
\usepackage{cleveref}

\newcommand{\bigboxplus}{
	\mathop{
		\vphantom{\bigoplus} 
		\mathchoice
		{\vcenter{\hbox{\resizebox{\widthof{$\displaystyle\bigoplus$}}{!}{$\boxplus$}}}}
		{\vcenter{\hbox{\resizebox{\widthof{$\bigoplus$}}{!}{$\boxplus$}}}}
		{\vcenter{\hbox{\resizebox{\widthof{$\scriptstyle\oplus$}}{!}{$\boxplus$}}}}
		{\vcenter{\hbox{\resizebox{\widthof{$\scriptscriptstyle\oplus$}}{!}{$\boxplus$}}}}
	}\displaylimits 
}

\allowdisplaybreaks

\theoremstyle{plain}
\newtheorem{theorem}{Theorem}[section]

\newtheorem{induction hypothesis}[theorem]{Induction Hypothesis}

\newtheorem{lemma}[theorem]{Lemma}
\newtheorem{proposition}[theorem]{Proposition}
\newtheorem{working hypothesis}[theorem]{Working Hypothesis}
\numberwithin{equation}{section}

\theoremstyle{definition}
 
  \newtheorem*{data availability}{Data availability}
  \newtheorem*{conflicts of interest}{Conflicts of interest}

\theoremstyle{remark}

\newtheorem*{acknowledgements}{Acknowledgments}

\makeatletter
\newcommand{\mytag}[2]{%
  \text{#1}%
  \@bsphack
  \begingroup
    \@onelevel@sanitize\@currentlabelname
    \edef\@currentlabelname{%
      \expandafter\strip@period\@currentlabelname\relax.\relax\@@@%
    }%
    \protected@write\@auxout{}{%
      \string\newlabel{#2}{%
        {#1}%
        {\thepage}%
        {\@currentlabelname}%
        {\@currentHref}{}%
      }%
    }%
  \endgroup
  \@esphack
}

\begin{document}

\title[Whittaker--Fourier Coefficients]{Petersson Inner Products and Whittaker--Fourier Periods on Even Special Orthogonal and Symplectic Groups}

\author{Yeongseong Jo}

\address{Department of Mathematics Education, Ewha Womans University, Seoul 03760, Republic of Korea}
\email{\href{mailto:jo.59@buckeyemail.osu.edu}{jo.59@buckeyemail.osu.edu};\href{mailto:yeongseong.jo@ewha.ac.kr}{yeongseong.jo@ewha.ac.kr}}


\dedicatory{Dedicated to James Cogdell with admiration}

\begin{abstract}
In this article, we would like to formulate a relation between the square norm of Whittaker--Fourier coefficients on even special orthogonal and symplectic groups
and Petersson inner products along with the critical value of $L$-functions up to constants. We follow the path of Lapid and Mao to reduce it to the conjectural local identity. 
Our strategy is based on the work of Ginzburg--Rallis--Soudry on automorphic descent. We present the analogue result for odd special orthogonal groups, which is conditional on 
unfolding Whittaker functions of descents.
\end{abstract}

\maketitle

\section{Introduction}
Let ${\rm GL}_n(\mathbb{A})$ be a general linear group over a number field $F$ and $\mathbb{A}$ the ring of adeles of $F$.
We let $\pi$ denote an irreducible cuspidal representation.
We let $B$ denote a Borel subgroup of ${\rm GL}_n$, $N$ the unipotent radical of $B$ 
and fix a non-degenerate character $\psi_N$ of $N(\mathbb{A})$, trivial on $N(F)$. 
For a cusp form $\varphi$ of ${\rm GL}_n(F) \backslash {\rm GL}_n(\mathbb{A})$, we consider the Whittaker--Fourier coefficient
\[
 \mathcal{W}^{\psi_N}(\varphi):=\int_{N(F)\backslash N(\mathbb{A})} \varphi(u)\psi^{-1}_N(u) \, du.
\]
If $\pi$ is an irreducible cuspidal representation, then non-zero Whittaker--Fourier coefficient $\mathcal{W}^{\psi_N}(\varphi)$
gives a realization of $\pi$ in the space of Whittaker functions on ${\rm GL}_n(\mathbb{A})$. By local multiplicity one,
 the space of Whittaker functions depends only on $\pi$ as an abstract representation and $\pi$ is decomposable in the sense that
\[
\pi \cong \otimes_v \mathbb{W}^{\psi_{N,v}}(\pi_v)
\]
where $\mathbb{W}^{\psi_{N,v}}(\pi_v)$ is the local Whittaker model of $\pi_v$. Upon further assuming that $\mathcal{W}^{\psi_N}(\varphi)$ is decomposable, we have a factorization of Whittaker functions
\[
  \mathcal{W}^{\psi_N}(\varphi)=\prod_v W_v, \quad W_v \in \mathbb{W}^{\psi_{N,v}}(\pi_v).
\]
 It therefore provides a useful tool for understanding $\pi$, both computationally and conceptually. 
 
\par
The construction of Whittaker functions has had numerous significant impact in global integrals of Rankin-Selberg type.  The prototype integral representation is the original construction for ${\rm GL}_2 \times {\rm GL}_2$ by R. Rankin and A. Selberg in late 1930.
Their analogues for higher rank has been developed and studied in the pioneering work by Jacquet, Piatetski--Shapiro, and Shalika \cite{JS81a, JS81b, JPSS83}
with some parts completed only in recent years by Jacquet \cite{Jac09}. 
The Rankin-Selberg integral has been investigated extensively in the literature
and are the subject matter in modern number theory of automorphic representations and their $L$-functions. 
Among other things, the Rankin--Selberg integral expresses the Petersson
inner product in terms of a canonical inner product on the local Whittaker model of $\pi_v$ under the decomposition $\pi=\otimes_v \pi_v$ \cite[\S 4]{JS81a}. The global factor which shows up
in this expression is the residue at $s=1$ of the adjoint $L$-function $L(s,\pi,{\rm Ad})$ of $\pi$ (or alternatively, $L(s,\pi \times \pi^{\vee})$)
where $\pi^{\vee}$ is the contragredient of $\pi$.

\par
We try to make the statement more precise, and to generalize it to other groups. Let ${\rm GL}_n(\mathbb{A})^1=\{ g \in {\rm GL}_n(\mathbb{A}) \;|\; |\det g|=1  \}$. It is know from \cite[\S 4]{LM15} that ${\rm vol}({\rm GL}_n(F) \backslash {\rm GL}_n(\mathbb{A})^1)=1$. We let $\zeta^S_F(s)=\prod_{v \notin S} \zeta_{F_v}(s)$, $\Re s>1$ be the partial Dedekind zeta function.
Similar notation applies to  the partial adjoint $L$-function $L^S(s,\pi,{\rm Ad})$. For a cusp form $\varphi$ and another cusp form $\varphi^{\vee}$ on ${\rm GL}_n(F) \backslash {\rm GL}_n(\mathbb{A})$,
 we define the Petersson inner product by
\[
  \langle \varphi, \varphi^{\vee} \rangle_{{\rm GL}_n(F) \backslash {\rm GL}_n(\mathbb{A})^1}:=\int_{{\rm GL}_n(F) \backslash {\rm GL}_n(\mathbb{A})^1} \varphi(g)\varphi^{\vee}(g)\, dg,
\]
while on $\mathbb{W}^{\psi_{N,v}}(\pi_v)$ an invariant inner product is given by
\[
 J_v\left( 1, W_v, \check{W}_v \right):=\int_{ ( {\rm GL}_{n-1}(F_v) \cap N(F_v) ) \backslash {\rm GL}_{n-1}(F_v)} W_v\begin{pmatrix} g & \\ & 1 \end{pmatrix} \check{W}_v \begin{pmatrix} g & \\ & 1 \end{pmatrix} \, dg
\]
for non-archimedean places $v$ of $F$ \cite[Proposition 3.4]{Jo23} and archimedean places $v$ of $F$ \cite[Remark 5.8]{HJ24}.
It is natural to compare the local invariant inner product with the Petersson inner product. The relation has been established in the pioneering work of Jacquet and Shalika \cite[\S 4]{JS81a};
\begin{equation}
\label{EulerProducts-General-Linear}
 \langle \varphi, \varphi^{\vee} \rangle_{{\rm GL}_n(F) \backslash {\rm GL}_n(\mathbb{A})^1}
 =\frac{L^S(1,\pi, {\rm Ad})}{\prod_{j=1}^n\zeta_F^S(j)} \prod_{v \in S} J_v\left( 1, W_v, \check{W}_v \right).
\end{equation}
(We ask readers to refer to \citelist{\cite{Zha14}*{Proposition 3.1} \cite{LM15}*{p.477, Lemma 4.4}}). 
Recently the above product formula can be further refined by the product of so-called {\it na\"ive Rankin--Selberg $L$-factors} over all places at $s=1$,
upon taking the {\it global newform}  \cite[Definition 8.1]{HJ24} and then specializing the value of the complex variable $s$ in the period integral \cite[Theorem 8.4]{HJ24}.

\par
As a consequence of the formula of Jacquet and Shalika \eqref{EulerProducts-General-Linear}, we take the $\psi_N$-th Fourier coefficient of a matrix coefficient 
of $\pi$ and turn our attention to associating it with the product of Whittaker functions. The integral
\[
 \int_{N(\mathbb{A})} \langle \pi(u)\varphi,\varphi^{\vee} \rangle_{{\rm GL}_n(F) \backslash {\rm GL}_n(\mathbb{A})^1} \psi_N^{-1}(u) \, du
\]
does not converge. In fact, even the local integrals
\[
 \int_{N(F_v)} \langle \pi_v(u_v) \varphi_v,\varphi_v^{\vee} \rangle_v \psi^{-1}_N(u_v) \, du_v
\]
(where $\varphi_v$ and $\varphi_v^{\vee}$ are now taken from the space of $\pi_v$ and $\pi^{\vee}_v$ respectively and $\langle \cdot,\cdot \rangle_v$ is the canonical pairing)
do not converge unless $\pi_v$ is square-integrable. Given a finite set of places $S$, it is possible to regularize the integral in such a way that
\[
 \int^{\rm st}_{N(F_S)} \langle \pi(u)\varphi,\varphi^{\vee} \rangle_{{\rm GL}_n(F)\backslash {\rm GL}_n(\mathbb{A})^1} \psi^{-1}_N(u) \, du
\]
just as in \cite{LM15}. In particular, if $S$ consists solely of non-archimedean places, then
\[
 \int^{\rm st}_{N(F_S)} \langle \pi(u)\varphi,\varphi^{\vee} \rangle_{{\rm GL}_n(F)\backslash {\rm GL}_n(\mathbb{A})^1} \psi^{-1}_N(u) \, du
 = \int_{N^{\ast}} \langle \pi(u)\varphi,\varphi^{\vee} \rangle_{{\rm GL}_n(F)\backslash {\rm GL}_n(\mathbb{A})^1} \psi^{-1}_N(u) \, du
\]
for any sufficiently large compact open subgroup $N^{\ast}$ of $N(F_S)$. The definition of the regularized integral is different in the archimedean case (cf. \cite[\S 4.3]{Mor24}), but in this paper
we do not use regularized integrals over archimedean fields directly. 
From the viewpoint of local multiplicity one, there exist constants $c_{\pi_v}$ depending on $\pi_v$ such that
\begin{equation}
\label{Main-Result-GeneralLinear}
 \mathcal{W}^{\psi_N}(\varphi)\mathcal{W}^{\psi^{-1}_N}(\varphi^{\vee})
 =\left(\prod_{v \in S} c^{-1}_{\pi_v} \right)\frac{\prod_{j=1}^n\zeta_F^S(j)}{L^S(1,\pi\times\pi^{\vee})}
 \int^{\rm st}_{N(F_S)} \langle \pi(u) \varphi, \varphi^{\vee} \rangle_{{\rm GL}_n(F) \backslash {\rm GL}_n(\mathbb{A})^1} \psi^{-1}_{N}(u) \; du
\end{equation}
for all $\varphi=\otimes_v \varphi_v \in \pi$, $\varphi^{\vee}=\otimes_v \varphi^{\vee}_v \in \pi^{\vee}$, and all $S$ sufficiently large \cite[(1.2)]{LM15}.
Implicit here is the existence and non-vanishing of $\left.\dfrac{\prod_{j=1}^n\zeta_F^S(js)}{L^S(s,\pi\times\pi^{\vee})}\right|_{s=1}$.
Lapid and Mao \cite[Theorem 4.1, Lemma 4.4]{LM15} utilize the theory of Rankin--Selbeg integrals to show that $c_{\pi_v}=1$
over all places $v$ for irreducible cuspidal representation of ${\rm GL}_n$, while a similar result was proved independently in the work of Sakellaridis and Venkatesh \cite[\S 6.3, \S 18.3]{SV17}. We mention in this context a slightly modified formulation presented in  \cite[\S 2.2]{Liu16}.
\par
It is desirable to extend this relation \eqref{Main-Result-GeneralLinear} based on the analog of the equality \eqref{EulerProducts-General-Linear} to symplectic and special orthogonal groups. The purpose of the paper is to compute Whittaker--Fourier coefficients on symplectic and special orthogonal groups to reduce the global conjectural identity to a local corresponding identity.
\begin{itemize}[leftmargin=*]
\item In Theorem \ref{Main-Result-Even-Orthogonal-Theorem}, we establish the explicit formula relating the product of Whittaker--Fourier coefficients
of cusp forms on adelic quotients of split and quasi-split even special orthogonal groups to the Petersson inner product. The formulation relies on the validity of Working Hypothesis \ref{even-orthogonal-working-hypothesis} and is built upon Proposition \ref{even-orthogonal-good-repn}.
\item In Theorem \ref{Main-Result-Symplectic-Theorem}, we establish the explicit formula relating the product of Whittaker--Fourier coefficients of
cusp forms on adelic quotients of symplectic groups to the Petersson inner product. The formulation relies on the validity of
Working Hypothesis \ref{symplectic-working-hypothesis} and is built upon Proposition \ref{symplectic-good-repn}.
\end{itemize}
In Theorem \ref{Main-Result-Odd-Orthogonal-Theorem}, we also study the formula relating the square norm of Whittaker--Fourier coefficients
on adelic quotients of split odd special orthogonal groups to the product of and special value of automorphic partial $L$-functions under Working Hypothesis \ref{odd-orthogonal-working-hypothesis}. Whereas the formulation in odd special orthogonal cases  is built upon Proposition \ref{odd-orthogonal-good-repn} as before,
it relies on the ongoing work of Lapid and Mao (Theorem \ref{Whittaker-main-odd-orthogonal}), which concerns with unfolding Whittaker functions of the global descent generated by Gelfand--Graev coefficients.  It is noteworthy that the bulk of Section \ref{Sec:5} is intended for gathering more or less expository groundwork toward a local conjecture identity,
rather than concentrating on proving precise statements.

\par
In this note, we follow the path of the adaptation of  \cite{LM16} to study Whittaker--Fourier coefficients for unitary groups and shed a light on the product of Whittaker--Fourier coefficients \eqref{Main-Result-GeneralLinear}.
The key ingredient is the theory of automorphic descent developed by Ginzburg--Rallis--Soudry \cite{GRS99,GRS11}. 
We address that this approach is different from that of Waldspurger \cite{Wal81} who uses the Shimura correspondence 
as well as from the relative trace formula approach \cite{Zha14}.
At the moment, the descent theory is more thoroughly understood in the case of metaplectic groups than
in the case of special orthogonal and symplectic groups. In addition there are some expected properties of the descent, specifically 
irreducibility of global descents (Working Hypothesis \ref{even-orthogonal-working-hypothesis} and \ref{symplectic-working-hypothesis}) in the even special orthogonal and symplectic case 
just as assumed in \cite{LM16}, and irreducibility of global descents (Working Hypothesis \ref{odd-orthogonal-working-hypothesis}) and Theorem \ref{Whittaker-main-odd-orthogonal} of Lapid--Mao
in odd special orthogonal case. Although it is likely that the similar method in metaplectic case work, we do not take up the issue of bridging theses gaps here.
This is because the ultimate goal is to settle down the local identity, once the reduction to the local identity from the global identity is achieved in the current project.
In this regard, we will take for granted the expected properties of the descent for special orthogonal and symplectic groups. Therefore our results are conditional.

\par
In the case of classical groups of different rank cases, it is one of the most important topic in modern number theory to explore these types of deep relations between certain period integrals of automorphic forms and critical value of $L$-functions. In 1990's, Gross and Prasad \cite{GP92,GGP12} predicted a relationship between the non-vanishing of Gross-Prasad period on special orthogonal groups of co-rank $1$ and the non-vanishing of certain $L$-functions. Recently the influential work of Ichino and Ikeda \cite{II10}
has refined the global Gan--Gross-Prasad conjecture \cite{GGP12} to a conjectural formula which computes the norm of the period integral in terms of the central $L$-value
in the case of orthogonal group of co-rank $1$. The construction of period integrals for cuspidal representations of unitary groups ${\rm U}_n \times {\rm U}_m$
depends on the parity of $n-m$, so one has to treat the case $n-m \neq 0$ even (Fourier--Jacobi cases) and $n-m$ odd (Bessel cases) separately. For unitary groups ${\rm U}_{n+1} \times {\rm U}_n$ and $n=1,2$,
the similar work was later accomplished by R. N. Harris \cite{Har14}. W. Zhang \cite{Zha14} has recently made a remarkable progress toward refined formula for ${\rm U}_{n+1} \times {\rm U}_n$ taking the relative trace formula approach initiated by Jacquet \cite{Jac89}. The reader may consult to \cite{BLX24,Liu16}
 for incomplete list of Jacobi and Bessel cases on unitary groups.

\par
 The case of the metaplectic group $\widetilde{\rm SL}_2(=\widetilde{\rm Sp}_2)$-the two-fold cover of the special linear group ${\rm SL}_2$ goes back to the seminal work of
 Waldspurger on the Fourier coefficients of half-integral weight modular forms \cite{Wal81}. 
 His result has been revisited in a flurry of work since then (see \cite{LM15,LM17,Qiu19} and the references therein). 
 Later Lapid and Mao \cite{LM17} generalize the relation between the Whittaker--Fourier coefficients and the Petersson inner product to
 arbitrary metaplectic groups $\widetilde{\rm Sp}_{2n}$. While these groups are not algebraic, they behave in many ways like algebraic groups. 
 More importantly, the descent method applies to metaplectic groups and give rise to irreducible representations. The formula relating the Whittaker--Fourier coefficients to the Peterson inner product for quasi-split unitary groups is carried out by Lapid and Mao \cite{LM15} but under the assumption of irreducibility of global descents.
 
 \par
 As seen in \eqref{Main-Result-GeneralLinear}, for all places $v$ of $F$, the local constant $c_{\pi_v}$ depending on $\pi_v$ is given as a proportionality constant of two local Whittaker periods. Afterwords, Lapid and Mao \cite{LM17-2} determine a value for the $c_{\pi_v}$, which is provided with a root number
 $c_{\pi_v}=\varepsilon(\frac{1}{2},\pi_v,\psi_v)$ in the case of metaplectic groups though they only consider $p$-adic places $v$.
 Very recently, Morimoto  \cites{Mor24} resolves what is known as {\it Lapid--Mao conjecture}, which specifies the value $c_{\pi_v}$ as the central character $\omega_{\pi}$ of $\pi$ evaluated at $\tau$ satisfying $-\tau=\mathfrak{c}(\tau)$ for all places $v$ in the case of unitary groups. Here $\mathfrak{c}$ is a Galois involution over a quadratic extension of $F_v$. 
It turns out that this result is preceded by his earlier work on even unitary groups at non-split finite places \cite{Mor22}. The main significance of this task is that
the presumptive local identity is expected to be equivalent to the formal degree conjecture of Hiraga--Ichino--Ikeda \cite{HII08}
in the case of generic discrete series representations.
Ichino--Lapid--Mao \cite{ILM17} use this method to establish the formal degree conjecture for the case of odd special orthogonal groups and metaplectic groups.
For unitary group cases, Morimoto \cite{Mor22,Mor24} generalizes the technique to work out the formal degree conjecture.
It is perhaps fair to raise a natural question as to analyzing the exact value of $c_{\pi_v}$ in the context of special orthogonal and symplectic groups.
 As was pointed out briefly, we follow the robust argument proposed by Lapid and Mao to attack the explicit identity in author's ongoing collaborations. However we think that our current content keeps our exposition a reasonable length and we hope to do so in the future. 
 
 \par
 Let us describe the organization of this paper in more detail. We set up basic notations in Section \ref{Sec:2}.
In Section \ref{Sec:3} through Section \ref{Sec:5}, we introduce and study the descent construction, both locally and globally.
We then explore analytic properties of Rakin--Selberg--Shimura type integral and compute it for unramified data.
After piecing together ingredients of the descent constructions and Euler products from Petersson inner products, 
we present our main result which allows us to reduce to a local statement. Section \ref{Sec:3} is devoted to the case for
both quasi-split and split even orthogonal groups and we dedicate Section \ref{Sec:4} to the case for symplectic groups.
 In Section \ref{Sec:5}, we present results in the case of odd special orthogonal groups, which parallel the analogue of previous cases.

\vspace{-.1cm}
\section{Notation and Preliminaries}
\label{Sec:2}

\subsection{Local settings}
Let $F$ be a local field of character zero. By abuse of notation, we will use the same letter for
an algebraic group over $F$ and its group of $F$-points. We denote by ${\rm Irr} \, Q$ the set of 
equivalence classes of smooth complex irreducible representations of the group of $F$-points
of an algebraic group $Q$ over $F$. We also write $\delta_Q$ for the modulus function of $Q$.
If $Q$ is quasi-split and $\psi_{N_Q}$ is a non-degenerate character of a maximal unipotent subgroup $N_Q$,
we denote by ${\rm Irr}_{{\rm gen},\psi_{N_Q}} Q$ the subset of representations that are generic with 
respect to $\psi_{N_Q}$. We suppress $\psi_{N_Q}$ from the notation and write it as ${\rm Irr}_{\rm gen}Q$,
if it is clear from the context or is irrelevant. We denote by $\boxplus$ the normalized parabolic induction for the general 
linear group, which we call the isobaric sum.  Specifically, we let ${\rm P}_{(n_1,\dotsm,n_r)}(F)$ be the standard upper parabolic subgroup of ${\rm GL}_n(F)$
with $n_1+\dotsm+n_r=n$. Given representations $\pi_i$ of ${\rm GL}_{n_i}(F)$
we form an induced representation $\pi$ of ${\rm GL}_n(F)$ by normalized parabolic induction;
\[
 \pi:= \bigboxplus_{i=1}^r \pi_i={\rm Ind}^{{\rm GL}_n(F)}_{{\rm P}_{(n_1,\dotsm,n_r)}(F)}(\pi_1 \otimes \dotsm \otimes \pi_r).
\]
\par
We caution that the modular character $\delta_Q$ in Lapid and Mao \cite{LM16,LM17} is really $\delta^{-1}_Q$ in our notation.
Our convention is consistent with that in Kaplan \cite{Kap15}.

\par
Let $F^n$ be  a $n$-dimensional space of column vectors. 
We let $\{ ^t\xi_i \, | \, 1 \leq i \leq n  \}$ be the standard column basis for $F^n$.
We let $1_n$ denote the $n \times n$ identity matrix
and $0_{m,l}$ the $m \times l$ zero matrix. Let
\[
  w_n=\begin{pmatrix}  \vspace{-2ex} && 1 \\ & \iddots& \vspace{-2ex} \\ 1 && \end{pmatrix}
\]
be the $n \times n$ matrix with ones on non-principal diagonal and zeros elsewhere. 
Let ${\rm Mat}_{l,m}$ be the vector space of $l \times m$ matrices. 
For $x \in {\rm Mat}_{l,m}$, let $\check{x}$ be the twisted transpose matrix in ${\rm Mat}_{m,l}$ given by $\check{x}=w_m\prescript{t}{}{x} w_{l}$. 
Let $g \mapsto g^{\ast}$ be the outer automorphism of ${\rm GL}_n$ given by $g^{\ast}=w^{-1}_n  \prescript{t}{}{g^{-1}} w_n$. 
Let $\tau \in F^{\times}$ and we let $G'$ denote one of the quasi-split group ${\rm SO}_{2n+2,\tau}$, the split group ${\rm Sp}_{2n}$ or  the split group ${\rm SO}_{2n+1}$.
Let $\mathbb{M}'={\rm GL}_n$. We let $Z_n=N'_{\mathbb{M}'}$ denote the group of unipotent upper triangular matrices of ${\rm GL}_n$. Let $\mathfrak{t}={\rm diag}(1,-1,\dotsm,(-1)^{n-1}) \in {\rm GL}_n$. Let $K_{{\rm GL}_n(F)}$ be the standard maximal compact subgroup of ${\rm GL}_n(F)$. 
Let $\widetilde{\rm GL}_n$ be the two-fold metaplectic cover of ${\rm GL}_n$. We write element of $\widetilde{\rm GL}_n$ as pairs $(g,\varepsilon)$, $g \in {\rm GL}_n$,
$\varepsilon=\pm 1$, where the multiplication rule is given by Banks--Levy--Sepanski's cocycle  (Refer to, for example, \cite[\S 1.1]{Tak14} for detailed conventions.) When $g \in {\rm GL}_n$,
 we write $\tilde{g}=(g,1) \in \widetilde{\rm GL}_n$. For each subgroup $H \subseteq {\rm GL}_n$, we denote the preimage of $H$ via the canonical projection
 $\widetilde{\rm GL}_n \rightarrow {\rm GL}_n$ by $\tilde{H}$ when the base field is clear from the context.
 We let $\delta'_{n,2n}$ and $\delta''_{n,2n}$ be $n \times 2n$ matrices
such that
\[
  \delta'_{n,2n}=\begin{pmatrix} 1 &\vspace{-.2cm}  & & & \\ & 0 \; 1 & \vspace{-.2cm}& &  \\ & & 0\; 1 &\vspace{-.2cm} &  \\ &  & & \ddots &\vspace{-.2cm} \\ 
  & & & & 1\; 0 \end{pmatrix} \quad \text{and} \quad \delta''_{n,2n}=\begin{pmatrix} 0\;1 &\vspace{-.2cm}&&&\\ &0\; 1&\vspace{-.2cm}&&\\ \ &&\ddots&\vspace{-.2cm}&\\ &&&1&\vspace{-.2cm} \\ &&&& 0\; 1 \end{pmatrix}.
\]
Let $\epsilon_{i,j} \in {\rm Mat}_{l,m}$ be the matrix one at the $(i,j)-$entry and zero elsewhere. 

\subsection{Global settings}
Let $F$ be a number field and $\mathbb{A}$ its ring of adeles. 
Most of the previous notation has an obvious meaning in the global context. 
Some changes in the global context are as follows. Let $S$ be a finite set of places of $F$. Let $F_S=\prod_{v \in S} F_v$ 
and $\mathcal{O}_S$ be the ring of $S$-integers of $F$.
Let $|\cdot|$ denote the standard ideal norm of $\mathbb{A}^{\times}$. We put $\mathbb{A}^1=\{ a\in \mathbb{A}^{\times}\,|\, |a|=1 \}$
and ${\rm GL}_n(\mathbb{A})^1=\{ g \in {\rm GL}_n(\mathbb{A}) \;|\; \det g \in \mathbb{A}^1 \}$. The global double cover $\widetilde{\rm GL}_n(\mathbb{A})$
of ${\rm GL}_n(\mathbb{A})$ is compatible with the local double covers at each place. (See, for example, \cite[\S 1.2]{Tak14} for complete regards.)

\par
We write ${\rm Cusp} \, Q$ for the set of equivalence class of irreducible cuspidal automorphic representations of $Q(\mathbb{A})$.
For any automorphic forms $\varphi$, $\varphi^{\vee}$ on $G'(F) \backslash G'(\mathbb{A})$, let $\langle \varphi,\varphi^{\vee} \rangle_{G'}$ be the canonical inner product
 on $G'(F)\backslash G'(\mathbb{A})$ defined by
 \begin{equation}
 \label{Petterson-Innerproduct}
  \langle \varphi,\varphi^{\vee} \rangle_{G'}:={\rm vol}(G'(F)\backslash G'(\mathbb{A}))^{-1} \int_{G'(F)\backslash G'(\mathbb{A})} \varphi(g) \varphi^{\vee}(g) \, dg,
 \end{equation}
 provided that the integral converges. Let $N'$ be a maximal unipotent subgroup of $G'$ and $\psi_{N'}$ a fixed non-degenerate character of $N'(\mathbb{A})$, trivial on $N'(F)$.
For an automorphic form $\varphi$ on $G'(\mathbb{A})$, we define the Whittaker--Fourier coefficient $\mathcal{W}^{\psi_{N'}}(g,\varphi)$ by
\[
 \mathcal{W}^{\psi_{N'}}(g,\varphi):={(\rm vol}(N'(F) \backslash N'(\mathbb{A})))^{-1}\int_{N'(F) \backslash N'(\mathbb{A})} \varphi(ng) \psi^{-1}_{N'}(n) \, dn, \quad g \in G'(\mathbb{A}).
\]
We abbreviate $\mathcal{W}^{\psi_{N'}}(e,\varphi)$ to $\mathcal{W}^{\psi_{N'}}(\varphi)$. 
Let $\mathcal{S}(\mathbb{A}^n)$ be the space of Schwartz-Bruhat functions on $\mathbb{A}^n$.
Given a finite set of places $S$ of $F$, we defined in \cite{LM15} 
a regularized integral
\[
 \int^{\rm st}_{N'(F_S)} f(u) \, du
\]
for a suitable class of smooth functions $f$ on $N'(F_S)$. If $S$ consists only of non-archimedean places then
\[
 \int^{\rm st}_{N'(F_S)} f(u) \,du=\int_{N'_1} f(u) \, du
\]
for any sufficiently large compact open subgroup $N'_1$ of $N'(F_S)$, whereas an ad-hoc definition is given in the archimedean case.

\section{The Case for Even Special Orthogonal Groups}
\label{Sec:3}

\subsection{Groups, Embeddings, and Weyl Elements} Let $G={\rm SO}_{4n+3}$ be the even split special orthogonal group defined by
\[
G= {\rm SO}_{4n+3}=\{ g \in {\rm GL}_{4n+3} \,|\, \prescript{t}{}{g} w_{4n+3} g=w_{4n+3}  \}.
\]
The group $G$ acts on the space $F^{4n+3}$ with the standard basis
$e_1, \dotsm, e_{2n+1},e_0,e_{-1-2n},\dotsm,e_{-1}$. We denote by $\langle \cdot,\cdot \rangle$ the symmetric form defined on the space $F^{4n+3}$ by
$\langle u,v \rangle=\prescript{t}{}{u} w_{4n+3} v$  for any $u, v \in F^{4n+3}$. Let $\tau \in F^{\times}$ and we define $G' ={\rm SO}_{2n+2,\tau} \subseteq G$ consisting of elements fixing
$e_1,\dotsm,e_{n},e_{-n},\dotsm,e_{-1}$ and $e_{2n+1}+\frac{\tau}{2} e_{-1-2n}$. Let $\mathcal{V}$ be the orthogonal complement of the space generated by elements
$e_1, \dotsm, e_{n},e_{-n}, \dotsm, e_{-1}$ and $e_{2n+1}+\frac{\tau}{2} e_{-1-2n}$.

\par
When $\tau=\beta^2$ is a square, we use 
\[
\mathcal{B}_{G'}=\{e_{n+1}, 
\dotsm, e_{2n}, e_0+\frac{1}{\beta}e_{2n+1}-\frac{\beta}{2}e_{-1-2n}, \frac{1}{2}e_0-\frac{1}{2\beta}e_{2n+1}+\frac{\beta}{4}e_{-1-2n}, e_{-2n}, \dotsm, e_{-n-1} \}.
\]
 \cite[\S 9.3]{GRS11} as an ordered basis for $\mathcal{V}$. The group $G'$ is given by
 \[
  G'={\rm SO}_{2n+2,\tau}=\{ g \in {\rm GL}_{2n+2} \,|\, \prescript{t}{}{g} w_{2n+2} g=w_{2n+2}  \}.
 \]
The group $G'$, viewed as an algebraic group over $F$, is called {\it split} if $\tau \in F^2$.

\par
When $\tau$ is not a square, we take $\mathcal{B}_{G'}=\{e_{n+1}, 
\dotsm, e_{2n}, e_0, e_{2n+1}-\frac{\tau}{2} e_{-1-2n}, e_{-2n}, \dotsm, e_{-n-1} \}$ \cite[\S 9.3]{GRS11} as an ordered basis for $\mathcal{V}$.
The group $G'$ is given by
\[
 G'={\rm SO}_{2n+2,\tau}=\{ g \in {\rm GL}_{2n+2} \,|\, \prescript{t}{}{g} \begin{pmatrix*}[r] &&\vspace{-.1cm} & w_{n} \\ &1 & 0& \\ \vspace{-.1cm} & 0 & -\tau & \\ w_{n}& & & \end{pmatrix*} g=
 \begin{pmatrix*}[r] && \vspace{-.1cm} & w_{n} \\ &1 & 0& \\ \vspace{-.1cm} & 0 & -\tau & \\ w_{n}& & & \end{pmatrix*}  \}.
\]
The group $G'$, viewed as an algebraic group $F$, is called {\it quasi-split} if it is non-split over $F$ and split over a quadratic extension of $F(\sqrt{\tau})$.

\par

Let $\mathcal{B}'_{G'}$ be the basis of $\mathcal{V}\oplus \langle e_{2n+1}+\frac{\tau}{2} e_{-1-2n} \rangle$ obtained by adding $ e_{2n+1}+\frac{\tau}{2} e_{-1-2n} $ to $\mathcal{B}_{G'}$ as the $n+2^{th}$ vector. In coordinates, the embedding $\eta : {\rm SO}_{2n+2,\tau} \rightarrow {\rm SO}_{4n+3}$ is given by
\[
\eta : \begin{pmatrix} A_1& A_2 & A_3 & A_4 \\ A_5 & A_6 & A_7 & A_8 \\ A_9 & A_{10} & A_{11} & A_{12} \\ A_{13} & A_{14} & A_{15} & A_{16}  \end{pmatrix} 
\mapsto {\rm diag}(1_{n},\mathcal{M} \begin{pmatrix} A_1&A_2& & A_3 & A_4 \\A_5&A_6&\vspace{-.1cm}  & A_7 & A_8   \\ \vspace{-.1cm}  && 1 & &   \\A_9&A_{10}& & A_{11} & A_{12}  \\ A_{13}&A_{14}& & A_{15} & A_{16} \end{pmatrix} \mathcal{M}^{-1},1_{n}),
\]
where $A_1$ and $A_{16}$ are $n \times n$ matrices, $A_6$ and $A_{11}$ are $1 \times 1$ matrices, and 
\[
 \mathcal{M}={\rm diag}(1_{n+1}, \begin{pmatrix} \;\;\;\frac{1}{\beta}&1&-\frac{1}{2\beta} \\ \;\;\;1&0&\;\;\; \frac{1}{2} \\  -\frac{\beta}{2}&\frac{\beta^2}{2}&\;\;\; \frac{\beta}{4} \end{pmatrix}, 1_{n+1}), \quad \text{split ${\rm SO}_{2n+2,\tau}$}
\]
and
\[
 \mathcal{M}={\rm diag}(1_{n+1}, \begin{pmatrix*}[r] 0&1&1 \\ 1&0&0 \\ 0&\frac{\tau}{2}&-\frac{\tau}{2} \\ \end{pmatrix*}, 1_{n+1}), \quad \text{quasi-split ${\rm SO}_{2n+2,\tau}$}
\]
are the change of coordinate matrices from the standard basis to $\mathcal{B}'_{G'}$. For example \cite[(9.10)]{GRS11}, if ${\rm SO}_{2n+2,\tau}$ is split,
\[
 g= \begin{pmatrix} 1_n & x_1& x_2 & y\\ &1&0& -\check{x}_2  \\ && 1 & -\check{x}_1 \\ &&& 1_n \end{pmatrix}, \quad  \eta(g)={\rm diag}(1_{n},\begin{pmatrix} 1_n &( x_4 , x_3 , -\frac{2}{\tau}x_4 ) & y \\ & 1_3 & \begin{pmatrix} \frac{2}{\tau} \check{x}_4 \\ -\check{x}_3 \\ -\check{x}_4  \end{pmatrix} \\ && 1_n \end{pmatrix},1_{n}) \in {\rm SO}_{4n+3}
\]
with $x_3$ and $x_4$ certain linear combinations of $x_1$ and $x_2$, and if ${\rm SO}_{2n+2,\tau}$ is quasi-split, a standard unipotent subgroup $g$ in ${\rm SO}_{2n+2,\tau}$
and the image $\eta(g)$ in ${\rm SO}_{4n+3}$ with respect to the standard basis $F^{4n+3}$ are 
\[
 g=\begin{pmatrix} 1_n& x_1 & 2x_2 & y \\  & 1 & 0 & -2\check{x}_1 \\ &  & 1 & \frac{2}{\tau}\check{x}_2 \\  & && 1_n \end{pmatrix},  \quad
  \eta(g)={\rm diag}(1_{n},\begin{pmatrix} 1_n &( x_2 , x_1 , -\frac{2}{\tau}x_2 ) & y \\ & 1_3 & \begin{pmatrix} \frac{2}{\tau}\check{x}_2 \\ -\check{x}_1 \\ -\check{x}_2  \end{pmatrix} \\ && 1_n \end{pmatrix},1_{n}) \in {\rm SO}_{4n+3}.
\]
Let $P'=M'\ltimes U'$ be the Seigel parabolic subgroup of $G'$ with the Levi part $M'=\varrho'(\mathbb{M}')$ and
the Seigel unipotent subgroup $U'$, where $\varrho': m \mapsto {\rm diag}(1_n,m,1_3,m^{\ast},1_n)$. Let us set 
\[
\mathfrak{a}_n=\{ x \in {\rm Mat}_{n,n} \,|\, \check{x}=-x \}.
\]
Then $U'=U'_1 \ltimes U'_0$ with
\[
 U'_0=\{ {\rm diag}(1_n, \begin{pmatrix} 1_n &&u \\ &1_3& \\ && 1_n \end{pmatrix},1_n)\,|\, u \in \mathfrak{a}_n \}
\] 
and
\[
 U'_1=\{ {\rm diag}(1_n, \begin{pmatrix} 1_n &( v_2 , v_1 , -\frac{2}{\tau}v_2 ) &  \\ & 1_3 & \begin{pmatrix} \frac{2}{\tau}\check{v}_2 \\ -\check{v}_1 \\ -\check{v}_2  \end{pmatrix} \\ && 1_n \end{pmatrix} ,1_n) \,|\, v_1, v_2 \in F^n \}.
\]
Let $N'$ (resp. $N$) be the standard maximal unipotent subgroup of $G'$ (resp. $G$). Then $N'=N'_{M'} \ltimes U'$ with $N'_{M'}=\varrho'(N'_{\mathbb{M}'})$.
Let $V$ be the unipotent radical in $G$ of the standard parabolic subgroup with Levi ${\rm GL}^n_1 \times {\rm SO}_{2n+3}$:
\[
V=\{ \begin{pmatrix}  v & u & \;\;\;y \\ &1_{2n+3} & -\check{u} \\ &&\;\;\; v^{\ast} \end{pmatrix} \in {\rm SO}_{4n+3} \,|\,  v \in Z_n, u \in {\rm Mat}_{n,2n+3}  \}.
\]
Let $P^{\diamond}=M^{\diamond} \ltimes U^{\diamond}$ (resp. $P=M \ltimes U$) be the standard Seigel parabolic subgroup of ${\rm SO}_{2n+3}$ (resp. $G$) with the Levi part $M^{\diamond}$ (resp. $M$) isomorphic to ${\rm GL}_{n+1}$ (rep. ${\rm GL}_{2n+1}$). We let $\mathbb{M}={\rm GL}_{2n+1}$ and use the isomorphism $\varrho(h)={\rm diag}(h,1,h^{\ast})$ to identify $\mathbb{M}$ with $M$. Let
\[
 w'_{U^{\diamond}}={\rm diag}(1_n,\begin{pmatrix} &\vspace{-.1cm}& 1_{2n+1} \\ &1&\vspace{-.1cm} \\ 1_{2n+1} && \end{pmatrix}, 1_n)
\]
represent the longest $M^{\diamond}$-reduced Weyl element of  ${\rm SO}_{2n+3}$. Let $w_U=\begin{pmatrix} &\vspace{-.1cm}& 1_{2n+1} \\ \vspace{-.1cm}&1&\\  1_{2n+1}&&\end{pmatrix}$
represent the longest $M$-reduced Weyl element of ${\rm SO}_{4n+3}$. We write  \cite[(4.14),(4.15)]{GRS11}
\[
 \gamma=w_U (w'_{U^{\diamond}})^{-1} =\begin{pmatrix} &1_{n+1}&&&\vspace{-.15cm} \\ &&\vspace{-.15cm}&& 1_n\\ \vspace{-.15cm}&&1&& \\ 1_n&&&\vspace{-.15cm}&\\ &&&1_{n+1}& \end{pmatrix}.
\]
Let us denote
\[
 V_{\gamma}=V \cap \gamma^{-1} N \gamma=\{ \begin{pmatrix} v&\vspace{-.1cm}&&u&\\ &1_{n+1}&\vspace{-.1cm}&&-\check{u} \\&&1&\vspace{-.1cm}&\\ &&&1_{n+1}&\vspace{-.15cm} \\ &&&& v^{\ast} \end{pmatrix} \,|\, v \in Z_n, u \in {\rm Mat}_{n,n+1} \}.
\]
Let $N_M$ (resp. $N_{\mathbb{M}}$) be the standard maximal unipotent subgroup of $M$ (resp. $\mathbb{M}$).
Then $N=N_M \ltimes U$ with $N_M=\varrho(N_{\mathbb{M}})$. Let $K=K_{{\rm GL}_{4n+3}(F)} \cap G$
and $K'=K_{{\rm GL}_{2n+2}(F)} \cap G'$.
Then $K$ (resp. $K'$) is a maximal compact subgroup of G (resp. $G'$). Let $E:=F(\sqrt{\tau})$ be the discriminant field over $F$.
Let $\chi_{\tau}$ be the quadratic character associated to $E / F$ by the class field theory.

\subsection{Additive Characters}
We fix a non-trivial additive character $\psi$ of $F$ and a non-degenerate character $\psi_{N_{\mathbb{M}}}$ of $N_{\mathbb{M}}$.
Let $\psi_{\circ}$ be non-trivial characters of $F$ given by 
\[
\psi_{\circ}(x)=\psi_{N_{\mathbb{M}}}(1_{2n+1}+x\epsilon_{n+1,n+2}) 
\]
for $x \in F$. As explained in \cite[Remark 6.4]{LM17} (cf. \cite[\S 4.1, \S 6.1]{LM16}), the statements
in the sequel will not depend on the choice of $\psi_{N_{\mathbb{M}}}$. For convenience, we set 
\begin{equation}
\label{usual-even-orthogonal-additive}
 \psi_{N_{\mathbb{M}}}(u)=\psi(u_{1,2}+\dotsm+u_{2n,2n+1}).
\end{equation}
Thus $\psi_{\circ}=\psi$. The character $\psi$ will determine characters of several other unipotent groups. 
Let $\psi_{N_M}$ be the non-degenerate character of $N_M$ given by $\psi_{N_M}(\varrho(u))=\psi_{N_{\mathbb{M}}}(u)$.
We denote by $\psi_N$ the degenerate character on $N$ given by $\psi_N(uv)=\psi_{N_M}(u)$ for any $u \in N_M$ and $v \in U$.
Let $\psi_{N'_{M'}}$ be the non-degenerate character of $N'_{M'}$ such that $\psi_{N'_{M'}}(\varrho(u))=\psi_{N'_{\mathbb{M'}}}(u)$.
With the choice of $\psi_{N_{\mathbb{M}}}$, we have
\[
 \psi_{N'_{\mathbb{M}'}}(u')=\psi(u'_{1,2}+\dotsm + u'_{n-1,n}).
\]
  Let $\psi_{U'}$ be the character on $U'$ given by $\psi_{U'}(u)=\psi(u_{2n,2n+2})$.
  Then the non-degenerate character $\psi_{N'}$ is given by 
  \[
    \psi_{N'}(uv)=\psi_{N'_{M'}}(u)\psi_{U'}(v)=\psi(u_{n+1,n+2}+\dotsm+u_{2n-1,2n}+v_{2n,2n+2})
  \]
  with $u \in N'_{M'}$ and $v \in U'$. Our choice is consistent with  \cite[(9.11)]{GRS11} as we have specified $\lambda=1$ and $\mu=0$. 
  An element in $V$ can be written as $vu$ where $u$ fixes $e_1,\dotsm, e_n$, $v$ fixes $e_{n+1},e_{n+2},\dotsm,e_{-n-1}$ and we set \cite[(3.48)]{GRS11}
  \[
   \psi_V(vu)=\psi^{-1}(v_{1,2}+\dotsm+v_{n-1,n})\psi^{-1}(v_{n,n+1}+\frac{\tau}{2}v_{n,n+3}).
  \]

\subsection{Whittaker Models and the Intertwining Operator} 
Let $\pi$ be an irreducible generic representation of $\mathbb{M}$ with its Whittaker model $\mathbb{W}^{\psi_{N_{\mathbb{M}}}}(\pi)$
 with respect to the character $\psi_{N_{\mathbb{M}}}$. Similarly we use the notation $\mathbb{W}^{\psi^{-1}_{N_{\mathbb{M}}}}$,
 $\mathbb{W}^{\psi_{N_M}}$, $\mathbb{W}^{\psi^{-1}_{N_M}}$.

\par
Given $g \in G$, we define $\nu(g)$ by $\nu(u\varrho(m)k)=|\det m|$ for any $u \in U$, $m \in \mathbb{M}$, and $k \in K$ via the Iwasawa decomposition.
For any $f \in C^{\infty}(G)$ and $s \in \mathbb{C}$, define $f_s(g)=f(g)\nu(g)^s$, $g \in G$. Let ${\rm Ind}(\mathbb{W}^{\psi_{N_M}}(\pi))$ be the space of $G$-smooth left $U$-invariant function $W : G \rightarrow \mathbb{C}$
such that for all $g \in G$, the function $m \mapsto \delta^{\frac{1}{2}}_P(m)W(mg)$ on $M$ belongs to $\mathbb{W}^{\psi_{N_M}}(\pi)$.
Similarly we define ${\rm Ind}(\mathbb{W}^{\psi^{-1}_{N_M}}(\pi))$, ${\rm Ind}(\mathbb{W}^{\psi_{N_{\mathbb{M}}}}(\pi))$, ${\rm Ind}(\mathbb{W}^{\psi^{-1}_{N_{\mathbb{M}}}}(\pi))$.
For any $s \in \mathbb{C}$, we have a representation $I(s,g)$ on the space ${\rm Ind}(s,\mathbb{W}^{\psi_{N_M}}(\pi))$ given
by $(I(s,g)W)_s)(x)=W_s(xg)$, $x, g \in G$.

\par
We define the intertwining operator \cite[\S 3.2]{Kap15}; 
\[
 M(s,\pi) : {\rm Ind}(s,\mathbb{W}^{\psi_{N_M}}(\pi)) \rightarrow {\rm Ind}(-s,\mathbb{W}^{\psi_{N_M}}(\pi^{\vee}))
\]
by (the analytic continuation of)
\[
 M(s,\pi) W(g)=\nu(g)^s \int_U W_s(\varrho(\mathfrak{t})w_U u g) \,du,
\]
where $\mathfrak{t} \in {\rm GL}_{2n+1}$ is introduced in order to preserve the character $\psi_{N_M}$.

\par
In case where $F$ is $p$-adic and $\pi$ and $\psi$ are unramified, and there exist (necessarily unique) $K$-fixed elements $W^{\circ} \in {\rm Ind}(\mathbb{W}^{\psi_{N_M}}(\pi))$
and $\check{W}^{\circ} \in {\rm Ind}(\mathbb{W}^{\psi_{N_M}}(\pi^{\vee}))$ such that $W^{\circ}(e)=\check{W}^{\circ}(e)=1$,
then we have \cite[\S 2]{Sha88} (assuming ${\rm vol}(U \cap K)=1$)
\begin{equation}
\label{unramified-even-orth-intertwining}
 M(s,\pi) W^{\circ}= \frac{L(2s,\pi,{\rm Sym}^2)}{L(2s+1,\pi,{\rm Sym}^2)} \check{W}^{\circ},
\end{equation}
The following result is an analogue of \cite[Proposition 2.1]{LM16} and \cite[Proposition 4.1]{LM17}.

\begin{proposition}
\label{holomorphy-even-orthgonal}
Suppose that $\pi \in {\rm Irr}_{\rm gen} \, \mathbb{M}$ such that $\pi \cong \pi^{\vee}$. Then $M(s,\pi)$
is holomorphic at $s=\frac{1}{2}$.
\end{proposition}

\begin{proof}
We write $\pi=\sigma_1 \boxplus \dotsm \boxplus \sigma_k$ where $\sigma_1,\dotsm,\sigma_k$ are essentially square-integrable. 
By decomposing $M(s,\pi)$ into a product of co-rank one intertwining operators, it is enough to show that 
the following intertwining operators 
\begin{enumerate}[label=$(\mathrm{\roman*})$]
\item \label{even-orthognal-intertwining-differ} $\sigma_i |\cdot |^s\boxplus \sigma_j^{\vee}|\cdot|^{-s} \rightarrow \sigma_j^{\vee}|\cdot|^{-s} \boxplus \sigma_i|\cdot|^s, \quad i,j=1,2,\dotsm,k, i \neq j$,
\item  \label{even-orthognal-intertwining-same} $M(s,\sigma_i), \quad i=1,2,\dotsm,k$,
\end{enumerate}
are holomorphic at $s=\frac{1}{2}$. The holomorphy of \ref{even-orthognal-intertwining-differ} is a consequence of \cite[Lemma 4.2]{LM17} and the irreducibility of $\sigma_i \boxplus \sigma^{\vee}_j$.
If $F$ is an archimedean local field, the holomorphy of \ref{even-orthognal-intertwining-same} follows from \cite[Lemma 4.3]{LM17} and the irreducibility of $\sigma_i \boxplus \sigma_i^{\vee}$. Under the condition
of \cite[Lemma 4.3]{LM17}, we have the stronger conclusion from the very recent paper \cite[\S 2.4 Main Theorem 1]{Luo24} that $L(2s,\sigma_i,{\rm Sym}^2)^{-1}M(s,\sigma_i)$ is holomorphic for the $p$-adic case.
\end{proof}

\subsection{Period and Distinction}

\subsubsection{Global period}
Let $\vartheta$ be the space of the global exceptional representation $\Theta$ of $\widetilde{\rm GL}_n(\mathbb{A})$,
which is the Langlands quotient of a certain induced representation of $\widetilde{\rm GL}_n(\mathbb{A})$ inherited from the
metaplectic preimage $\tilde{B}_n(\mathbb{A})$ of the Borel subgroup $B_n(\mathbb{A})$.
The best reference for the construction of this representation is \cite[\S 2.1]{Tak14}.
Let $F$ be a number field. We say that $\pi \in {\rm Cusp}\, \mathbb{M}$ is of {\it orthogonal symplectic type} if
\[
 \int_{{\rm GL}_n(F) \backslash {\rm GL}_n(\mathbb{A})^1} \varphi(g)\Theta(g) \overline{\Theta'}(g) \, dg \neq 0
\] 
for some $\varphi$ in the space of $\pi$, and $\Theta, \Theta'$ in the space of $\vartheta$. The following characterization is due to Bump--Ginzburg \cite{BG92}, Takeda \cite{Tak14},
and Kaplan--Yamana \cite[Theorem A]{KY17}.

\begin{proposition}
Let $\pi \in {\rm Cusp} \, \mathbb{M}$. Then $\pi$ is of orthogonal symplectic type if and only if $L^S(s,\pi,{\rm Sym}^2)$
has a pole at $s=1$.
\end{proposition}

We let ${\rm OScusp}_k\,\mathbb{M}$ be the set of automorphic representations $\pi$
of the form $\pi_1 \boxplus \dotsm \boxplus \pi_k$, where $\pi_i \in {\rm Cusp}\, {\rm GL}_{n_i}(F)$, $i=1,\dotsm,k$
are pairwise inequivalent (that is, $\pi_i \not \cong \pi_j$ if $i \neq j$) and of orthogonal symplectic type with $n_1+\dotsm+n_k=2n+1$ and $n_i > 1$ for all $i$.
In particular $\pi^{\vee}_i \cong \pi_i$ and the central character $\omega_{\pi}$ of $\pi$ is trivial on $(\mathbb{A}^{\times})^2$. We let ${\rm OScusp} \,\mathbb{M}=\cup_k {\rm OScusp}_k\,\mathbb{M}$.

\subsubsection{Local distinction}
Let $F$ be a local field.
If $F$ is $p$-adic, we say that $\pi \in {\rm Irr}\, {\rm GL}_n(F)$ is of {\it orthogonal symplectic type} if there is an exceptional representation $\theta$ of $\widetilde{\rm GL}_n(F)$
(defined as in Bump--Ginzburg \cite{BG92} and Takeda \cite{Tak14}) such that $\pi$ admits a non-trivial continuous functionals $L$ on the space $\pi \times \theta \times \theta^{\vee}$.
In particular, the central character $\omega_{\pi}$ of $\pi$ is trivial on $(F^{\times})^2$ (cf. \cite[Corollary 4.3]{Kap16}). 

\par
Let $P_{\circ} \subseteq \mathbb{M}$ be the standard parabolic subgroup in $\mathbb{M}$ whose Levi subgroup is
isomorphic to $\prod_{i=1}^k {\rm GL}_{n_i}$ with $n_1+\dotsm+n_k=2n+1$ and $n_i > 1$ for all $i$. 
If $F=\mathbb{R}$ or $\mathbb{C}$, we let $\pi={\rm Ind}^{\mathbb{M}}_{P_{\circ}} \otimes_{i=1}^k \pi_i$
be a self-dual representation of $\mathbb{M}$ (i.e., $\pi \cong \pi^{\vee}$), where
$\pi_i$, $i=1,\dotsm,k$ are pairwise inequivalent (i.e., $\pi_i \not \cong \pi_j$ if $i \neq j$)
and each $\pi_i$ owns a non-trivial continuous functionals $L$ on the space $\pi \times \theta \times \theta^{\vee}$ with $n_1+\dotsm+n_k=2n+1$ and $n_i > 1$ for all $i$.
Such a representation $\pi={\rm Ind}^{\mathbb{M}}_{P_{\circ}} \otimes_{i=1}^k \pi_i$ is called a representation of {\it orthogonal symplectic type}.

\par
For $F$ a local field, we write ${\rm Irr}_{\rm os} \, {\rm GL}_n(F)$ for the set of irreducible representations
of orthogonal symplectic type. Clearly, if $\pi$ is of orthogonal symplectic type in the global setting, then all its local components $\pi_v$
are irreducible and of orthogonal symplectic type as well.

\begin{lemma} Assume that $F$ is a $p$-adic field. 
  \begin{enumerate}[label=$(\mathrm{\arabic*})$]
\item (\cite[Corollary 4.19]{Kap17}) Suppose that $\pi \in {\rm Irr}_{\rm os} \, {\rm GL}_n(F)$. Then $\pi^{\vee} \cong \pi$.
\item (\cite[Theorem A-(1)]{Kap17}) Suppose that $\pi \in {\rm Irr} \, {\rm GL}_n(F)$ is square integrable. Then $\pi \in {\rm Irr}_{\rm os}\, {\rm GL}_n(F)$ if and only if
$L(s,\pi,{\rm Sym}^2)$ has a pole at $s=0$.
\item  (\cite[Theorem 3]{Kap16})  Suppose that $\pi_i \in {\rm Irr}_{\rm os} \, {\rm GL}_{n_i}(F)$, $i=1,2$ and the parabolic induction $\pi_1 \boxplus \pi_2$ is irreducible.
Then $\pi_1 \boxplus \pi_2 \in {\rm Irr}_{\rm os} \, {\rm GL}_{n_1+n_2}(F)$.
\end{enumerate}
\end{lemma}

For the sake of completeness, we recall the following classification theorem, due to Kaplan, of the set ${\rm Irr}_{\rm gen,os}\, {\rm GL}_n(F)$ of generic representations
of orthogonal symplectic types.

\begin{proposition} (\cite[Theorem A-(2)]{Kap17}) Assume that $F$ is a $p$-adic field. Then the set ${\rm Irr}_{\rm gen,os} \, {\rm GL}_n(F)$ consists of the irreducible representations of the form
\[
 \pi=(\sigma_1 \boxplus \sigma^{\vee}_1) \boxplus \dotsm \boxplus (\sigma_k \boxplus \sigma^{\vee}_k) \boxplus \tau_1 \boxplus \dotsm \boxplus \tau_{\ell}
\]
where $\sigma_1, \dotsm, \sigma_k$ are essentially square-integrable and $\tau_1,\dotsm,\tau_{\ell}$ are square-integrable of orthogonal symplectic type (i.e., $L(0,\tau_i,{\rm Sym}^2)=\infty$ for all $i=1,\dotsm,\ell$).
\end{proposition}

\subsection{Descent map}

\subsubsection{Global descent}
We now go to the global setting. For any automorphic form $\varphi$ on $G(\mathbb{A})$, let ${\rm GG}(\varphi)$ be the {\it Gelfand--Graev coefficient} (a function on $G'(F) \backslash G'(\mathbb{A})$)
 \[
  {\rm GG}(\varphi)(g)=\int_{V(F) \backslash V(\mathbb{A})} \varphi(vg) \psi_V^{-1}(v) \, dv, \quad g \in G'(\mathbb{A}).
 \]

 \par
The automorphic representation $\pi$ of $\mathbb{M}(\mathbb{A})$ in ${\rm OScusp} \,\mathbb{M}$ is realized on the space of Eisenstein series induced from
$\pi_1 \otimes \dotsm \otimes \pi_k$. We view $\pi$ as a representation of $M(\mathbb{A})$ via $\varrho$. 
Let $\mathcal{A}(\pi)$ be the space of functions $\varphi : M(F)U(\mathbb{A})\backslash G(\mathbb{A}) \rightarrow \mathbb{C}$ such that
$m \mapsto \delta^{\frac{1}{2}}_P(m)\varphi(mg)$, $m \in M(\mathbb{A})$ belongs to the space of $\pi$ for all $g \in G(\mathbb{A})$.
Let $\varphi_s(g)=\nu(g)^s\varphi(g)$. We define Eisenstein series
\[
 \mathcal{E}(s,\varphi)=\sum_{\gamma \in P(F) \backslash G(F)} \varphi_s(\gamma g), \quad \varphi \in \mathcal{A}(\pi).
\]
The space of $\pi$ is invariant under the conjugation by $w_U$. Then we define intertwining operator $M(s,\pi) : \mathcal{A}(\pi) \rightarrow \mathcal{A}(\pi)$
given by
\[
 M(s,\pi) \varphi(g)=\nu(g)^s \int_{U(\mathbb{A})} \varphi_s(w_Uug) \, du.
\]
Both $\mathcal{E}(s,\varphi)$ and $M(s,\pi)$ converge absolutely for ${\rm Re}(s) \gg 0$ and extend meromorphically to $\mathbb{C}$.
It follows from \cite[Theorem 2.1]{GRS11} that both $\mathcal{E}(s,\varphi)$ and $M(s,\pi)$ have a pole of order $k$ at $s=\frac{1}{2}$.
Let
\[
 \mathcal{E}_{-k} \varphi=\lim_{s \rightarrow \frac{1}{2}} \left( s-\frac{1}{2} \right)^k \mathcal{E}(s,\varphi)
\]
and
\[
 M_{-k}=\lim_{s \rightarrow \frac{1}{2}} \left( s-\frac{1}{2}\right)^kM(s,\pi)
\]
be the leading coefficient in the Laurent series expansion at $s=\frac{1}{2}$. The constant term
\[
 \mathcal{E}^U_{-k} \varphi(g)=\int_{U(F) \backslash U(\mathbb{A})} \mathcal{E}_{-k}\varphi(ug) \, du
\]
of $\mathcal{E}_{-k} \varphi$ along $U$ is given by
\[
  \mathcal{E}^U_{-k} \varphi(g)=(M_{-k}(\varphi)(g))_{-\frac{1}{2}}.
\]
The {\it global descent} of $\pi$ (with respect to $\psi_{N_{\mathbb{M}}}$) is the space $\mathcal{D}_{\psi}(\pi)$
generated by ${\rm GG}(\mathcal{E}_{-k}\varphi)$ with $\varphi \in \mathcal{A}(\pi)$.
It is known from \cite[Theorem 3.1]{GRS11} that $\mathcal{D}_{\psi}(\pi)$ is multiplicity free.
By \cite[Theorem 9.6]{GRS11}, $\mathcal{D}_{\psi}(\pi)$ is a non-trivial cuspidal (and globally generic) automorphic representation of $G'$.
We expect it to be irreducible. This result will be the analogue of \cite[Theorem 2.3]{GJS12}  in metaplectic group case and \cite[Theorem 1.1]{Mor18} in even unitary group case. It is likely that the same methods of \cite[Theorem 2.3]{GJS12} and \cite[Theorem 1.1]{Mor18}
no doubt work in the case at hand. Unfortunately, to the best of our knowledge this has not been carried out in the literature. 
Since this is beyond the scope of the current paper we simply make it a working hypothesis. 

\begin{working hypothesis}
\label{even-orthogonal-working-hypothesis}
When $\pi \in {\rm OScusp} \,\mathbb{M}$, $\mathcal{D}_{\psi}(\pi)$ is irreducible. 
\end{working hypothesis}

The image of the weak lift is then a consequence of \cite[Theorem 11.2]{GRS11}. In particular, \cite[Theorem 9.6 (2)]{GRS11} describes the image for unramified places.

\begin{proposition} With the above Working Hypothesis \ref{even-orthogonal-working-hypothesis}, 
$\pi \mapsto \sigma=\mathcal{D}_{\psi}(\pi)$ defines a bijection between ${\rm OScusp}\, \mathbb{M}$ and ${\rm Cusp}_{\psi_{N'}}\, G'$,
the set of $\psi_{N'}$-generic cuspidal automorphic representation of $G'$. Moreover $\sigma^{\vee}$ weekly lifts to $\pi \boxplus \omega_{\pi}\chi_{\tau}$.
\end{proposition}

\subsubsection{Local descent} 
We go back to the local setup. Suppose now that $F$ is local fields.
Let $\pi \in {\rm Irr}_{\rm gen} \, M$. 
 For any $W \in {\rm Ind}(\mathbb{W}^{\psi_{N_M}}(\pi))$ and $s \in \mathbb{C}$,
we define a function on $G'$:
\[
  A^{\psi}(s,g,W)=\int_{V_{\gamma} \backslash V} W_s((-1)^n\gamma v g) \psi_V^{-1}(v) \, dv, \quad g \in G'.
\] 
The basic properties of $ A^{\psi}$ are summarized in the following proposition. Its proof is identical to \cite[Lemmata 4.5 and 4.9]{LM16} and \cite[Lemmata 5.1 and 8.1]{LM17}, and will be omitted.

\begin{proposition} Suppose that $\pi \in {\rm Irr}_{\rm gen} \, M$. Let $W \in {\rm Ind}(\mathbb{W}^{\psi_{N_M}}(\pi))$ and $s \in \mathbb{C}$. Then
\begin{enumerate}[label=$(\mathrm{\arabic*})$]
\item The integral $A^{\psi}(s,g,W)$ is well-defined and absolutely convergent uniformly for $s \in \mathbb{C}$ and $g \in G'$ in compact sets.
Therefore $A^{\psi}(s,g,W)$ is entire in $s$ and smooth in $g$. In the non-archimedean case the integrand 
is compactly supported on $V_{\gamma} \backslash V$.
\item For any $W \in {\rm Ind}(\mathbb{W}^{\psi_{N_M}}(\pi))$ and $s \in \mathbb{C}$, the function $g \mapsto A^{\psi}(s,g,W)$
is smooth and $(N',\psi_{N'})$-equivariant.
\item For any $g, x \in G'$, we have
\begin{equation}
\label{even-orthogonal-G'-equivalent}
 A^{\psi}(s,g,I(s,x)W)=A^{\psi}(s,gx,W).
\end{equation}
\item Suppose that $\pi$ is unramified, $\psi$ is unramified, and $W^{\circ} \in {\rm Ind}(\mathbb{W}^{\psi_{N_M}}(\pi))$
is the standard unramified vector. Then $A^{\psi}(s,e,W^{\circ}) \equiv 1$ (assuming ${\rm vol}(V \cap K)={\rm vol}(V_{\gamma}\cap K)=1$).
\end{enumerate}
\end{proposition}

Let $\pi \in {\rm Irr}_{\rm gen} \, \mathbb{M}$, regarded also as a representation of $M$ via $\varrho$. By the same argument as in \cite[Theorem in \S 1.3]{GRS99} in conjunction with \cite[Remark 4.13]{LM17},
for any non-zero subrepresentation $\pi'$ of ${\rm Ind}(\mathbb{W}^{\psi_{N_M}}(\pi))$ there exists $W' \in \pi'$ such that $A^{\psi}(s,\cdot,W') \not \equiv 0$.
Upon twisting $\pi$, we deduce that for any given $s_0 \in \mathbb{C}$, there exists $W' \in \pi'$ such that $A^{\psi}(s_0,\cdot,W') \not \equiv 0$.

\par 
Assume now that $\pi \in {\rm Irr}_{\rm gen,os}\, \mathbb{M}$. By Proposition \ref{holomorphy-even-orthgonal}, $M(s,\pi)$ is holomorphic at $s=\frac{1}{2}$. The {\it local descent} of $\pi$ is the space $\mathcal{D}_{\psi}(\pi)$ of Whittaker functions on $G'$
generated by $A^{\psi}\left(-\frac{1}{2},\cdot, M\left( \frac{1}{2},\pi \right)W \right)$ with $W \in {\rm Ind}(\mathbb{W}^{\psi_{N_M}}(\pi))$. By the above observation,
$\mathcal{D}_{\psi}(\pi) \neq 0$.

\par
Let $\pi'$ be the image of ${\rm Ind}(\mathbb{W}^{\psi_{N_M}}(\pi))$ under $M\left( \frac{1}{2},\pi \right)$. In virtue of \eqref{even-orthogonal-G'-equivalent},
the space $\mathcal{D}_{\psi}(\pi)$ is canonically a quotient of the $G'$-module $J_V(\pi')$ of the $V$-coinvariant of $\pi'$.
In this perspective, we view $J_V(\pi')$ as the {\it abstract descent}, whereas $\mathcal{D}_{\psi}(\pi)$ as the {\it explicit descent}.

\subsubsection{Whittaker function of descent} 
We let $\hat{\psi}_N$ be a degenerate character on $N$ that is trivial on $U$ and for $u \in N_{\mathbb{M}}$
\[
\hat{\psi}_N(\varrho(u)) := \psi(u_{1,2}+\dotsm+u_{n,n+1}-u_{n+1,n+2}-\dotsm-u_{2n,2n+1}).
\]
to be consistent with the additive character in \cite[(8.14)]{GRS11}. Thus $\psi_{\circ}(x)=\psi(x)$. By \cite[Remark 6.4]{LM17} (cf. \cite[\S 4.1, \S 6.1]{LM16}), we will
verify Theorem \ref{Whittaker-main-even-orthogonal} for a specific choice of $\psi_{N_{\mathbb{M}}}$ given by
\[
 \psi_{N_{\mathbb{M}}}(u)=\psi(u_{1,2}+ \dotsm +u_{n,n+1}-\frac{\tau}{2}u_{n+1,n+2} -u_{n+2,n+3}-\dotsm-u_{2n,2n+1}).
\]
Henceforth $\psi_{\circ}(x)=\psi(x)$. Let
\[
  \delta=\left(\begin{array}{cc|c|cc} \delta'_{n,2n} & 0 & 0 & 0 & 0_{n \times 2n} \\ 0 & 1 & 0 & 0 & 0\\ 0_{n \times 2n} & 0 & 0 &0  &  \delta'_{n,2n} \\ \hline
  0 & 0 & 1 & 0 & 0 \\ \hline \delta''_{n,2n} & 0 & 0 & 0 & 0_{n \times 2n} \\ 0 & 0 & 0 & 1 & 0 \\ 0_{n \times 2n} & 0 & 0 & 0 & \delta''_{n\times 2n}  \end{array} \right).
\]
We let $\kappa=\varrho(\kappa')$, where $\kappa'$ is the Weyl element of $\mathbb{M}$ such that $\kappa'_{2i,i}$, $i=1,\dotsm,n$, $\kappa'_{2i-1,n+i}=1$, $i=1,\dotsm,n+1$,
and all other entires are zero. Let 
$\epsilon={\rm diag}(A,\dotsm,A, 1_3,A^{\ast},\dotsm A^{\ast})$, where  $A=\begin{pmatrix*}[r] \frac{1}{2}& 1 \\  -\frac{\tau}{4} & \frac{\tau}{2}  \end{pmatrix*}$
(cf. \cite[(9.57)]{GRS11}). Here $A$ is repeated $n$ times.
Writing
\[
 A=\begin{pmatrix} 1 & \frac{2}{\tau} \\ & 1 \end{pmatrix} \begin{pmatrix} 1 & \\ & \frac{\tau}{2} \end{pmatrix} \begin{pmatrix*}[r] 1 & \\ -\frac{1}{2} & 1 \end{pmatrix*},
\]
we obtain $\epsilon=\epsilon_U\epsilon_T\epsilon_{\overline{U}}$, where
\[
\begin{split}
  &\epsilon_U={\rm diag}( \begin{pmatrix} 1 & \frac{2}{\tau} \\ & 1 \end{pmatrix} , \dotsm, \begin{pmatrix} 1 & \frac{2}{\tau} \\ & 1 \end{pmatrix}, 1_3,   \begin{pmatrix*}[r] 1 & -\frac{2}{\tau} \\ & 1 \end{pmatrix*}, \dotsm,  \begin{pmatrix*}[r] 1 & -\frac{2}{\tau} \\ & 1 \end{pmatrix*} ), \\
  &\epsilon_{\overline{U}}={\rm diag}(  \begin{pmatrix*}[r] 1 & \\ -\frac{1}{2} & 1 \end{pmatrix*}, \dotsm,  \begin{pmatrix*}[r] 1 & \\ -\frac{1}{2} & 1 \end{pmatrix*}, 1_3,  \begin{pmatrix} 1 & \\ \frac{1}{2} & 1 \end{pmatrix}, \dotsm, \begin{pmatrix} 1 & \\ \frac{1}{2} & 1 \end{pmatrix} ), \\
  \end{split}
\]
and
\[
 \epsilon_T=\left( \begin{pmatrix} 1 & \\ & \frac{\tau}{2} \end{pmatrix}, \dotsm \begin{pmatrix} 1 & \\ & \frac{\tau}{2} \end{pmatrix}, 1_3, 
 \begin{pmatrix} \frac{2}{\tau}  & \\ & 1 \end{pmatrix}, \dotsm  \begin{pmatrix} \frac{2}{\tau}  & \\ & 1 \end{pmatrix}    \right).
\]
Let $X$ be the subspace of $\mathfrak{a}_{2n+1}$  consisting of the strictly upper triangular matrices. 
We regard $\mathbb{A}^n$ as subspace of $\mathbb{A}^{2n+1}$ by the embedding 
\[
^t(a_1,\dotsm,a_n) \in \mathbb{A}^n \mapsto a:={^t(}a_1,\dotsm,a_n,0,\dotsm,0), \quad (\text{\rm namely $a_i=0$ for $i > n$}).
\]
We let
\[
  \overline{\ell}(a,x)=\left(\begin{array}{c|c|c} 1_{2n+1} && \\ \hline -\check{a}& 1& \\ \hline x  & a & 1_{2n+1} \end{array}\right)
\]
for $x \in \mathfrak{a}_{2n+1}$ and $a \in \mathbb{A}^n \subseteq \mathbb{A}^{2n+1}$. Let $Y$  be the subgroup of $N_M$ consisting of 
the matrices of the form $\varrho \left( \begin{pmatrix} 1_n & y \\ & 1_{n+1} \end{pmatrix} \right)$
where $y$ is lower triangular, (namely $y_{i,j}=0$ if $j > i$).

\begin{theorem} \label{Whittaker-main-even-orthogonal}
$(${\rm Reformulation of \cite[Theorem 9.6, part (1)]{GRS11}}$)$
 Let $\pi \in {\rm OScusp} \, \mathbb{M}$ and $\varphi \in \mathcal{A}(\pi)$. Then for any $g \in G'$ we have
\[
 \mathcal{W}^{\psi_{N'}}(g,{\rm GG}(\mathcal{E}_{-k}\varphi))=\int_{V_{\gamma}(\mathbb{A}) \backslash V(\mathbb{A})} \mathcal{W}^{\psi_N}((-1)^n\gamma vg,\mathcal{E}_{-k}\varphi) \psi^{-1}_V(v) \, dv
\]
where the integral is absolutely convergent.
\end{theorem}

\begin{proof}
 We use Tamagawa measure in the proof. It is enough to prove the required identity for $g=e$.
The expression for $\mathcal{W}^{\psi_{N'}}(e,{\rm GG}(\mathcal{E}_{-k}\varphi))$ in \cite[Theorem 9.6, part (1)]{GRS11} is (with $\alpha=\tau$ and $\lambda=1$ in their notation):
\[
 \int_{Y(\mathbb{A})} \left( \int_{\mathbb{A}^n}  \left( \int_{X(\mathbb{A})} \mathcal{W}^{\hat{\psi}_{N}}( (-1)^n\overline{\ell}(a,x) \delta \epsilon \kappa y,\mathcal{E}_{-k} \varphi) \, dx \right) da \right) dy. 
\]
We note from \cite[\S 8.2, (8.15)]{GRS11} that $\delta\epsilon_U\delta^{-1}=\begin{pmatrix} 1_{2n+1} &\vspace{-.1cm}& -\epsilon_{\delta} \\ &1&\vspace{-.1cm} \\ && 1_{2n+1} \end{pmatrix}$, where
$\epsilon_{\delta}={\rm diag}(-\frac{2}{\tau}\cdot1_n,0, \frac{2}{\tau}\cdot1_n)$.
For any $x \in X$ and $a \in \mathbb{A}^n \subseteq \mathbb{A}^{2n+1}$, we proceed to write
\begin{multline*}
 \left(\begin{array}{c|c|c} 1_{2n+1} && \\ \hline -\check{a}& 1& \\ \hline x  & a & 1_{2n+1} \end{array}\right) \left(\begin{array}{c|c|c} 1_{2n+1} && -\epsilon_{\delta} \\ \hline & 1& \\ \hline   &  & 1_{2n+1} \end{array}\right) \\
 = \left(\begin{array}{c|c|c} (1_{2n+1}-x\epsilon_{\delta})^{\ast} &\epsilon_{\delta} (1_{2n+1}-x\epsilon_{\delta})^{-1}a & -\epsilon_{\delta} \\ \hline  & 1& \check{a}\epsilon_{\delta} \\ \hline   &  & 1_{2n+1}-x\epsilon_{\delta} \end{array}\right) \left(\begin{array}{c|c|c} 1_{2n+1} && \\ \hline -\check{a}'& 1& \\ \hline   x' & a' & 1_{2n+1} \end{array}\right),
\end{multline*}
where $x'=(1_{2n+1}-x\epsilon_{\delta})^{-1}x=x+x\epsilon_{\delta}x+\dotsm+(x\epsilon_{\delta})^{2n}x$ and $a'=(1_{2n+1}-x\epsilon_{\delta})^{-1}a $. (Note that $(1_{2n+1}-x\epsilon_{\delta})^{\ast}=(1_{2n+1}-\epsilon_{\delta}x)^{-1}$.) After changing the variables $x \mapsto (1_{2n+1}-x\epsilon_{\delta})x$ and $a \mapsto (1_{2n+1}-x\epsilon_{\delta})a$, we get
 \[
 \int_{Y(\mathbb{A})} \left( \int_{\mathbb{A}^n}  \left( \int_{X(\mathbb{A})} \hat{\psi}_N(\varrho((1_{2n+1}-x\epsilon_{\delta})^{\ast}))\mathcal{W}^{\hat{\psi}_{N}}( (-1)^n\overline{\ell}(a,x) \delta \epsilon_T \epsilon_{\overline{U}} \kappa y, \mathcal{E}_{-k} \varphi)  \, dx \right)  da \right) dy. 
\]
Since $x_{2n+1-k,2n+2-k}=-x_{k,k+1}$, we have 
\[
\hat{\psi}_N(\varrho((1_{2n+1}-x\epsilon_{\delta})^{\ast}))
=\hat{\psi}_N(\varrho((1_{2n+1}-\epsilon_{\delta}x)^{-1}))
=\psi^{-1}\left(\frac{2}{\tau}x_{n,n+1}\right)
\] 
from which it follows that
 \[
 \int_{Y(\mathbb{A})} \left( \int_{\mathbb{A}^n}  \left( \int_{X(\mathbb{A})}  \psi^{-1}\left(\frac{2}{\tau}x_{n,n+1}\right) \mathcal{W}^{\hat{\psi}_{N}}( (-1)^n\overline{\ell}(a,x) \delta \epsilon_T \epsilon_{\overline{U}} \kappa y, \mathcal{E}_{-k} \varphi) \, dx \right)  da \right) dy. 
\]

\par
Furthermore, $\delta\epsilon_T\delta^{-1}={\rm diag}(1_{n+1},\frac{2}{\tau}1_n,1,\frac{\tau}{2}1_n,1_{n+1})$. This element conjugates $\hat{\psi}_N$ to $\psi_N$ and $\psi(\frac{2}{\tau}x_{n,n+1})$ to  $\psi(x_{n,n+1})$, so by conjugating we obtain
 \[
 \int_{Y(\mathbb{A})} \left( \int_{\mathbb{A}^n}  \left( \int_{X(\mathbb{A})}  \psi^{-1}(x_{n,n+1}) \mathcal{W}^{\psi_{N}}( (-1)^n\overline{\ell}(a,x) \delta  \epsilon_{\overline{U}} \kappa y, \mathcal{E}_{-k} \varphi) \, dx \right)  da \right) dy. 
\]

\par
With the direct computation just as in \cite[Proof of Proposition 8.2]{LM16}), we find $\kappa^{-1}\epsilon_{\overline{U}}\kappa \in Y$  and $\delta\kappa=\gamma$.
Upon making change of variables $y \mapsto (\kappa^{-1}\epsilon_{\overline{U}}\kappa)^{-1}y$, we obtain that the above equals
 \[
 \int_{Y(\mathbb{A})} \left( \int_{\mathbb{A}^n}  \left( \int_{X(\mathbb{A})}  \psi^{-1}(x_{n,n+1}) \mathcal{W}^{\psi_{N}}( (-1)^n\overline{\ell}(a,x) \delta  \gamma y, \mathcal{E}_{-k} \varphi) \, dx \right)  da \right) dy. 
\]
Note that the group $\gamma^{-1}\overline{\ell}(\mathbb{A}^n,X) \gamma$ is equal to
\[
 \{ \begin{pmatrix}1_{n} &(x_1,x'_1)&x''_1&\vspace{-.07cm}&x_2 \\  &1_{n+1}&&\vspace{-.07cm}&\\ &&1&&\check{x}''_1\\ &&&1_{n+1}&\begin{pmatrix} \check{x}'_1\\ \check{x}_1 \end{pmatrix} \\ &&&\vspace{-.07cm}&1_n   \end{pmatrix} \, | \, x_1 \in {\rm Mat}_{n,n}\; \text{\rm strictly upper triangular}, x_2 \in \mathfrak{a}_n, x'_1, x''_1 \in \mathbb{A}^n \},
\]
so that $(\gamma^{-1}\overline{\ell}(\mathbb{A}^n,X)  \gamma Y) \rtimes V_{\gamma}=V$. In conclusion, we arrive at
\[
\int_{V_{\gamma}(\mathbb{A}) \backslash V(\mathbb{A})} \mathcal{W}^{\psi_N}((-1)^n\gamma v,\mathcal{E}_{-k}\varphi) \psi^{-1}_V(v) \, dv
\]
provided that it converges, since $\psi_V(\gamma^{-1}\overline{\ell}(a,x)\gamma)=\psi(x_{n,n+1})$ for $x \in X(\mathbb{A})$ and $a \in \mathbb{A}^n$.
\end{proof}

\subsection{Local Rankin--Selberg--Shimura integrals}
We are now going back to our choice of $\psi_{N_{\mathbb{M}}}$ specified in \eqref{usual-even-orthogonal-additive}.
Now let $F$ be either $p$-adic or archimedean, $\pi \in {\rm Irr}_{\rm gen} \, M$ and $\sigma \in {\rm Irr}_{{\rm gen}, \psi^{-1}_{N'}}\, G'$
with Whittaker model $\mathbb{W}^{\psi^{-1}_{N'}}(\sigma)$. For any $W' \in \mathbb{W}^{\psi^{-1}_{N'}}(\sigma)$ and $W \in {\rm Ind}(\mathbb{W}^{\psi_{N_M}}(\pi))$, we define the local Rankin--Selberg--Shimura type integral
\[
  J(s,W',W):=\int_{N' \backslash G'} W'(g) A^{\psi}(s,g,W) \, dg.
\]
The analytic properties of this integral were extensively investigated by Kaplan \cite{Kap10,Kap12,Kap15}.
 
\begin{proposition} Suppose that $\pi \in {\rm Irr}_{\rm gen} \, M$ and $\sigma \in {\rm Irr}_{{\rm gen},\psi^{-1}_{N'}}\, G'$. Let $W' \in \mathbb{W}^{\psi^{-1}_{N'}}(\sigma)$
and $W \in {\rm Ind}(\mathbb{W}^{\psi_{N_M}}(\pi))$. Then
\begin{enumerate}[label=$(\mathrm{\arabic*})$]
\item $J(s,W',W)$ converges in some right-half plane (depending only on $\pi$ and $\sigma$) and admits a meromorphic continuation in $s$.
\item For any $s \in \mathbb{C}$, we can choose $W'$ and $W$ such that $J(s,W',W)\neq 0$.
\item If $\pi$, $\sigma$, and $\psi$ are unramified, $W^{\circ} \in {\rm Ind}(\mathbb{W}^{\psi_{N_M}}(\pi))$ is the standard unramified vector and
${W'}^{\circ}$ is $K'$-invariant with ${W'}^{\circ}(e)=1$, then
\[
 J(s,{W'}^{\circ},W^{\circ})=\frac{{\rm vol}(K')}{{\rm vol}(N' \cap K')} \frac{L(s+\frac{1}{2},\sigma\times\pi)}{L(2s+1,\pi,{\rm Sym}^2)}
\]
assuming the Haar measure on $V$ and $V_{\gamma}$ are normalized so that ${\rm vol}(V \cap K)$ and ${\rm vol}(V_{\gamma} \cap K)$ are all $1$.
\item $J(s,W',W)$ satisfies local functional equations:
\[
 J(-s,W',M(s,\pi)W)
 = c(s,\sigma,\pi,\tau)\frac{\gamma(s+\frac{1}{2},\sigma\times\pi,\psi)}{\gamma(2s+1,\pi,{\rm Sym}^2,\psi)} J(s,W',W)
\]
where $c(s,\sigma,\pi,\tau)=(\omega_{\sigma}(-1)|\tau|^s(-1,\tau),\gamma_{\psi}(\tau))^{2n+1} \omega_{\pi}((-1)^{n+1}\tau)$
and $\gamma(s,\pi,{\rm Sym}^2,\psi)$ is the gamma factor defined by Shahidi (pertaining to the symmetric square representation).
\end{enumerate}
\end{proposition}
It is noteworthy that
$\displaystyle
 {\rm vol}(K')=(\prod_{j=1}^n \zeta_F(2j)L(n+1,\chi_{\tau}))^{-1}.
$

\subsection{Factorization of Gelfand--Graev periods}
We define for $\varphi \in \mathcal{A}(\pi)$ an automorphic form on $G(\mathbb{A})$;
\[
  A^{\psi}(s,g,\varphi):=\int_{V_{\gamma}(\mathbb{A}) \backslash V(\mathbb{A})} \mathcal{W}^{\psi_{N_M}}((-1)^n\gamma v g,\varphi_s)\psi^{-1}_V(v) \,dv, \quad g \in G'.
\]
The global descent $\mathcal{D_{\psi}}(\pi)$ and the local descent $\mathcal{D}_{\psi}(\pi_v)$ are related in the following way:
\begin{proposition} 
\label{GG-Whittaker-EvenOrthogonal}
Suppose that $\varphi \in \mathcal{A}(\pi)$ is a factorizable vector and $\mathcal{W}^{\psi_{N_M}}(\cdot,\varphi)=\prod_v W_v$ with $W_v \in {\rm Ind}(\mathbb{W}^{\psi_{N_M}}(\pi_v))$. Then for any sufficiently large
finite set of places $S$, we have
\begin{equation}
\label{GG-Whittaker-EvenOrthogonal-decomp}
\mathcal{W}^{\psi_{N'}}(g,{\rm GG}(\mathcal{E}_{-k}\varphi))=\frac{m^S_{-k,1}(\pi)}{2^k} \prod_{v \in S} A^{\psi} \left( -\frac{1}{2},g_v, M \left( \frac{1}{2}, \pi_v \right)W_v \right),
\end{equation}
where
\[
m^S_{-k,1}(\pi)= \frac{\lim_{s \rightarrow 1}(s-1)^kL^S(s,\pi,{\rm Sym}^2)}{L^S(2,\pi,{\rm Sym}^2)}.
\]
\end{proposition}

\begin{proof}
We note that
\[
 \mathcal{W}^{\psi_N}(\mathcal{E}_{-k}\varphi)=\mathcal{W}^{\psi_{N_M}}(\mathcal{E}^U_{-k}\varphi)
 =\mathcal{W}^{\psi_{N_M}}((M_{-k}(\varphi))_{-\frac{1}{2}}).
\]
With this in hand, we rewrite Theorem \ref{Whittaker-main-even-orthogonal} in terms of the global transform $A^{\psi}$ as
\begin{equation}
\label{global-Afunction-even-orthogoal}
 \mathcal{W}^{\psi_{N'}}(g,{\rm GG}(\mathcal{E}_{-k}\varphi))=A^{\psi} \left( -\frac{1}{2}, g, M_{-k}(\varphi) \right).
\end{equation}
We can identify $\mathcal{A}(\pi)$ with ${\rm Ind}^{G(\mathbb{A})}_{P(\mathbb{A})}\,\pi=\otimes_v {\rm Ind}^{G(F_v)}_{P(F_v)} \, \pi_v$.
For ${\rm Re}(s) \gg 0$, we have
\begin{multline*}
\mathcal{W}^{\psi_{N_M}}(g,M(s,\pi)\varphi)
=\int_{N_M(F) \backslash N_M(\mathbb{A})} M(s,\pi)\varphi(ng)\psi^{-1}_{N_M}(n) \,dn \\
=\int_{N_M(F) \backslash N_M(\mathbb{A})} \int_{U(\mathbb{A})} \varphi(\varrho(\mathfrak{t})w_Uung) \psi^{-1}_{N_M}(n) \,du\,dn
=\int_{U(\mathbb{A})} \mathcal{W}^{\psi_{N_M}}(\varrho(\mathfrak{t})w_Uug,\varphi) \,du
\end{multline*}
from which it follows that for $S$ large enough
\[
 \mathcal{W}^{\psi_{N_M}}(M(s,\pi)\varphi)\\
= m^S(s,\pi) \left(  \prod_{v \in S} M(s,\pi_v) W_v\right) \prod_{v \notin S} W_v,
\]
as meromorphic functions in $s \in \mathbb{C}$. Here
\[
  m^S(s,\pi)=\frac{L^S(2s,\pi,{\rm Sym}^2)}{L^S(2s+1,\pi,{\rm Sym}^2)}.
\]
Taking the limit as $s \rightarrow \frac{1}{2}$, we get for $S$ sufficiently large,
\begin{equation}
\label{limit-onehalf-even-orthogoal}
 \mathcal{W}^{\psi_{N_M}}(M_{-k}\varphi)=m^S_{-k}(\pi) \prod_{v \in S} M\left( \frac{1}{2},\pi_v \right)W_v \prod_{v \not \in S} W_v
\end{equation}
where
\begin{multline*}
m^S_{-k}(\pi) =\lim_{s \rightarrow \frac{1}{2}} \left( s-\frac{1}{2} \right)^k m^S(s,\pi)=2^{-k}\frac{\lim_{s \rightarrow \frac{1}{2}}L^S(2s,\pi,{\rm Sym}^2)}{L(2,\pi,{\rm Sym}^2)}\\
=2^{-k} \frac{\lim_{s \rightarrow 1}(s-1)^kL^S(s,\pi,{\rm Sym}^2)}{L(2,\pi,{\rm Sym}^2)}=\frac{m^S_{-k,1}(\pi)}{2^k}.
\end{multline*}
For any factorizable $\varphi \in \mathcal{A}(\pi)$, the desired conclusion follows from \eqref{global-Afunction-even-orthogoal} and \eqref{limit-onehalf-even-orthogoal}.
\end{proof}

We unfold Petersson inner products of automorphic forms against the Gelfand--Graev coefficient of Eisenstein series \cite[Theorem 10.3, (10.4)]{GRS11}
and then perform unramified computations \cite[(10.60)]{GRS11} to find the following Euler product. 

\begin{proposition} 
Let $\pi \in {\rm OScusp}_k \, \mathbb{M}$ and $\sigma=\mathcal{D}_{\psi^{-1}}(\pi)$.
Assume our Working Hypothesis \ref{even-orthogonal-working-hypothesis}. 
Let $\varphi' \in \sigma$ be an irreducible $\psi^{-1}_{N'}$-generic cuspidal automorphic representation of $G'$ and $\varphi \in \mathcal{A}(\pi)$.
Suppose that $\mathcal{W}^{\psi^{-1}_{N'}}(\cdot,\varphi')=\prod_v W'_v$
and $\mathcal{W}^{\psi_{N_M}}(\cdot,\varphi)=\prod_v W_v$. Then for any sufficiently large finite set of places $S$, we have
\begin{equation}
\label{Before-Limit-Even-Orthogoal}
\langle \varphi', {\rm GG}(\mathcal{E}(s,\varphi)) \rangle_{G'}=\frac{1}{2}\cdot \frac{{\rm vol}(N'(\mathcal{O}_S)\backslash N'(F_S))}{\prod_{j=1}^n \zeta^S_F(2j)L^S(n+1,\chi_{\tau})}
\cdot \frac{L^S(s+\frac{1}{2},\sigma \times \pi)}{L^S(2s+1,\pi,{\rm Sym}^2)} \prod_{v \in S} J(s,W'_v,W_v),
\end{equation}
 where on the right-hand side we take the unnormalized Tamagawa measure on $G'(F_S)$, $N'(F_S)$, $V(F_S)$ and $V_{\gamma}(F_S)$
(which are independent of the choice of gauge forms when $S$ is sufficiently large.)
\end{proposition}

\begin{proof}
It is worthwhile to mention that the volume of $G'(F) \backslash G'(\mathbb{A})$, appearing on the right-hand side of \eqref{Petterson-Innerproduct}, is $2$ equal to the Tamagawa number when we use the Tamagawa measure as in \cite{LM16,LM17}. We apply \cite[Theorem 10.3, (10.4)]{GRS11} in order to unfold
\begin{equation}
\label{even-orthgonal-unfolding}
 \langle \varphi', {\rm GG}(\mathcal{E}(s,\varphi)) \rangle_{G'}=\frac{1}{2} \cdot {\rm vol}(N'(F) \backslash N'(\mathbb{A})) \int_{N'(\mathbb{A})\backslash G'(\mathbb{A})} \mathcal{W}^{\psi^{-1}_{N'}}(g,\varphi') A^{\psi}(s,g,\varphi) \, dg.
\end{equation}
 Similar to \cite[Propositions 5.7 and 8.4]{LM16}, our desired result is immediate from \eqref{even-orthgonal-unfolding} in conjunction with the unramified computation \cite{Kap10,Kap15} (cf. \cite[(10.60)]{GRS11});
 \begin{equation}
 \label{even-orthgonal-spherical-comp}
  J(s,{W'_v}^{\circ},W^{\circ}_v)=\frac{{\rm vol}(K')}{{\rm vol}(N'(F_v) \cap K')}\frac{L(s+\frac{1}{2},\sigma_v \times \pi_v)}{L(2s+1,\pi_v,{\rm Sym}^2)} 
 \end{equation}
with $K'=K_{{\rm GL}_{2n+2}(F_v)} \cap G'(F_v)$. Since  we use the Tamagawa measure, the factor  $\Delta^S_{{\rm SO}_{2n+2,\tau}}(1)^{-1}$ shows up,
where $\Delta^S_{{\rm SO}_{2n+2,\tau}}(s)=L(M^{\vee}(s))$ is the partial $L$-function of the dual motive $M^{\vee}$ of the motive $M$ associated to ${\rm SO}_{2n+2,\tau}$
by Gross \cite{Gro97}. According to \citelist{\cite{Xue17}*{Notation and Convention} \cite{II10}} 
\[
 \Delta^S_{{\rm SO}_{2n+2,\tau}}(1)=\prod_{j=1}^n \zeta^S_F(2j)L^S(n+1,\chi_{\tau})
\]
for any finite set of places $S$.
\end{proof}

\subsection{Petersson inner products and Fourier coefficients}

Let $F$ be a local field. We let $\pi \in {\rm Irr}_{\rm gen,os} \, \mathbb{M}$.
We say that $\pi$ is {\it good} if the following conditions are satisfied for all $\psi$;
\begin{enumerate}[label=$(\mathrm{\roman*})$]
\item $\mathcal{D}_{\psi}(\pi)$ is irreducible. 
\item $J(s,W',W)$ is holomorphic at $s=\frac{1}{2}$ for any $W' \in \mathcal{D}_{\psi^{-1}}(\pi)$ and $W \in {\rm Ind}(\mathbb{W}^{\psi_{N_M}}(\pi))$.
\item For any $W' \in \mathcal{D}_{\psi^{-1}}(\pi)$, \\
$J\left(\frac{1}{2},W',W\right)$ factors through the map $W \mapsto \left(A^{\psi}\left(-\frac{1}{2}, \cdot, M \left( \frac{1}{2},\pi \right)W \right)\right).$
\end{enumerate}

We know from \eqref{even-orthogonal-G'-equivalent} that for $x \in G'$ we have
\[
 J(s,\sigma(x)W',I(s,x)W)=J(s,W',W),
\]
where $\sigma=\mathcal{D}_{\psi^{-1}}(\pi)$. Therefore if $\pi$ is good, there is a non-degenerate $G'$-invariant 
pairing $[\cdot,\cdot]$ on $\mathcal{D}_{\psi^{-1}}(\pi) \times \mathcal{D}_{\psi}(\pi)$ such that
\[
 J\left(\frac{1}{2},W',W\right)=\left[W',A^{\psi}\left(-\frac{1}{2},\cdot, M\left(\frac{1}{2},\pi\right)W\right) \right]
\]
for any $W' \in \mathcal{D}_{\psi^{-1}}(\pi)$ and $W \in {\rm Ind}(\mathbb{W}^{\psi_{N_M}}(\pi))$.
By \cite[\S 2]{LM15}, when $\pi$ is good, there exists $c_{\pi}$ such that
\begin{equation}
\label{cpi-equation-even-orthogoal}
 \int^{\rm st}_{N'} J\left(\frac{1}{2},\sigma(u)W',W\right) \psi_{N'}(u) \, du=c_{\pi} W'(e)A^{\psi}\left( -\frac{1}{2}, e, M \left( \frac{1}{2}, \pi \right)W\right).
\end{equation}

\begin{proposition}
\label{unramified-constant-even-orthogonal}
Suppose that $F$ is $p$-adic, $\pi \in {\rm Irr}_{\rm gen,os}\,\mathbb{M}$, and $\psi$ are unramified. 
Then $\sigma=\mathcal{D}_{\psi^{-1}}(\pi)$ is irreducible and unramified. Let $W^{\circ} \in {\rm Ind}(\mathbb{W}^{\psi_{N_M}}(\pi))$
be the standard unramified vector and let ${W'}^{\circ}$ be $K'$-invariant with ${W'}^{\circ}(e)=1$.
Then holds with $c_{\pi}=1$.
\end{proposition}

\begin{proof}
The first part is proved exactly as in the proof of \cite[Lemma 5.4]{LM17} relying on \cite[Theorem 5.6]{GRS11} instead of \cite[Theorem6.4]{GRS11}. Assuming ${\rm vol}(U \cap K)=1$ as in \eqref{unramified-even-orth-intertwining}, we know that
\[
 {\rm vol}(K')={\rm vol}(U \cap K) (\zeta^S_F(2j)L^S(n+1,\chi_{\tau}))^{-1}=(\zeta^S_F(2j)L^S(n+1,\chi_{\tau}))^{-1}.
\]
We combine it and \cite[Proposition 2.14]{LM15} to see that
\[
 \int^{\rm st}_{N'} J\left(\frac{1}{2},\sigma(u){W'}^{\circ},W^{\circ} \right) \psi^{-1}_{N'}(u) \, du=\frac{{\rm vol}(N' \cap K')}{{\rm vol}(K')} \frac{J\left(\frac{1}{2}, {W'}^{\circ},W^{\circ} \right)}{L(1,\sigma, {\rm Ad})}.
\]
It follows from the description of $\sigma$ in this case \cite[Theorem 5.6]{GRS11} together with \cite[\S 7]{GGP12} that
\[
L(1,\sigma, {\rm Ad})=L(1,\pi \boxplus \omega_{\pi}\chi_{\tau},\wedge^2)=L(1,\pi,\wedge^2)L(1, \pi\times \omega_{\pi}\chi_{\tau})
\]
and $L(s,\sigma\times \pi)=L(s,\pi \times \pi)L(s,\omega_{\pi}\chi_{\tau} \times \pi)$. We recall that $\pi=\pi^{\vee}$ in our case. Hence \eqref{unramified-even-orth-intertwining} aligned with of \eqref{even-orthgonal-spherical-comp} implies that the above is equal to
\begin{multline*}
\frac{1}{L(1,\pi,\wedge^2)L(1, \pi\times \omega_{\pi}\chi_{\tau})}\frac{L(1,\pi \times \pi)L(1,\omega_{\pi}\chi_{\tau} \times \pi)}{L(2,\pi,{\rm Sym}^2)}\\
=\frac{L(1,\pi,{\rm Sym}^2)}{L(2,\pi,{\rm Sym}^2)}=M\left(\frac{1}{2},\pi \right)W^{\circ}(e)
={W'}^{\circ}(e)A^{\psi}\left(-\frac{1}{2},e, M\left(\frac{1}{2},\pi \right)W^{\circ}\right) 
\end{multline*}
that we seek for.
\end{proof}

At first glance, $c_{\pi}$ implicitly depends on the choice of Haar measures on $G'$ (the former used in the definition of $J(s,W',W)$) and $U$
(the latter used in the definition of the intertwining operator), but not on any other groups (refer to \cite[Remark 5.2]{LM17} for more details).
However the vector spaces underlying the Lie algebras of $G'$ and $U$ are both canonically isomorphic as vector spaces to
\[
 \{ X \in {\rm Mat}_{2n+1,2n+1}(F) \, |\, {X^t}\,w_{2n+1}= -w_{2n+1}X \}.
\]
We can thereby identify the gauge forms on $G'$ and $U$ and so $c_{\pi}$
does not depend on any choice if we use the unnormalized Tamagawa measure on $G'$ and $U$ with
respect to the same gauge form. We will say that the Haar measures on $G'$ and $U$ are {\it compatible}
in this case.

\begin{proposition} 
\label{even-orthogonal-good-repn}
Let $\pi \in {\rm OScusp}_k\,\mathbb{M}$ and $\sigma=\mathcal{D}_{\psi^{-1}}(\pi)$. Assume our Working Hypothesis \ref{even-orthogonal-working-hypothesis}.
Let $\varphi' \in \sigma$ be an irreducible $\psi^{-1}_{N'}$-generic cuspidal automorphic representation of $G'$ and $\varphi \in \mathcal{A}(\pi)$.
Suppose that $\mathcal{W}^{\psi^{-1}_{N'}}(\cdot,\varphi')=\prod_v W'_v$ and 
$\mathcal{W}^{\psi_{N_M}}(\cdot,\varphi)=\prod_v W_v$.
Then for all $v$ $\pi_{v}$ is good. Moreover for any sufficiently large finite set of places $S$, we have
\begin{equation}
\label{EulerProducts-Even-Orthogonal}
\langle \varphi', {\rm GG}(\mathcal{E}_{-k} \varphi) \rangle_{G'}=\frac{1}{2}\cdot \frac{{\rm vol}(N'(\mathcal{O}_S)\backslash N'(F_S))}{\prod_{j=1}^n \zeta^S_F(2j)L^S(n+1,\chi_{\tau})} \cdot n^S_{-k,1} \prod_{v \in S} J\left(\frac{1}{2},W'_v,W_v\right),
\end{equation}
where
\[
n^S_{-k,1}=\frac{\lim_{s \rightarrow 1}(s-1)^kL^S(s,\pi \times \pi)L^S(s,\pi\times \omega_{\pi}\chi_{\tau})}{L^S(2,\pi,{\rm Sym}^2)}.
\]
\end{proposition}

\begin{proof}
We confirm that $\pi_v$ is good for all $v$. To this end, it follows from \eqref{GG-Whittaker-EvenOrthogonal-decomp} together with the irreducibility of the global descent that $\mathcal{D}_{\psi^{-1}}(\pi_v)$ is irreducible for all $v$.

\par
When $\pi \in {\rm OScusp}_k \, \mathbb{M}$ and $\sigma=\mathcal{D}_{\psi^{-1}}(\pi)$, we obtain $L^S(s,\sigma\times \pi)=L^S(s,\pi \times \pi)L^S(s,\omega_{\pi}\chi_{\tau} \times \pi)$,
which has a pole of order $k$ at $s=\frac{1}{2}$. On the one hand, $J(s,W'_v,W_v)$  is non-vanishing at $s=\frac{1}{2}$ for suitable $W'_v$ and $W_v$.
Consequently the right-hand side of \eqref{Before-Limit-Even-Orthogoal} (when $\varphi' \in \sigma$) has a pole of order at least $k$ at $s=\frac{1}{2}$ for proper $\varphi'$
and $\varphi$. On the other hand, the left-hand side of \eqref{Before-Limit-Even-Orthogoal} has a pole of order at most $k$ at $s=\frac{1}{2}$ because this is true for $\mathcal{E}(s,\varphi)$
and $\varphi'$ is rapidly decreasing. Multiplying \eqref{Before-Limit-Even-Orthogoal} by $(s-\frac{1}{2})^k$ and then taking the limit as $s \rightarrow \frac{1}{2}$,
we conclude that $J(s,W'_v,W_v)$ is holomorphic at $s=\frac{1}{2}$ for all $v$, and for $\varphi' \in \sigma$ and $\varphi \in \mathcal{A}(\pi)$ we get
\begin{equation}
\label{After-Limit-Even-Orthogoal}
\langle \varphi', {\rm GG}(\mathcal{E}_{-k} \varphi) \rangle_{G'}=\frac{1}{2}\cdot \frac{{\rm vol}(N'(\mathcal{O}_S)\backslash N'(F_S))}{\prod_{j=1}^n \zeta^S_F(2j)L^S(n+1,\chi_{\tau})} \cdot n^S_{-k,1} \prod_{v \in S} J\left(\frac{1}{2},W'_v,W_v\right).
\end{equation}
\par
We fix a place $v_0$. We assume that $W_{v_0}$ is such that 
$A^{\psi}\left(-\frac{1}{2},\cdot, M\left(\frac{1}{2},\pi_{v_0} \right)W_{v_0}\right) \equiv 0$.
Then by \eqref{GG-Whittaker-EvenOrthogonal-decomp}, $\mathcal{W}^{\psi_{N'}}(\cdot,{\rm GG}(\mathcal{E}_{-k}\varphi)) \equiv 0$
and therefore ${\rm GG}(\mathcal{E}_{-k}\varphi)\equiv 0$ benefited from the irreducibility and genericity of the descent.
We conclude from \eqref{After-Limit-Even-Orthogoal} that $J(\frac{1}{2},W'_{v_0},W_{v_0})=0$, which amounts to saying that $\pi_{v_0}$ is good.
\end{proof}

While we expect that any irreducible generic representation of orthogonal symplectic type $\pi_v$ is good,
the proof that we give under the assumption that $\pi_v$ is the local component of $\pi \in {\rm OScusp}_k\, \mathbb{M}$ is valid at the very least 
for the case of our interest, which pertains to the global period integrals. Our proof here requires a local and global argument. It behoves us to highlight the 
fact that the global approach is overkill and indirect; one may directly prove the desired result by using a local functional equation \cite[Proposition 5.3]{LM17}.
For the sake of brevity, we only include this indirect mean in keeping with the spirit of \cite[Propositions 5.7 and 8.4]{LM16}.

\begin{theorem}
\label{Main-Result-Even-Orthogonal-Theorem}
Let $\pi \in {\rm OScusp}_k\,\mathbb{M}$ and $\sigma=\mathcal{D}_{\psi}(\pi)$. Assume our Working Hypothesis \ref{even-orthogonal-working-hypothesis}.
Let $S$ be a finite set of places including all the archimedean places such that $\pi$ and
$\psi$ are all unramified outside $S$.  
Then for any $\varphi \in \sigma$ and $\varphi^{\vee} \in \sigma^{\vee}$
which are fixed under $K'_v$ for all $v \notin S$, we have
\begin{multline}
\label{Main-Result-Even-Orthogonal}
 \quad\quad \mathcal{W}^{\psi_{N'}}(\varphi) \mathcal{W}^{\psi^{-1}_{N'}}(\varphi^{\vee})
 =2^{1-k} \left( \prod_{v \in S} c^{-1}_{\pi_v} \right) \frac{\prod_{j=1}^n \zeta^S_F(2j)L^S(n+1,\chi_{\tau})}{L^S(1,\pi,\wedge^2)L^S(1, \pi\times \omega_{\pi}\chi_{\tau})} \\
 \times ({\rm vol}(N'(\mathcal{O}_S)\backslash N'(F_S)))^{-1} \int^{\rm st}_{N'(F_S)} \langle \sigma(u) \varphi, \varphi^{\vee} \rangle_{G'} \psi^{-1}_{N'}(u) \; du. \quad\quad
\end{multline}
\end{theorem}

\begin{proof}
Since the descent of $\pi$ is the space $\mathcal{D}_{\psi_{N_{\mathbb{M}}}}(\pi)$ generated by ${\rm GG}(\mathcal{E}_{-k}\varphi^{\flat})$ for $\varphi^{\flat} \in \mathcal{A}(\pi)$,
we initially deal with this for the case $\varphi^{\vee}={\rm GG}(\mathcal{E}_{-k}\varphi^{\sharp})$ with $\varphi^{\sharp} \in \mathcal{A}(\pi)$, and this identity \eqref{Main-Result-Even-Orthogonal} then extends via  linearity to all $\varphi^{\vee}$. Appealing to Proposition \ref{GG-Whittaker-EvenOrthogonal}, we see that
\[
  \mathcal{W}^{\psi_{N'}}(\varphi) \mathcal{W}^{\psi^{-1}_{N'}}({\rm GG}(\mathcal{E}_{-k}\varphi^{\sharp}))
 = \mathcal{W}^{\psi_{N'}}(\varphi) 
 \frac{ m^S_{-k,1}(\pi)}{2^k}\prod_{v \in S} A^{\psi^{-1}} \left( -\frac{1}{2},e, M \left( \frac{1}{2}, \pi_v \right)W^{\sharp}_v \right).
\]
Relying on the fact that $\pi_v$ is good in Proposition \ref{even-orthogonal-good-repn}, we use Proposition \ref{unramified-constant-even-orthogonal} to 
insert the identity \eqref{cpi-equation-even-orthogoal} for all $v$ with $c_{\pi_v}=1$ for almost all $v$. In this way, we write
\begin{equation}
\label{Even-Orthogonal-Product-Form-one}
\mathcal{W}^{\psi_{N'}}(\varphi) \mathcal{W}^{\psi^{-1}_{N'}}({\rm GG}(\mathcal{E}_{-k}\varphi^{\sharp})) 
= 
 \frac{m^S_{-k,1}(\pi)}{2^k} \left( \prod_{v \in S} c^{-1}_{\pi_v}  \int^{\rm st}_{N'(F_v)} J\left(\frac{1}{2},\sigma_v(u_v)W_v,W^{\sharp}_v\right) \psi^{-1}_{N'(F_v)}(u_v) \, du_v \right).
\end{equation}
Integrating \eqref{EulerProducts-Even-Orthogonal} over $N'(F_v)$ for $v \in S$, a product of regularized integrals becomes
\begin{multline}
\label{Even-Orthogonal-Product-Form-two}
 \prod_{v \in S}  \int^{\rm st}_{N'(F_v)} J\left(\frac{1}{2},\sigma_v(u_v)W_v,W^{\sharp}_v\right) \psi^{-1}_{N'(F_v)}(u_v) \, du_v \\
 =\frac{2\prod_{j=1}^n \zeta^S_F(2j)L^S(n+1,\chi_{\tau})}{{\rm vol}(N'(\mathcal{O}_S)\backslash N'(F_S))\cdot n^S_{-k,1} } \int^{\rm st}_{N'(F_S)} \langle \sigma(u) \varphi, {\rm GG} (\mathcal{E}_{-k}\varphi^{\sharp} \rangle_{G'} \psi^{-1}_{N'}(u) \; du.
\end{multline}
The factorization $L^S(s,\pi \times \pi)=L^S(s,\pi \times \pi^{\vee})=L^S(s,\pi,{\rm Sym}^2)L^S(s,\pi,\wedge^2)$ permits us to simplify the ratio
\begin{equation}
\label{Even-Orthogonal-Ratio}
m^S_{-k,1}/n^S_{-k,1}=1/(L^S(1,\pi, \wedge^2)L^S(1,\pi\times \omega_{\pi}\chi_{\tau}))
\end{equation}
whose partial $L$-function occurring in the denominator is nothing but the adjoint $L$-function $L^S(s,\sigma, {\rm Ad})$ for $\sigma$.
The desired conclusion follows from \eqref{Even-Orthogonal-Product-Form-one}, \eqref{Even-Orthogonal-Product-Form-two}, and  \eqref{Even-Orthogonal-Ratio}.
\end{proof}

\section{The Case for Symplectic Groups}
\label{Sec:4}

\subsection{Groups, Embeddings, and Weyl Elements}

Let $G$ be the symplectic group defined by
\[
G={\rm Sp}_{4n}=\{ g \in {\rm GL}_{4n} \,|\, \prescript{t}{}{g} J_{4n} g=J_{4n} \}, \quad \text{where} \quad J_{4n}=\begin{pmatrix} & w_{2n} \\ -w_{2n} & \end{pmatrix},
\]
and $\prescript{t}{}{g}$ is the transpose of $g$. 
The group $G$ acts on a $4n$ dimensional space with the standard basis $e_1, \dotsm, e_{2n}, e_{-2n}, e_{-2n+1}, \dotsm, e_{-1}$.
We let 
\[
G'={\rm Sp}_{2n}=\{ g \in  {\rm GL}_{2n}  \,|\, \prescript{t}{}{g} J_{2n} g=J_{2n}  \}, \quad \text{where} \quad J_{2n}=\begin{pmatrix} & w_{n} \\ -w_{n} & \end{pmatrix}.
\]
The group $G'$ is embedded as a subgroup of $G$ via $g \mapsto \eta(g):={\rm diag}(1_n,g,1_n)$, and thus it is the subgroup of $G$ consisting of elements fixing 
$e_1,\dotsm, e_n$ and $e_{-n},\dotsm, e_{-1}$. 
Let $P'=M'\ltimes U'$ be the Seigel parabolic subgroup of $G'$ with the Levi part $M'=\varrho'(\mathbb{M}')$ and the Seigel unipotent subgroup $U'$, where $\varrho' : m \mapsto {\rm diag}(1_n,m,m^{\ast},1_n)$. We set 
\[
\mathfrak{s}_n=\{ x \in {\rm Mat}_{n,n} \, | \, \check{x}=x \}.
\]
Then 
\[
 U'=\{ {\rm diag}(1_n,  \begin{pmatrix} 1_n & u \\ & 1_n \end{pmatrix}, 1_n) \,|\, u \in \mathfrak{s}_n \}.
\]
Let $N'$ be the standard maximal unipotent subgroup of $G'$. Then $N'=N'_{M'} \ltimes U'$ with $N'_{M'}=\varrho'(N'_{\mathbb{M}'})$.
Let $V$ be the unipotent radical of the standard parabolic subgroup of ${\rm Sp}_{4n}$, whose Levi part is isomorphic to ${\rm GL}^n_1 \times {\rm Sp}_{2n}$:
\[
 V=\{ \begin{pmatrix} v & u & y \\ &1_{2n} & \check{u} \\ & & v^{\ast} \end{pmatrix} \in {\rm Sp}_{4n} \,|\, v \in Z_n, u \in {\rm Mat}_{n,2n} \}.
\]
Let $P=M \ltimes U$ be the Seigel parabolic subgroup of $G$ with $U$ the standard maximal unipotent subgroup and Levi part $M=\varrho(\mathbb{M})$,
where $\mathbb{M}={\rm GL}_{2n}$ and $\varrho : h \mapsto {\rm diag}(h,h^{\ast})$.
Let
\[
 w'_{U'}={\rm diag}(1_n, \begin{pmatrix} & 1_n \\ -1_n & \end{pmatrix}, 1_n) \in G'
\]
represent the longest $M'$-reduced Weyl element of $G'$. Let 
$  w_U=\begin{pmatrix} & 1_{2n} \\ -1_{2n} & \end{pmatrix} \in G $
represent the longest $M$-reduced Weyl element of $G$.
We write \cite[(4.14),(4.15)]{GRS11}
\[
  \gamma=w_U ({w'_{U'}})^{-1}=\begin{pmatrix} &1_n&\vspace{-.1cm} & \\ \vspace{-.1cm} &&& 1_n \\ -1_n &&\vspace{-.1cm}& \\ &&1_n& \end{pmatrix}.
\]
Let us denote
\[
 V_{\gamma}=V \cap \gamma^{-1} N \gamma=\{ \begin{pmatrix} v&\vspace{-.1cm} &u&\\ &1_n&\vspace{-.1cm} &\check{u}\\ &&1_n&\vspace{-.1cm}  \\ &&& v^{\ast}  \end{pmatrix} \,|\, v \in Z_n, u \in {\rm Mat}_{n,n} \}.
\]
Let $N_M$ (resp. $N_{\mathbb{M}}$) be the standard maximal unipotent subgroup of $M$ (resp. $\mathbb{M}$).
Let $K=K_{{\rm GL}_{4n}} \cap {\rm Sp}_{4n}$. Then $K$ is a maximal compact subgroup of ${\rm Sp}_{4n}$; similarly for $K'$.

\par
Let $\tilde{G}'=\widetilde{\rm Sp}_n$ be the metaplectic group which is the two-fold cover of $G'$. 
We write elements of $\tilde{G}'$ as pairs $(g,\varepsilon)$, $g \in G$, $\varepsilon=\pm 1$ where the multiplication rule is given by Rao's cocycle \cite[\S 5]{Rao93} (see also \cite[\S 1.3]{Szp13}). When $g \in G'$, we write $\tilde{g}=(g,1) \in \tilde{G}'$. (Of course, $g \mapsto \tilde{g}$ is not a group homomorphism). 
For each subgroup $H$ of $G$, let $\tilde{H}$ be the inverse image of $H$ under the canonical projection $\tilde{G} \rightarrow G$. We let $\tilde{M}$ the double cover of $M$ inside $G$ and
we will identify $U$ with a subgroup $\tilde{U}$ via $u \mapsto \tilde{u}$. 
For $g \in G'$, let $\tilde{\eta}(g)=(\eta(g),1) \in \widetilde{\rm Sp}_{4n}$. When $F$ is $p$-adic,
and $p \neq 2$, we view $K'$ as a subgroup of $\tilde{G}'$.

\subsection{Additive Characters} We fix a non-trivial additive character $\psi$ of $F$ and a non-degenerate character $\psi_{N_{\mathbb{M}}}$ of $N_{\mathbb{M}}$.
Let $\psi_{\circ}$ be the non-trivial character of $F$ given by 
\[
\psi_{\circ}(x)=\psi_{N_{\mathbb{M}}}(1_{2n}+x\epsilon_{n,n+1}).
\]
As illustrated in \cite[Remark 6.4]{LM17} (cf. \cite[\S 4.1, \S 6.1]{LM16}), the results stated below are independent of the choice of $\psi_{N_{\mathbb{M}}}$.
For convenience, we set
\begin{equation}
\label{usual-symplectic-additive}
 \psi_{N_{\mathbb{M}}}(u)=\psi(u_{1,2}+\dotsm+u_{2n-1,2n}).
\end{equation}
Thus $\psi_{\circ}=\psi$. We define the characters of various unipotent groups. 
 We denote by $\psi_N$ the degenerate character on $N$ given by $\psi_N(uv)=\psi_{N_M}(u)$
 for any $u \in N_M$ and $v \in U$.
 Let $\psi_{N_M}$ be the non-degenerate character of $N_M$ given by $\psi_{N_M}(\varrho(u))=\psi_{N_{\mathbb{M}}}(u)$.
 Let $\psi_{N'_{M'}}$ be the non-degenerate character of $N'_{M'}$ such that $\psi_{N'_{M'}}(\varrho(u))=\psi_{N'_{\mathbb{M}'}}(u)$. With the choice of $\psi_{N_{\mathbb{M}}}$,
  we have
  \[
   \psi_{N'_{\mathbb{M}'}}(u')=\psi(u'_{1,2}+\dotsm + u'_{n-1,n}).
  \]
Let $\psi_{U'}$  be the character on $U'$ given by $\psi_{U'}(u)=\psi(\frac{1}{2}u_{2n,2n+1})^{-1}$. 
Then the non-degenerate character $\psi_{N'}$  \cite[\S 2.2]{LM17} is given by
\[
 \psi_{N'}(uv)=\psi_{N'_{M'}}(u) \psi_{U'}(v)=\psi(u_{n+1,n+2}+\dotsm+u_{2n-1,2n}-\frac{1}{2} v_{2n,2n+1})
\]
 with $u \in N'_{M'}$ and $v \in U'$. (Cf. we have specified $\lambda=-1/2$ in their notation \cite[(9.9)]{GRS11}).
 An element in $V$ can be written as $vu$ where $u$ fixes $e_1,\dotsm,e_n$, $v$ fixes $e_{n+1},e_{n+2},\dotsm,e_{-n-1}$
 and we set \cite[(3.60)]{GRS11}
 \[
  \psi_V(vu)=\psi^{-1}(v_{1,2}+\dotsm+v_{n-1,n}).
 \]

\subsection{Weil Representation}
Let $W$ be a $2n$-dimensional space over $F$, equipped with a non-degenerate, symplectic form $\langle \cdot, \cdot \rangle$.
We let $\mathcal{H}_W=W \oplus F$ be the Heisenberg group attached to $W$ with the multiplication
\[
 (x,t)\cdot(y,z)=\left( x+y,t+z+\frac{1}{2}\langle x,y \rangle \right).
\]
We write
\[
  W=W_+ \oplus W_-
 \]
as a direct sum of two maximal isotropic subspaces of $W$, which are dual under $\langle \cdot, \cdot \rangle$.
Then the group ${\rm Sp}(W)$ acts on the right on $W$. We write a prototype element of ${\rm Sp}(W)$ as
$\begin{pmatrix} A & B \\ C & D \end{pmatrix}$ where $A \in {\rm Hom}(W_+,W_+)$, $B \in {\rm Hom}(W_+,W_-)$,
$C \in {\rm Hom}(W_-,W_+)$, and $D \in {\rm Hom(W_-,W_-)}$.
Let $\gamma_{\psi}$ be Weil's factor given by $\gamma_{\psi}(a)=\gamma(\psi_a)/\gamma(\psi)$, where
$\gamma(\psi)$ is Weil's index, which is an eighth root of unity, and $\psi_a=\psi(a \cdot)$. We denote by $\gamma_{\psi}$ in \cite[(1.4)]{GRS11} the Weil index denoted by $\gamma_{\psi_{1/2}}$ in our notation.
Let $\widetilde{\rm Sp}(W)$ be the metaplectic two-fold cover of ${\rm Sp}(W)$ with respect to the Rao cocycle determined by the splitting \cite[\S 5]{Rao93} (see also \cite[\S 1.3]{Szp13}). We consider the Weil representation $\omega_{\psi}$ of the group $\mathcal{H}_W \rtimes \widetilde{\rm Sp}(W)$ on the Schwartz--Bruhat space $\mathcal{S}(W_+)$.
Explicitly for any $\Phi \in \mathcal{S}(W_+)$ and $X \in W_+$, the action of $\mathcal{H}_W$ is given by
\begin{align}
\omega_{\psi}(a,0)\Phi(X)&=\Phi(X+a), \quad a \in W_+, \label{Weil-Rep-Action-one} \\
\omega_{\psi}(b,0)\Phi(X)&=\psi(\langle X,b \rangle)\Phi(X), \quad b \in W_-, \label{Weil-Rep-Action-two} \\
\omega_{\psi}(0,t)\Phi(X)&=\psi(t)\Phi(X), \quad t \in F, \label{Weil-Rep-Action-three}
\end{align}
while the action of $\widetilde{\rm Sp}(W)$  is given by
\[
\begin{split}
 \omega_{\psi} \left( \begin{pmatrix} g & \\ & g^{\ast} \end{pmatrix}, \varepsilon \right)\Phi(X)&=\varepsilon\gamma_{\psi}(\det g) |\det g|^{\frac{1}{2}}\Phi(Xg), \quad g \in {\rm GL}(W_+), \\
 \omega_{\psi} \left( \begin{pmatrix} 1_n & B \\ & 1_n \end{pmatrix}, \varepsilon \right) \Phi(X) &=\varepsilon\psi\left( \frac{1}{2} \langle X, XB \rangle \right)\Phi(X),
 \quad B \in {\rm Hom}(W_+,W_-) \text{\;\;self-dual}.
\end{split}
\]

\quad
We now take $W=F^{2n}$ with the standard symplectic form
\[
 \langle (x_1,\dotsm,x_{2n}), (y_1,\dotsm,y_{2n}) \rangle=\sum_{i=1}^n x_i y_{2n+1-i}-\sum_{i=1}^n y_i x_{2n+1-i}
\]
and the standard polarization $W_+=\{ x_1,\dotsm,x_n,0,\dotsm,0 \}$, $W_-=\{(0,\dotsm,0,y_1,\dotsm,y_n) \}$, with the standard basis $e_1,\dotsm,e_n,e_{-n},\dotsm,e_{-1}$.
For $X=(x_1,\dotsm,x_n)$ and $Y=(y_1,\dotsm,y_n) \in F^n$, we define
\[
 \langle X,Y \rangle'=\langle X, Y \begin{pmatrix} & 1_n \\ -1_n & \end{pmatrix} \rangle
 =x_1y_n+\dotsm+x_ny_1.
\]
For $\Phi$ in the Schwartz--Bruhat space $\mathcal{S}(F^n)$, we define the Fourier transform
\[
 \hat{\Phi}(X)=\int_{F^n} \Phi(Y) \psi \left( \langle X, Y \rangle' \right) \, dY,
\]
where we take the self-dual Haar measure on $F$ with respect to $\psi$
and the product measure on $F^n$. Then, reapplied on $\mathcal{S}(F^n)$, the Weil representation satisfies
\[
\begin{split}
 \omega_{\psi}(\widetilde{\rho'(g)})\Phi(X)&=\beta_{\psi}(\rho'(g))|\det g|^{\frac{1}{2}}\Phi(Xg), \quad g \in \mathbb{M}',\\
 \omega_{\psi}(\widetilde{w'_{U'}})\Phi(X)&=\beta_{\psi}(w'_{U'})\hat{\Phi}(X),\\
 \omega_{\psi} \left( \widetilde{\begin{pmatrix} 1_n & B \\ & 1_n \end{pmatrix}} \right) \Phi(X)&=\psi\left( \frac{1}{2} \langle X,XB \rangle' \right) \Phi(X), \quad B \in \mathfrak{s}_n,
\end{split}
\]
where $\beta_{\psi}$ is a certain eight root of unity, which satisfies $\beta_{\psi}(\cdot)^{-1}=\beta_{\psi^{-1}}(\cdot)$.
\par

Let $V_0 \subset V$ be the unipotent radical of the standard parabolic subgroup of $G$ with Levi ${\rm GL}^{n-1}_1 \times {\rm Sp}_{2n+2}$.
Then the map 
\[
  v \mapsto v_{\mathcal{H}_{F^{2n}}}:=\left((v_{n,n+j})_{j=1,2,\dotsm,2n}, \frac{1}{2} v_{n,3n+1} \right)
\]
gives an isomorphism from $V / V_0$ to a Heisenberg group $\mathcal{H}_{F^{2n}}$.
Then we may regard $\omega_{\psi}$ as a representation of $V / V_0  \rtimes \widetilde{\rm Sp}(F^{2n})$.
Starting with $\psi_{N_{\mathbb{M}}}$, and with $\psi_{\circ}$ and $\psi_V$ as above, we extend 
$\omega_{\psi_0}$ to a representation $\omega_{\psi_{N_{\mathbb{M}}}}$ of $V \rtimes \tilde{G}'$ by setting
\begin{equation}
\label{Weil-Rep-Action-Extensions}
 \omega_{\psi_{N_{\mathbb{M}}}}(v\tilde{g})\Phi=\psi_V(v) \omega_{\psi_{\circ}}(v_{\mathcal{H}_{F^{2n}}})(\omega_{\psi_{\circ}}(\tilde{g})\Phi),
 \quad v \in V, g \in G'.
\end{equation}

\par
Our $\omega_{\psi_{\circ}}$ corresponds to the definition given in \cite[(1.4)]{GRS11} with the conductor $\psi_{\circ}$
there replaced by $\psi_{\circ}(\frac{1}{2}x)$. We find it more convenient to use this convention and our convention follows from those of Kaplan \cite[\S 3.3.1, Remark 4.4]{Kap15} and of Lapid and Mao \cite{LM16,LM17}.

\subsection{Whittaker Models and the Intertwining Operator}

Let $\pi$ be an irreducible generic representation of $\mathbb{M}$ and let $\Pi$ denote the genuine representation $\pi \times \gamma_{\psi}$
of the double cover $\tilde{M}$ of $M$ defined by 
$ (\pi \times \gamma_{\psi})(\varrho(g),\varepsilon)=\varepsilon \gamma_{\psi}(\det g) \pi(g)$ for $(g,\varepsilon) \in \mathbb{M} \times \{ \pm 1 \}$.
We denote by ${\rm Irr}_{{\rm gen},\psi_{N_M}} \, \tilde{M}$ the class of irreducible genuine generic representations of $\tilde{M}$ with respect to $\psi_{N_M}$.
The analogous notation applies to ${\rm Irr}_{{\rm gen},\psi_{N_{\mathbb{M}}}} \, \tilde{\mathbb{M}}$.
We use $\mathbb{W}^{\psi_{N_{\mathbb{M}}}}(\pi)$ to denote the (uniquely determined) Whittaker space of $\pi$
with respect to the character $\psi_{N_{\mathbb{M}}}$. Similarly we use the notation $\mathbb{W}^{\psi^{-1}_{N_{\mathbb{M}}}}(\pi)$,
$\mathbb{W}^{\psi_{N_M}}(\pi)$, $\mathbb{W}^{\psi^{-1}_{N_M}}(\pi)$, $\mathbb{W}^{\psi_{N_{M}}}(\Pi)$, $\mathbb{W}^{\psi^{-1}_{N_{M}}}(\Pi)$,
$\mathbb{W}^{\psi_{N_{\mathbb{M}}}}(\Pi)$, $\mathbb{W}^{\psi^{-1}_{N_{\mathbb{M}}}}(\Pi)$.

\par
Using the Iwasawa decomposition, we extend the character $\varrho(g) \mapsto | \det g|$, $g \in \mathbb{M}$ to a right $K$ left $U$-invariant $\nu(g)$ 
function on $G$. For any genuine function $f \in C^{\infty}(\tilde{G})$ and $s \in \mathbb{C}$ define $f_s(\tilde{g})=f(\tilde{g})\nu(g)^s$, $g \in G$. 
Let ${\rm Ind}(\mathbb{W}^{\psi_{N_M}}(\Pi))$ be the space of $\tilde{G}$-smooth, genuine, left $\tilde{U}$-invariant function $W : \tilde{G} \rightarrow \mathbb{C}$
such that for all $g \in G$, the function $\tilde{m} \mapsto \delta^{\frac{1}{2}}_P(m)W(\tilde{m}\tilde{g})$ on $\tilde{M}$
belongs to $\mathbb{W}^{\psi_{N_M}}(\Pi)$. Similarly we define ${\rm Ind}(\mathbb{W}^{\psi^{-1}_{N_M}}(\Pi))$, ${\rm Ind}(\mathbb{W}^{\psi_{N_M}}(\pi))$,
${\rm Ind}(\mathbb{W}^{\psi^{-1}_{N_M}}(\pi))$, ${\rm Ind}(\mathbb{W}^{\psi_{N_{\mathbb{M}}}}(\pi))$, ${\rm Ind}(\mathbb{W}^{\psi^{-1}_{N_{\mathbb{M}}}}(\pi))$.
For any $s \in \mathbb{C}$, we define the intertwining operator \cite[\S 3.3]{Kap15}
\[
 M(s,\Pi) : {\rm Ind}(s,\mathbb{W}^{\psi_{N_M}}(\Pi)) \rightarrow {\rm Ind}(-s,\mathbb{W}^{\psi_{N_M}}(\Pi^{\vee}))
\]
by (the analytic continuation of)
\[
 M(s,\Pi ) \tilde{W}(\tilde{g})=\nu(g)^s \int_U \tilde{W}_s(\tilde{\varrho}(\mathfrak{t})\tilde{w}_U \tilde{u} \tilde{g}) \, du
\]
where $\mathfrak{t} \in {\rm GL}_{2n}$ is introduced in order to preserve the character $\psi_{N_M}$ and $\widetilde{\varrho}(m)=(\varrho(m),1) \in \tilde{G}$.

\par
In the case when $F$ is $p$-adic with $p$ odd, \cite[Lemma 4.2]{Szp13}, and there exist (necessarily unique) $\tilde{K}$-fixed elements
$\tilde{W}^{\circ} \in {\rm Ind}(\mathbb{W}^{\psi_{N_M}}(\Pi))$ and $\tilde{W}^{{\vee}^{\circ}} \in {\rm Ind}(\mathbb{W}^{\psi_{N_M}}(\Pi^{\vee}))$ such that
\begin{equation}
\label{unramified-symplectic-intertwining}
 M(s,\Pi) \tilde{W}^{\circ}={\rm vol}(U \cap K) \frac{L(2s,\pi,{\rm Sym}^2)}{L(2s+1,\pi,{\rm Sym}^2)}\tilde{W}^{{\vee}^{\circ}}.
\end{equation}
The following result is an analogue of \cite[Proposition 2.1]{LM16} and \cite[Proposition 4.1]{LM17}.

\begin{proposition}
\label{holomorphy-symplectic}
Let $\Pi$ be a genuine representation of $\tilde{\mathbb{M}}$ such that $\Pi \cong \Pi^{\vee}$. Then $M(s,\Pi)$ is holomorphic at $s=\frac{1}{2}$.
\end{proposition}

\begin{proof}
We let $\Pi$ be the irreducible generic quotient of a representation parabolically induced from $(\sigma_1 \times \gamma_{\psi})\otimes \dotsm \otimes (\sigma_k \times \gamma_{\psi})$ where $\sigma_1,\dotsm,\sigma_k$ are essentially square-integrable. It can be seen in the course of the proof of \cite[Lemma 3.3]{Szp13} that $M(s,\Pi)$ 
is the restriction of $M(s,\pi)$ where $\pi=\sigma_1 \boxplus \dotsm \boxplus \sigma_k$. We decompose $M(s,\pi)$ as the composition of co-rank one intertwining operators.
Just as in Proposition \ref{holomorphy-even-orthgonal}, the problem boils down to the holomorphy of the following intertwining operators 
\begin{enumerate}[label=$(\mathrm{\roman*})$]
\item  \label{symplectic-intertwining-differ} $\sigma_i|\cdot|^s \boxplus \sigma_j^{\vee}|\cdot|^{-s} \rightarrow \sigma_j^{\vee}|\cdot|^{-s} \boxplus \sigma_i|\cdot|^s,\quad, i,j=1,\dotsm,k, i \neq j$,
\item \label{symplectic-intertwining-same} $M(s,\sigma_i), \quad i=1,2,\dotsm,k$,
\end{enumerate}
at $s=1/2$. Then we can derive from \cite[Lemma 4.2]{LM17} and the irreducibility of $\sigma_i \boxplus \sigma_j^{\vee}$ that the intertwining operator \ref{symplectic-intertwining-differ} is holomorphic at $s=1/2$. The holomorphy of  \ref{symplectic-intertwining-same} is an immediate consequence of \cite[Lemma 4.3]{LM17}
and the irreducibility of $\sigma_i \boxplus \sigma_i^{\vee}$.
\end{proof}

\subsection{Descent map}

\subsubsection{Global Descent}

Suppose that $F$ is a number field. Let $\tilde{G}'(\mathbb{A}):=\widetilde{\rm Sp}_{2n}(\mathbb{A})$ be the metaplectic two-fold cover of $G'(\mathbb{A})$
which splits over $G'(F)$. The global double cover $\tilde{G}'(\mathbb{A})$ of $G'(\mathbb{A})$ is compatible with the local 
double covers at each place. (See, for example, \cite[\S 2.2]{JS07} for conventions.)
Let $P_{\circ} \subseteq \mathbb{M}$ be the standard parabolic subgroup in $\mathbb{M}$
whose Levi subgroup is isomorphic to $\prod_{i=1}^k {\rm GL}_{n_i}$ with $n_1+\dotsm+n_k=2n$ and $n_i > 1$ for all $i$.
For $\pi \in {\rm OScusp}_k \, \mathbb{M}$, we associate it to the genuine representation $\Pi={\rm Ind}^{\tilde{\mathbb{M}}}_{\tilde{P}_{\circ}}\, \widetilde{\pi}$,
where $\tilde{\pi}$ acts as (cf. \cite[(2.11)]{GRS11})
\begin{equation}
\label{block-compatible}
 \tilde{\pi}((m,\varepsilon))=\prod_{1 \leq i < j \leq k} (\det m_i,\det m_j) \prod_{i=1}^k \gamma_{\psi}(\det m_i) \otimes_{i=1}^k \pi_i(m_i),
\end{equation}
for $m={\rm diag}(m_1,\dotsm,m_k) \in \mathbb{M}$. We let ${\rm Scusp}_k\, \tilde{\mathbb{M}}$ be the set of representations of $\Pi$ constructed as above and let ${\rm Scusp}\,  \tilde{\mathbb{M}}=\cup_k {\rm Scusp}_k \, \tilde{\mathbb{M}}$. Any genuine representation $\Pi$ in  ${\rm Scusp}\,  \tilde{\mathbb{M}}$ is called {\it a representation of symplectic type}.

\par
For $\Sigma \in {\rm Scusp}\,  \tilde{\mathbb{M}}$, we view it as a representation $\Pi$ of $\tilde{M}$ via $\Sigma((\varrho(m),\varepsilon)):=\Pi((m,\varepsilon))$. Let $\mathcal{A}(\Pi)$ be
the space of functions $\tilde{\varphi} : \tilde{M}(F)\tilde{U}(\mathbb{A}) \backslash \tilde{G}(\mathbb{A}) \rightarrow \mathbb{C}$
such that $\tilde{m} \mapsto \delta_P^{\frac{1}{2}}(m)\varphi(\tilde{m}\tilde{g})$, $m \in M(\mathbb{A})$ belongs to 
the space of $\Pi$ for all $\tilde{g} \in \tilde{G}(\mathbb{A})$. 
For $\Phi \in \mathcal{S}(\mathbb{A}^n)$ we define the theta function
\[
 \Theta^{\Phi}_{\psi_{N_{\mathbb{M}}}}(v\tilde{g})=\sum_{\xi \in F^n} \omega_{\psi_{N_{\mathbb{M}}}}(v \tilde{g})\Phi(\xi), \quad v \in V(\mathbb{A}), g \in G'(\mathbb{A}).
\]
let ${\rm FJ}_{\psi_{N_{\mathbb{M}}}}$ be the {\it Fourier--Jacobi coefficient} (a function on $G'(F) \backslash G'(\mathbb{A})$) 
\[
 {\rm FJ}_{\psi_{N_{\mathbb{M}}}}(\tilde{\varphi},\Phi)(g)=\int_{V(F) \backslash V(\mathbb{A})} \tilde{\varphi}(\tilde{v}\tilde{g}) \Theta^{\Phi}_{\psi^{-1}_{N_{\mathbb{M}}}}(v\tilde{g}) \, dv, \quad g \in G'(\mathbb{A}).
\]
Let $\tilde{\varphi}_s(g)=\nu(g)^s\tilde{\varphi}(\tilde{g})$.
We define Eisenstein series
\[
 \mathcal{E}(s,\tilde{\varphi})=\sum_{\gamma \in P(F) \backslash G(F)} \varphi_s(\tilde{\gamma}\tilde{g}), \quad \tilde{\varphi} \in \mathcal{A}(\Pi).
\]
The space of $\Pi$ is invariant under the conjugation by $\tilde{w}_U$. Then we define intertwining operator $M(s,\Pi) : \mathcal{A}(\Pi) \rightarrow \mathcal{A}(\Pi)$
given by
\[
 M(s,\Pi) \tilde{\varphi}(\tilde{g})=\nu(g)^s \int_{U(\mathbb{A})}\tilde{\varphi}_s(\tilde{w}_U\tilde{u}\tilde{g}) \, du.
\]
Both $\mathcal{E}(s,\tilde{\varphi})$ and $M(s,\Pi)$ converge absolutely for ${\rm Re}(s) \gg 0$ and extend meromorphically to $\mathbb{C}$.
It follows from \cite[Theorem 2.1]{GRS11} that both $\mathcal{E}(s,\tilde{\varphi})$ and $M(s,\Pi)$ have a pole of order $k$ at $s=\frac{1}{2}$.
Let
\[
 \mathcal{E}_{-k} \tilde{\varphi}=\lim_{s \rightarrow \frac{1}{2}} \left( s-\frac{1}{2} \right)^k \mathcal{E}(s,\tilde{\varphi})
\]
and
\[
 M_{-k}=\lim_{s \rightarrow \frac{1}{2}} \left(s-\frac{1}{2} \right)^k M(s,\Pi)
 \]
 be the leading coefficient in the Laurent series expansion at $s=\frac{1}{2}$. The constant term
 \[
 \mathcal{E}^U_{-k} \tilde{\varphi}(\tilde{g})=\int_{U(F) \backslash U(\mathbb{A})} \mathcal{E}_{-k} \tilde{\varphi}(\tilde{u}\tilde{g}) \,du
 \]
 of $\mathcal{E}_{-k} \tilde{\varphi}$ along $U$ is given by
 \[
  \mathcal{E}^U_{-k} \tilde{\varphi}(\tilde{g})=(M_{-k}(\tilde{\varphi})(\tilde{g}))_{-\frac{1}{2}}.
 \]
The {\it descent} of $\Pi$ (with respect to $\psi_{N_{\mathbb{M}}}$) is the space $\mathcal{D}_{\psi}(\Pi)$ generated by ${\rm FJ}_{\psi_{N_{\mathbb{M}}}}(\mathcal{E}_{-k}\tilde{\varphi},\Phi)$ with $\tilde{\varphi} \in \mathcal{A}(\Pi)$ and $\Phi \in \mathcal{S}(\mathbb{A}^n)$.
It is known from \cite[Theorem 3.1]{GRS11} that $\mathcal{D}_{\psi}(\Pi)$ is multiplicity free. 
By \cite[Theorem 9.7]{GRS11}, $\mathcal{D}_{\psi}(\Pi)$ is a non-trivial cuspidal (and globally generic) automorphic representation of $G'$.
We expect it to be irreducible. This result will be the analogue of  \cite[Theorem 2.3]{GJS12}  in metaplectic group case and \cite[Theorem 1.1]{Mor18} in even unitary group case. It is likely that the method of  \cite[Theorem 2.3]{GJS12} and \cite[Theorem 1.1]{Mor18}
is extended to the case at hand.

\begin{working hypothesis}
\label{symplectic-working-hypothesis}
When $\Pi \in {\rm Scusp}\, \tilde{\mathbb{M}}$, $\mathcal{D}_{\psi}(\Pi)$ is irreducible.
\end{working hypothesis}

The image of the weak lift is then a consequence of \cite[Theorem 11.2]{GRS11}.
In particular, \cite[Theorem 9.7 (2)]{GRS11} describe the image of the lift for unramified places.

\begin{proposition} With the above Working Hypothesis \ref{symplectic-working-hypothesis}, the map $\Pi={\rm Ind}^{\tilde{\mathbb{M}}}_{\tilde{P}_{\circ}} \tilde{\pi} \mapsto \sigma=\mathcal{D}_{\psi}(\Pi)$ 
defines a bijection between ${\rm Scusp}\, \tilde{\mathbb{M}}$ and ${\rm Cusp}_{\psi_{N'}}\, G'$, the set of $\psi_{N'}$-generic 
irreducible cuspidal automorphic representations of $G'$. Moreover $\sigma^{\vee}$ weekly lifts to $\pi \boxplus \omega_{\pi}$.
\end{proposition}

\subsubsection{Local descent}

We go back to the local setup. Suppose now that $F$ is local fields. Let $\Pi \in {\rm Irr}_{{\rm gen}}\, \tilde{M}$.
For any $\tilde{W} \in {\rm Ind}(\mathbb{W}^{\psi_{N_M}}(\Pi))$, $\Phi \in \mathcal{S}(F^n)$, and $s \in \mathbb{C}$, we define a function on $G'$:
\[
 A^{\psi}(s,g,\tilde{W},\Phi)=\int_{V_{\gamma} \backslash V} \tilde{W}_s(\tilde{\gamma}\tilde{v}\tilde{\eta}(g)) \omega_{\psi_{N_{\mathbb{M}}}}(v \tilde{g}) \Phi(\xi_n) \, dv, \quad g \in G'.
\]
The basic properties of $A^{\psi}$ are summarized in the following proposition. Its proof is identical to \cite[Lemmata 4.5 and 4.9]{LM16} and \cite[Lemmata 5.1 and 8.1]{LM17},
and will be omitted.

\begin{proposition}
Suppose that $\Pi \in {\rm Irr}_{{\rm gen}}\, \tilde{M}$. Let $\tilde{W} \in {\rm Ind}(\mathbb{W}^{\psi_{N_M}}(\Pi))$, $\Phi \in \mathcal{S}(F^n)$, and $s \in \mathbb{C}$. Then
\begin{enumerate}[label=$(\mathrm{\arabic*})$]
\item The integral $A^{\psi}(s,g,\tilde{W},\Phi)$ is well defined and absolutely convergent uniformly for $s \in \mathbb{C}$ and $g \in G'$
in compact sets. Thus $A^{\psi}(s,g,\tilde{W},\Phi)$ is entire in $s$ and smooth in $g$. In the non-archimedean case the integrand is compactly supported on $V_{\gamma} \backslash V$.
\item For any $\tilde{W} \in {\rm Ind}(\mathbb{W}^{\psi_{N_M}}(\Pi))$ and $s \in \mathbb{C}$, the function $g \mapsto A^{\psi}(s,g,\tilde{W},\Phi)$
is smooth and $(N',\psi_{N'})$-equivariant.
\item For any $g, x \in G'$ and $v \in V$, we have
\begin{equation}
\label{symplectic-G'-equivalent}  
 A^{\psi}(s,g,I(s,\tilde{v}\tilde{\eta}(x))\tilde{W},\omega_{\psi^{-1}_{N_{\mathbb{M}}}}(\tilde{v} \tilde{x})\Phi)=A^{\psi}(s,gx,\tilde{W},\Phi).
\end{equation}
\item Suppose that $p \neq 2$, $\Pi$ is unramified, $\psi$ is unramified, $\tilde{W}^{\circ}\in {\rm Ind}(\mathbb{W}^{\psi_{N_M}}(\Pi))$ is the standard 
unramified vector, and $\Phi_0={\bf 1}_{\mathcal{O}^n}$. Then $A^{\psi}(s,e,\tilde{W}^{\circ},\Phi_0)\equiv 1$ (assuming ${\rm vol}(V \cap K)={\rm vol}(V_{\gamma} \cap K)=1$).
\end{enumerate}
\end{proposition}

Let $\pi_i, i=1,\dotsm,k$ be pairwise inequivalent (that is, $\pi_i \not\cong \pi_j$ 
if $i \neq j$) and of orthogonal symplectic type with $n_1+\dotsm+n_k=2n$ with $n_i > 1$ for all $i$.
Let $\pi={\rm Ind}^{\mathbb{M}}_{P_{\circ}}\otimes \pi_i$ of $\mathbb{M}$.
We relate it to the genuine representation $\Pi={\rm Ind}^{\tilde{\mathbb{M}}}_{\tilde{P_{\circ}}}\tilde{\pi}$, where $\tilde{\pi}$ is defined in \eqref{block-compatible}.
Such a genuine representation $\Pi$ is called {\it a representation of symplectic type}.
\par

For $F$ a local field, we write ${\rm Irr}_{\rm sym}\tilde{\rm GL}_n(F)$ for the set of irreducible genuine representations of symplectic type.
 It is clear that if a genuine representation $\Pi$ is of symplectic type in the global setting, then all its local components $\Pi_v$ are irreducible and of symplectic type as well.

\par
 Let $\Pi \in {\rm Irr}_{\rm gen} \, \tilde{\mathbb{M}}$, considered also as a representation of $\tilde{M}$ via $\varrho$. By the similar argument as in \cite[Theorem in \S 1.3]{GRS99}
 together with \cite[Remark 4.13]{LM17}, for any non-zero subrepresentation $\Pi'$ of ${\rm Ind}(\mathbb{W}^{\psi_{N_M}}(\Pi))$, there exist $\tilde{W}' \in \Pi'$ and $\Phi \in \mathcal{S}(F^n)$
 such that  $A^{\psi}(s,\cdot,\tilde{W}',\Phi) \not \equiv 0$.
 
 \par
 We assume that $\Pi \in {\rm Irr}_{{\rm gen,sym}} \tilde{\mathbb{M}}$. By Proposition \ref{holomorphy-symplectic}, $M(s,\Pi)$ is holomorphic at $s=\frac{1}{2}$. The {\it local descent} of $\pi$ is the space $\mathcal{D}_{\psi}(\Pi)$ of Whittaker functions on $G'$
 generated by $A^{\psi}(-\frac{1}{2},\cdot, M ( \frac{1}{2},\Pi) \tilde{W},\Phi )$ with $\tilde{W} \in {\rm Ind}(\mathbb{W}^{\psi_{N_M}}(\Pi))$
 and $\Phi \in \mathcal{S}(F^n)$. By the above observation, $\mathcal{D}_{\psi}(\Pi) \neq 0$.
 
 \par
Let $\Pi'$ be the image of ${\rm Ind}(\mathbb{W}^{\psi_{N_M}}(\Pi))$ under $M(\frac{1}{2},\Pi)$.
Thanks to \eqref{symplectic-G'-equivalent}, the space $\mathcal{D}_{\psi}(\Pi)$ is canonically  a quotient of the $G'$-module $J_V(\Pi)$
of the $V$-coinvariant of $\Pi'$. In this regard, we view $J_V(\Pi')$ as the {\it abstract descent},
whereas $\mathcal{D}_{\psi}(\Pi)$ as the {\it explicit descent}.

\subsubsection{Whittaker function of descent} 
We let $\hat{\psi}_N$ be a degenerate character on $N$ that is trivial on $U$ and for $u \in N_{\mathbb{M}}$
\[
  \hat{\psi}(\varrho(u)):=\psi(u_{1,2}+\dotsm+u_{n,n+1}-u_{n+1,n+2}-\dotsm - u_{2n-1,2n}).
\]
to be consistent with the additive character in \cite[(8.14)]{GRS11}. Thus $\psi_{\circ}(x)=\psi(x)$.
By \cite[Remark 6.4]{LM17} (cf. \cite[\S 4.1, \S 6.1]{LM16}), we confirm Theorem \ref{Whittaker-main-symplectic} for a specific choice of $\psi_{N_{\mathbb{M}}}$
given by 
\[
  \psi_{N_{\mathbb{M}}}(u)=\psi(u_{1,2}+\dotsm+u_{n-1,n}+2u_{n,n+1}-u_{n+1,n+2}-\dotsm -u_{2n-1,2n}).
\]
just as in \cite[Theorem 10.3]{GRS11}. Henceforth $\psi_{\circ}(x)=\psi(2x)$. 
Let
\[
 \delta=\left( \begin{array}{c|c} \delta'_{n,2n} & 0_{n \times 2n} \\ 0_{n \times 2n} & \delta'_{n,2n} \\ \hline \delta''_{n,2n} & 0_{n \times 2n} \\ 0_{n \times 2n} & \delta''_{n,2n}  \end{array} \right).
\]
We let $\kappa=\varrho(\kappa')$, where $\kappa'$ is the Weyl element of $\mathbb{M}$ such that $\kappa'_{2i,i}=\kappa'_{2i-1,n+i}=1$, $i=1,\dotsm,n$,
and all other entries are zero. Let $\epsilon={\rm diag}(A,\dotsm,A,A^{\ast},\dotsm,A^{\ast})$, where $A=\begin{pmatrix*}[r] 1 & 1 \\ -1 & 1\end{pmatrix*}$.
Here $A$ is repeated $n$ times (cf. \cite[Proof of Proposition 9.5]{GRS11}). Writing
\[
 A=\begin{pmatrix} 1 & 1 \\  & 1 \end{pmatrix} \begin{pmatrix} 2 &  \\  & 1  \end{pmatrix} \begin{pmatrix*}[r] 1 &  \\ -1 & 1  \end{pmatrix*}
\]
we get $\epsilon=\epsilon_U\epsilon_T\epsilon_{\overline{U}}$, where
\[
  \epsilon_U={\rm diag}( \begin{pmatrix} 1 & 1 \\ & 1 \end{pmatrix}, \dotsm, \begin{pmatrix} 1 & 1 \\ & 1 \end{pmatrix}, \begin{pmatrix*}[r] 1 & -1 \\ & 1 \end{pmatrix*}
  ,\dotsm, \begin{pmatrix*}[r] 1 & -1 \\ & 1 \end{pmatrix*} ),
\]
\[
 \epsilon_{\overline{U}}={\rm diag}(\begin{pmatrix*}[r] 1 & \\ -1 & 1 \end{pmatrix*}, \dotsm, \begin{pmatrix*}[r] 1 & \\ -1 & 1 \end{pmatrix*}, \begin{pmatrix} 1 & \\ 1 & 1 \end{pmatrix}, \dotsm, \begin{pmatrix} 1 & \\ 1 & 1 \end{pmatrix}),
\]
and
\[
 \epsilon_T={\rm diag}(\begin{pmatrix} 2 & \\ & 1 \end{pmatrix}, \dotsm, \begin{pmatrix} 2 & \\ & 1 \end{pmatrix}, \begin{pmatrix} 1 & \\ & \frac{1}{2} \end{pmatrix}, \dotsm, 
 \begin{pmatrix} 1 & \\ & \frac{1}{2} \end{pmatrix}).
\]
Let $X$ be the subspace of $\mathfrak{s}_{2n}$ consisting of the strictly upper triangular matrices. We let
\[
 \overline{\ell}(x)=\begin{pmatrix} 1_{2n} & \\ x & 1_{2n} \end{pmatrix}
\]
for $x \in \mathfrak{s}_{2n}$. Let $Y$ be the subgroup of $N_M$ consisting of the matrices of the form $\varrho(\begin{pmatrix} 1_n & y \\ & 1_n \end{pmatrix})$
where $y$ is lower triangular, (namely $y_{i,j}=0$ if $j > i$). Let $Y'$ be the subgroup of $Y$ consisting of the matrices of the form
$\varrho(\begin{pmatrix} 1_n & x \\ & 1_n \end{pmatrix})$ where $x$ is lower triangular with last row $0$ and $a^+=\varrho(\begin{pmatrix} 1_n & z \\ & 1_n \end{pmatrix})$ 
where the last row of $z$ is $a$ and all other rows are zero.

\begin{theorem}
\label{Whittaker-main-symplectic} 
$(${\rm Reformulation of \cite[Theorem 9.7, part (1)]{GRS11}}$)$ 
Let $\Pi \in {\rm Scusp} \, \mathbb{M}$ and $\tilde{\varphi}$. Then for any $g \in G'$ we have
\[
 \mathcal{W}^{\psi_{N'}}(g,{\rm FJ}_{\psi_{N_{\mathbb{M}}}}(\mathcal{E}_{-k} \tilde{\varphi},\Phi))=\int_{V_{\gamma}(\mathbb{A}) \backslash V(\mathbb{A})} \mathcal{W}^{\psi_{N}}( \tilde{\gamma} \tilde{v} \tilde{\eta}(g), \mathcal{E}_{-k} \tilde{\varphi}) \omega_{\psi^{-1}_{N_{\mathbb{M}}}}(v \tilde{g}) \Phi(\xi_n) \, dv,
\]
where the integral is absolutely convergent.
\end{theorem}

\begin{proof}
We use Tamagawa measure in the proof. It is enough to prove the required identity for $g=e$. 
The expression for $ \mathcal{W}^{\psi_{N'}}(e,{\rm FJ}_{\psi_{N_{\mathbb{M}}}}(\mathcal{E}_{-k} \tilde{\varphi},\Phi)$
in \cite[Theorem 9.7, part (1)]{GRS11} is (with $\lambda=-1$ in their notation):
\begin{equation}
\label{symplectic-decomp}
 \int_{\mathbb{A}^n} \left( \int_{Y'(\mathbb{A})} \left( \int_{X(\mathbb{A})} \mathcal{W}^{\hat{\psi}_{N}}(\tilde{\overline{\ell}}(x)\tilde{\delta}\tilde{\epsilon}\tilde{\kappa}\tilde{y}'\tilde{a}^+,\mathcal{E}_{-k}\tilde{\varphi}) \Phi(a) \,dx \right) \, dy' \right) \, da.
\end{equation}
Combining $\tilde{y}'$ and $\tilde{a}^+$, it is evident from \eqref{Weil-Rep-Action-one} and \eqref{Weil-Rep-Action-Extensions} that we can rewrite the above expression as
\[
 \int_{Y(\mathbb{A})} \left( \int_{X(\mathbb{A})} \mathcal{W}^{\hat{\psi}_{N}}(\tilde{\overline{\ell}}(x)\tilde{\delta}\tilde{\epsilon}\tilde{\kappa}\tilde{y},\mathcal{E}_{-k}\tilde{\varphi}) \omega_{\psi^{-1}_{N_{\mathbb{M}}}}(y) \Phi(0) \,dx \right) \, dy. 
\]
We note from \cite[\S 8.2, (8.15)]{GRS11} that $\delta\epsilon_U\delta^{-1}=\begin{pmatrix*}[r] 1_{2n} & -1_{2n} \\  & 1_{2n} \end{pmatrix*}$.
For any $x \in X(\mathbb{A})$, we proceed to write
\[
 \begin{pmatrix} 1_{2n} & \\ x & 1_{2n} \end{pmatrix} \begin{pmatrix*}[r] 1_{2n} & -1_{2n} \\  & 1_{2n} \end{pmatrix*}=\begin{pmatrix} (1_{2n}-x)^{-1} & -1_{2n} \\ & 1_{2n}-x \end{pmatrix} \begin{pmatrix} 1_{2n} & \\ x' & 1_{2n} \end{pmatrix},
\]
where $x'=(1_{2n}-x)^{-1}x=x+x^2+\dotsm+x^{2n-1}$. (Note that $(1_{2n}-x)^{-1}=(1_{2n}-x)^{\ast}$.) After a change of variable $x \mapsto (1_{2n}-x)x$ we get
\[
 \int_{Y(\mathbb{A})} \left(\int_{X(\mathbb{A})} \psi(x_{n,n+1})\mathcal{W}^{\hat{\psi}_{N}}( \tilde{\overline{\ell}}(x)\tilde{\delta}\tilde{\epsilon}_T\tilde{\epsilon}_{\overline{U}}\tilde{\kappa}\tilde{y},\mathcal{E}_{-k}\tilde{\varphi}) \omega_{\psi^{-1}_{N_{\mathbb{M}}}}(y)\Phi(0) \,dx \right) \, dy
\]
because the relation $x_{k,k+1}=x_{2n-k,2n+1-k}$ yields that $\hat{\psi}_{N}(\varrho((1_{2n}-x)^{-1}))=\psi(x'_{n,n+1})=\psi(x_{n,n+1})$.

\par
In addition, $\delta\epsilon_T\delta^{-1}={\rm diag}(2\cdot1_n,1_n,1_n,\frac{1}{2}\cdot1_n)$.
This element corresponds $\hat{\psi}_{N}$ to $\psi_{N}$, and stabilizes $\psi(x_{n,n+1})$, so by conjugating we obtain
\[
 \int_{Y(\mathbb{A})} \left(\int_{X(\mathbb{A})} \psi(x_{n,n+1}) \mathcal{W}^{\psi_{N}}(\tilde{\overline{\ell}}(x)\tilde{\delta}\tilde{\epsilon}_{\overline{U}}\tilde{\kappa}\tilde{y},\mathcal{E}_{-k}\tilde{\varphi}) \omega_{\psi^{-1}_{N_{\mathbb{M}}}}(y)\Phi(0) \, dx \right) \,dy
\]
Note that $\kappa^{-1} \epsilon_{\overline{U}}\kappa=\varrho(\begin{pmatrix*}[r] 1_n & -1_n \\ & 1_n \end{pmatrix*})$ and $\gamma=\delta\kappa$. Upon changing variables $y \mapsto (\kappa^{-1} \epsilon_{\overline{U}}\kappa)^{-1}y$, we get by \eqref{Weil-Rep-Action-one} and \eqref{Weil-Rep-Action-Extensions} that the above equals
\[
 \int_{Y(\mathbb{A})} \left( \int_{X} \psi(x_{n,n+1}) \mathcal{W}^{\psi_{N}}(,\tilde{\overline{\ell}}(x)\tilde{\gamma}\tilde{y},\mathcal{E}_{-k}\tilde{\varphi}) \omega_{\psi^{-1}_{N_{\mathbb{M}}}}(y)\Phi(\xi_n) \, dx \right) \, dy.
\]
We note that
\[
 \gamma^{-1} \overline{\ell}(X) \gamma=\{ \begin{pmatrix} 1_n & x_1 &\vspace{-.07cm}&x_2\\ &1_n&&\vspace{-.07cm}\\ \vspace{-.07cm}&&1_n& \check{x}_1 \\ &&&1_n \end{pmatrix} \,|\, x_1 \in {\rm Mat}_{n,n}
 \; \text{\rm strictly upper triangular}, x_2 \in \mathfrak{s}_n \},
\]
so that $ (\gamma^{-1} \overline{\ell}(X) \gamma Y) \rtimes V_{\gamma}=V$. In conclusion, we end up at
\begin{equation}
\label{symplectic-conclusion}
 \int_{V_{\gamma}(\mathbb{A}) \backslash V(\mathbb{A})} \mathcal{W}^{\psi_{N}} ( \tilde{\gamma}\tilde{v},\mathcal{E}_{-k} \tilde{\varphi}) \omega_{\psi^{-1}_{N_{\mathbb{M}}}}(v) \Phi(\xi_n) \,dx
\end{equation}
provided that it converges, since by \eqref{Weil-Rep-Action-three} and \eqref{Weil-Rep-Action-Extensions} $\omega_{\psi^{-1}_{N_{\mathbb{M}}}}(\gamma^{-1}\overline{\ell}(x)\gamma)$ acts by the scalar $\psi(x_{n,n+1})$
for $x \in X(\mathbb{A})$. Apparently \eqref{symplectic-decomp} and \eqref{symplectic-conclusion} are equal up to an appropriate sign $\mathfrak{e} \in \{ \pm 1 \}$. We can actually determine the precise constant $\mathfrak{e}=1$ by using the algorithm for computing the cocycle. But since this computation is extremely tedious, though not so deep, we do not include detailed calculation of $\mathfrak{e}$.
\end{proof}

\subsection{Local Rankin--Selberg--Shimura integrals}
We are now going back to our choice of $\psi_{N_{\mathbb{M}}}$ specified in \eqref{usual-symplectic-additive}.
Now let $F$ be either $p$-adic or archimedean, $\Pi \in {\rm Irr}_{\rm gen} \, \tilde{M}$ and $\sigma \in {\rm Irr}_{{\rm gen},\psi^{-1}_{N'}} \, G'$
with Whittaker model $\mathbb{W}^{\psi^{-1}_{N'}}(\sigma)$.
For any $W' \in \mathbb{W}^{\psi^{-1}_{N'}}(\sigma)$, $\tilde{W} \in {\rm Ind}(\mathbb{W}^{\psi_{N_M}}(\Pi))$, and $\Phi \in \mathcal{S}(F^n)$,
we define the local Rankin--Selberg--Shimura type integral
\[
 J(s,W',\tilde{W},\Phi):=\int_{N' \backslash G'} W'(g) A^{\psi}(s,g,\tilde{W},\Phi)\, dg.
\]
The analytic properties of this integral were studied by Ginzburg--Rallis--Soudry \cite{GRS98} and subsequently established by Kaplan \cite{Kap15}. 

\begin{proposition} Suppose that $\Pi \in {\rm Irr}_{{\rm gen}}\, \tilde{M}$ and $\sigma \in {\rm Irr}_{{\rm gen},\psi^{-1}_{N'}}\, G'$. 
Let $W' \in \mathbb{W}^{\psi^{-1}_{N'}}(\sigma)$, $\tilde{W} \in {\rm Ind}(\mathbb{W}^{\psi_{N_M}}(\Pi))$, and $\Phi \in \mathcal{S}(F^n)$. Then
\begin{enumerate}[label=$(\mathrm{\arabic*})$]
\item $J(s,W',\tilde{W},\Phi)$ converges in some right-half plane (depending only on $\pi$ and $\sigma$)
and admits a meromorphic continuation in $s$.
\item For any $s \in \mathbb{C}$, we can choose $W'$, $\tilde{W}$, and $\Phi$ such that $J(s,W',\tilde{W},\Phi) \neq 0$.
\item If $\Pi$, $\sigma$, and $\psi$ are unramified, $\tilde{W}^{\circ} \in {\rm Ind}(\mathbb{W}^{\psi_{N_M}}(\Pi))$ is the standard unramified vector,
$\Phi_0={\bf 1}_{\mathcal{O}^n}$, and ${W'}^{\circ}$ is $K'$-invariant with ${W'}^{\circ}(e)=1$, then
\[
 J(s,{W'}^{\circ},\tilde{W}^{\circ},\Phi_0)=\frac{{\rm vol}(K')}{{\rm vol}(N' \cap K')}\frac{L(s+\frac{1}{2},\sigma\times \pi)}{L(2s+1,\pi,{\rm Sym}^2)}
\]
assuming the Haar measure on $V$ and $V_{\gamma}$ are normalized so that ${\rm vol}(V \cap K)$ and ${\rm vol}(V_{\gamma} \cap K)$ 
are all $1$.
\item $J(s,W',\tilde{W},\Phi)$ satisfies local functional equations:
\[
 J(-s,W',M(s,\Pi)\tilde{W},\Phi)=\omega ^n_{\pi}(-1)\frac{\gamma(s+\frac{1}{2},\sigma\times \pi,\psi)}{C^{\widetilde{\rm Sp}_{2n}}_{\psi}(s+\frac{1}{2},\pi,w_U)}J(s,W',\tilde{W},\Phi)
 \]
where $C^{\widetilde{\rm Sp}_{2n}}_{\psi}(s+\frac{1}{2},\pi,w_U)$ is the Langlands-Shahidi local coefficient for the metaplectic group defined by Szpruch \cite{Szp13}.
 \end{enumerate}
\end{proposition}
It is worthwhile to mention that
$\displaystyle
{\rm vol}(K')=( \prod_{j=1}^n \zeta_F(2i) )^{-1}.
$

\subsection{Factorization of Fourier--Jacobi periods}
We define for $\tilde{\varphi} \in \mathcal{A}(\Pi)$ an automorphic form on $G(\mathbb{A})$ by
\[
 A^{\psi}(s,g,\tilde{\varphi},\Phi):=\int_{V_{\gamma}(\mathbb{A})\backslash V(\mathbb{A})} \mathcal{W}^{\psi_{N}}(\tilde{\gamma}\tilde{v}\tilde{\eta}(g),\tilde{\varphi}_s)
 \omega_{\psi^{-1}_{N_{\mathbb{M}}}}(v \tilde{g})\Phi(\xi_n) \, dv.
\]
The global descent $\mathcal{D}_{\psi}(\Pi)$ and the local descent $\mathcal{D}_{\psi}(\Pi_v)$ are related in the following way:
\begin{proposition} 
\label{FJ-Whittaker-Symplectic}
Let $\Phi=\otimes_v \Phi_v \in \mathcal{S}(\mathbb{A}^n)$ and $\tilde{\varphi} \in \mathcal{A}(\Pi)$ be factorizable vectors. 
Suppose that $\mathcal{W}^{\psi_{N_M}}(\cdot,\tilde{\varphi})=\prod_v \tilde{W}_v$ with $\tilde{W}_v \in {\rm Ind}(\mathbb{W}^{\psi_{N_M}}(\Pi_v))$.
Then for any sufficiently large finite set of places $S$, we have
\begin{equation}
\label{FJ-Whittaker-Symplectic-decomp}
 \mathcal{W}^{\psi_{N'}}(g,{\rm FJ}_{\psi_{N_{\mathbb{M}}}}(\mathcal{E}_{-k}\tilde{\varphi},\Phi))=
 \frac{ m^S_{-k,1}(\pi)}{2^k} \cdot \prod_{v \in S} A^{\psi} \left( -\frac{1}{2},g_v, M\left( \frac{1}{2}, \Pi_v \right)\tilde{W}_v,\Phi_v\right), 
\end{equation}
where
\[
 m^S_{-k,1}(\pi)=\frac{\lim_{s \rightarrow 1}(s-1)^kL^S(s,\pi,{\rm Sym}^2) }{L^S(2,\pi,{\rm Sym}^2)}.
\]
\end{proposition}

\begin{proof}
We note that
\[
 \mathcal{W}^{\psi_N}(\mathcal{E}_{-k}\tilde{\varphi})=\mathcal{W}^{\psi_{N_M}}(\mathcal{E}^U_{-k} \tilde{\varphi})
 =\mathcal{W}^{\psi_{N_M}}(M_{-k}(\tilde{\varphi})_{-\frac{1}{2}}).
\]
With this in hand, we rewrite Theorem \ref{Whittaker-main-symplectic} in terms of the global transforms $A^{\psi}$ as
\begin{equation}
\label{global-Afunction-symplectic}
  \mathcal{W}^{\psi_{N'}}(g,{\rm FJ}_{\psi_{N_{\mathbb{M}}}}(\mathcal{E}_{-k} \tilde{\varphi},\Phi))=A^{\psi} \left(-\frac{1}{2},g, M_{-k} (\tilde{\varphi}),\Phi \right).
\end{equation}
We can identify $\mathcal{A}(\Pi)$ with ${\rm Ind}^{\tilde{G}(\mathbb{A})}_{\tilde{P}(\mathbb{A})}\,\Pi=\otimes_v {\rm Ind}^{\tilde{G}(F_v)}_{\tilde{P}(F_v)}\,\Pi_v$. 
For ${\rm Re}(s) \gg 0$, we have
\begin{multline*}
  \mathcal{W}^{\psi_{N_M}}(\tilde{g}, M(s,\Pi)\tilde{\varphi})=\int_{N_M(F) \backslash N_M(\mathbb{A})} M(s,\Pi)\varphi(\tilde{n}\tilde{g}) \psi^{-1}_{N_M}(n)\,dn \\
  =\int_{N_M(F) \backslash N_M(\mathbb{A})} \int_{U(\mathbb{A})} \tilde{\varphi}(\tilde{\varrho}(\mathfrak{t})\tilde{w}_U\tilde{u}\tilde{n}\tilde{g})  \psi^{-1}_{N_M}(n) \, du \,dn
  =\int_{U(\mathbb{A})} \mathcal{W}^{\psi_{N_M}}(\tilde{\rho}(\mathfrak{t})\tilde{w}_U\tilde{u}\tilde{g},\tilde{\varphi}) \,du
\end{multline*}
from which it follows that for $S$ sufficiently large
\[
  \mathcal{W}^{\psi_{N_M}}(\tilde{g}, M(s,\Pi)\tilde{\varphi})=m^S(s,\pi)\left( \prod_{v \in S} M(s,\Pi_v)\tilde{W}_v \right) \prod_{v \notin S} \tilde{W}_v,
\]
as meromorphic functions in $s \in \mathbb{C}$. Here
\[
 m^S(s,\pi)=\frac{L^S(2s,\pi,{\rm Sym}^2)}{L^S(2s+1,\pi,{\rm Sym}^2)}.
\]
Taking the limit as $s \rightarrow \frac{1}{2}$, we get for $S$ large enough,
\begin{equation}
\label{limit-onehalf-symplectic}
  \mathcal{W}^{\psi_{N_M}}(M_{-k}\tilde{\varphi})=m^S_{-k}(\pi) \prod_{v \in S} M \left( \frac{1}{2},\Pi_v \right)\tilde{W}_v \prod_{v \notin S} \tilde{W}_v,
\end{equation}
where
\[
 m^S_{-k}(\pi)=\lim_{s \rightarrow \frac{1}{2}} \left( s-\frac{1}{2}\right)^k m^S(s,\pi)
 =2^{-k} \frac{\lim_{s \rightarrow 1}^kL^S(s,\pi,{\rm Sym}^2)}{L(2,\pi,{\rm Sym}^2)}=\frac{m^S_{-k,1}(\pi)}{2^k}.
\]
For any factorizable $\tilde{\varphi} \in \mathcal{A}(\Pi)$, the desired conclusion follows from \eqref{global-Afunction-symplectic} and \eqref{limit-onehalf-symplectic}
\end{proof}

We unfold the Petersson inner product of automorphic forms against the Fourier--Jacobi coefficient of Eisenstein series \cite[Theorem 10.4, (10.6)]{GRS11}
and then perform unramified computations \cite[(10.62)]{GRS11} to deduce the following Euler product.

\begin{proposition} 
Let $\Pi \in {\rm Scusp}_k \, \tilde{\mathbb{M}}$ and $\sigma=\mathcal{D}_{\psi^{-1}}(\Pi)$.
Let $\varphi' \in \sigma$ be an irreducible $\psi^{-1}_{N'}$-generic cuspidal representation of $G'$ and $\tilde{\varphi} \in \mathcal{A}(\Pi)$.
Suppose that $\mathcal{W}^{\psi^{-1}_{N'}}(\cdot,\varphi')=\prod_v W'_v$, $\mathcal{W}^{\psi_{N_M}}(\cdot,\tilde{\varphi})=\prod_v \tilde{W}_v$, $\Phi=\otimes_v \Phi_v$.
Then for any sufficiently large finite set of places $S$, we have
\begin{equation}
\label{Before-Limit-Symplectic}
  \langle \varphi', {\rm FJ}_{\psi_{N_{\mathbb{M}}}}(\mathcal{E}(s,\tilde{\varphi}),\Phi) \rangle_{G'}=\frac{{\rm vol}(N'(\mathcal{O}_S) \backslash N'(F_S))}{\prod_{j=1}^n\zeta^S_F(2j)} \cdot
   \frac{L^S(s+\frac{1}{2},\sigma \times \pi)}{L^S(2s+1,\pi, {\rm Sym}^2)} \prod_{v \in S} J(s,W'_v,\tilde{W}_v,\Phi_v),
\end{equation}
where on the right-hand side we take the unnormalized Tamagawa measure on $G'(F_S)$, $N'(F_S)$, $V(F_S)$ and $V_{\gamma}(F_S)$
(which are independent of the choice of gauge forms )
\end{proposition}

\begin{proof}
It is worthwhile to give our attention to the fact that the volume of $G'(F) \backslash G'(\mathbb{A})$, appearing on the right-hand side of \eqref{Petterson-Innerproduct},
is 1 equal to the Tamagawa number when we use the Tamagawa measure as in \cite{LM16,LM17}. 
We apply \cite[Theorem 10.4, (10.5)]{GRS11} in order to unfold
\begin{equation}
\label{symplectic-unfolding}
\langle \varphi', {\rm FJ}_{\psi_{N_{\mathbb{M}}}}(\mathcal{E}(s,\tilde{\varphi}),\Phi) \rangle_{G'}={\rm vol}(N'(F) \backslash N'(\mathbb{A}))\int_{N'(\mathbb{A})\backslash G'(\mathbb{A})} \mathcal{W}^{\psi^{-1}_{N'}}(g,\varphi')A^{\psi}(s,g,\tilde{\varphi},\Phi) \, dg.
\end{equation}
Similar to \cite[Propositions 5.7 and 8.4]{LM16}, our desired result is immediate from \eqref{symplectic-unfolding} in conjunction with the unramified computation \cite{GRS98} (cf. \cite[(10.62)]{GRS11}):
\begin{equation}
\label{symplectic-spherical-comp} 
 J(s,{W'_v}^{\circ},\tilde{W}_v^{\circ},\Phi_0)=\frac{{\rm vol}(K')}{{\rm vol}(N'(F_v)\cap K')} \frac{L(s+\frac{1}{2},\sigma_v\times \pi_v)}{L(2s+1,\pi_v,{\rm Sym}^2)}
\end{equation}
with $K'=K_{{\rm GL}_{2n}(F_v)}\cap G'(F_v)$.
Since we use the Tamagawa measure, the factor $\Delta^S_{{\rm Sp}_{2n}}(1)^{-1}$ pops up, where $\Delta^S_{{\rm Sp}_{2n}}(1)=L(M^{\vee}(s))$
is the partial $L$-function of the dual motive $M^{\vee}$ of the motive $M$ associated to $\Delta^S_{{\rm Sp}_{2n}}(1)$ by Gross \cite{Gro97}. Thanks to \cite[Notation and Convention]{Xue17},
\[
\Delta^S_{{\rm Sp}_{2n}}(1)=\prod_{j=1}^n \zeta^S_F(2i)
\]
for any finite set of places $S$.
\end{proof}

\subsection{Petersson inner products and Fourier coefficients}

Let $F$ denote a local field and let $\Pi$ be a genuine representation of $\tilde{\mathbb{M}}(F)$ of symplectic type.
We say that $\Pi$ is {\it good} if the following condition are satisfied for all $\psi$: 

\begin{enumerate}[label=$(\mathrm{\arabic*})$]
\item $\mathcal{D}_{\psi}(\Pi)$ is irreducible as a representation of $G'$.
\item $J(s,W',\tilde{W},\Phi)$ is holomorphic at $s=\frac{1}{2}$ for any $W' \in \mathcal{D}_{\psi^{-1}}(\Pi)$, $\tilde{W} \in {\rm Ind}^{\psi_{N_M}}(\Pi)$,
and $\Phi \in \mathcal{S}(F^n)$.
\item For any $W' \in \mathcal{D}_{\psi^{-1}}(\Pi)$, \\
$J(\frac{1}{2},W',\tilde{W},\Phi)$ factors through the map $\tilde{W} \mapsto (A^{\psi}(-\frac{1}{2}),\cdot, M(\frac{1}{2},\Pi)\tilde{W},\Phi)$.
\end{enumerate}
We know from \eqref{symplectic-G'-equivalent} that for $x \in G'$ and $v \in V$ we have
\[
 J(s,\sigma(x)W',I(s,\tilde{v}\tilde{\eta}(x))\tilde{W},\omega_{\psi^{-1}_{N_{\mathbb{M}}}}(v\tilde{x})\Phi)=J(s,W',\tilde{W},\Phi)
\]
where $\sigma=\mathcal{D}_{\psi^{-1}}(\Pi)$. Henceforth if $\Pi$ is good, there is a non-degenerate $G'$-invariant pairing $[\cdot,\cdot]$
on $\mathcal{D}_{\psi^{-1}}(\Pi) \times \mathcal{D}_{\psi}(\Pi)$ such that
\[
  J\left( \frac{1}{2}, W', \tilde{W},\Phi \right)=\left[ W',A^{\psi}\left(-\frac{1}{2},\cdot,M\left(\frac{1}{2},\Pi \right)\tilde{W}, \Phi \right) \right]
\]
for any $W' \in \mathcal{D}_{\psi^{-1}}(\Pi)$, $\tilde{W} \in {\rm Ind}(\mathcal{W}^{\psi_{N_M}}(\Pi))$, and $\Phi \in \mathcal{S}(F^n)$.
\begin{equation}
\label{cpi-equation-symplectic}
  \int^{\rm st}_{N'} J \left( \frac{1}{2},\sigma(u)W',\tilde{W},\Phi \right) \psi_{N'}(u) \, du=c_{\Pi}W'(e) A^{\psi} \left(-\frac{1}{2}, e, M \left(\frac{1}{2},\Pi \right)\tilde{W},\Phi \right).
\end{equation}

\begin{proposition}
\label{unramified-constant-symplectic}
Suppose that $F$ is $p$-adic with $p \neq 2$, $\Pi$ is of symplectic type and unramified and $\psi$ is unramified.
Then $\sigma=\mathcal{D}_{\psi^{-1}}(\Pi)$ is irreducible and unramifield. Let $\tilde{W}^{\circ} \in {\rm Ind}(\mathbb{W}^{\psi_{N_M}}(\Pi))$ be the standard
unramified $\tilde{K}$-vector and let ${W'}^{\circ}$ be $K'$-invariant with $W'(e)=1$. Then holds with $c_{\Pi}=1$.
\end{proposition}

\begin{proof}
The first part of the proof is similar to the one of \cite[Lemma 5.4]{LM17}, so we omit it. Assuming ${\rm vol}(U \cap K)=1$ as in \eqref{unramified-symplectic-intertwining}, we deduce
\[
  {\rm vol}(K')={\rm vol}(U \cap K)(\prod_{j=1}^n\zeta^S_F(2j))^{-1}
  =(\prod_{j=1}^n\zeta^S_F(2j))^{-1}.
\]
In virtue of it along with \cite[Proposition 2.14]{LM15}, we find that
\[
 \int^{\rm st}_{N'} J\left( \frac{1}{2},\sigma(u){W'}^{\circ},\tilde{W}^{\circ},\Phi_0 \right) \psi^{-1}_{N'}(u) \,du
 =\frac{{\rm vol}(N' \cap K')}{{\rm vol}(K')} \frac{J(\frac{1}{2},{W'}^{\circ},W^{\circ},\Phi_0)}{L(1,\sigma,{\rm Ad})}
\]
It is verified from the description of $\sigma$ in this case \cite[Theorem 6.4]{GRS11} together with \cite[\S 7]{GGP12} that
\[
L(1,\sigma,{\rm Ad})=L(1,\pi \boxplus \omega_{\pi},\wedge^2)=L(1,\pi,\wedge^2)L(1,\pi\times \omega_{\pi})
\]
and $L(s,\sigma\times\pi)=L(s,\pi\times\pi)L(s,\omega_{\pi}\times\pi)$. We recall that $\pi=\pi^{\vee}$ in our case.
Thus \eqref{unramified-symplectic-intertwining} aligned with of \eqref{symplectic-spherical-comp} implies that the above is equal to
\begin{multline*}
 \frac{1}{L(1,\pi,\wedge^2)L(1,\pi\times \omega_{\pi})}\frac{L(1,\pi \times \pi)L(1,\omega_{\pi}\times\pi)}{L(2,\pi,{\rm Sym}^2)}
 =\frac{L(1,\pi,{\rm Sym}^2)}{L(2,\pi,{\rm Sym}^2)}\\
 =M\left(\frac{1}{2},\Pi\right)\tilde{W}^{\circ}(e)={W'}^{\circ}(e) A^{\psi}\left(-\frac{1}{2},e,M\left( \frac{1}{2},\Pi \right)\tilde{W}^{\circ}, \Phi_0 \right)
\end{multline*}
that we look for.
\end{proof}

At first glance, $c_{\Pi}$ implicitly depends on the choice of Haar measure on $G'$ (the former exploited in the definition of $J(s,W',\tilde{W},\Phi)$)
and $U$ (the latter used in the definition of the intertwining operator), but not on any other groups (refer to \cite[Remark 5.2]{LM17} for further details).
However the vector spaces underlying the Lie algebras of $G'$ and $U$ are both canonically isomorphic as vectors spaces to
\[
 \{ X \in {\rm Mat}_{2n,2n}(F) \,|\, {X^t}w_{2n}=w_{2n}X \}.
\]
We can henceforth identify the gauge forms on $G'$ and $U$. With this in hand, $c_{\Pi}$ does not depend on any choice provided that
we use the unnormalized Tamagawa measure on $G'$ and $U$ with respect to the same gauge form.
We will say that the Haar measures on $G'$ and $U$ are {\it compatible} in this case.

\begin{proposition}
\label{symplectic-good-repn}
 Let $\Pi \in {\rm Scusp}_k\, \widetilde{\mathbb{M}}$ and $\sigma=\mathcal{D}_{\psi^{-1}}(\Pi)$. Assume our Working Hypothesis \ref{symplectic-working-hypothesis}. 
Let $\varphi' \in \sigma$ be an irreducible $\psi^{-1}_{N'}$-generic cuspidal automorphic representation of $G'$ and $\tilde{\varphi} \in \mathcal{A}(\Pi)$.
Suppose that $\mathcal{W}^{\psi^{-1}_{N'}}(\cdot,\varphi')=\prod_v W'_v$, $\mathcal{W}^{\psi_{N_M}}(\cdot,\tilde{\varphi})=\prod_v \tilde{W}_v$ , $\Phi=\otimes_v \Phi_v$.
Then for all $v$, $\Pi_v$ is good. Moreover for any sufficiently large finite set of places $S$, we have
\begin{equation}
\label{EulerProducts-Symplectic}
  \langle \varphi', {\rm FJ}_{\psi_{N_{\mathbb{M}}}}(\mathcal{E}_{-k} \tilde{\varphi},\Phi) \rangle_{G'}=
  \frac{{\rm vol}(N'(\mathcal{O}_S) \backslash N'(F_S))}{\prod_{j=1}^n\zeta^S_F(2j)}\cdot n^S_{-k,1} \prod_{v \in S} J\left(\frac{1}{2},W'_v,\tilde{W}_v,\Phi_v\right),
\end{equation}
where
\[
 n^S_{-k,1}=\frac{\lim_{s \rightarrow 1} L^S(s,\pi \times \pi)L^S(s,\pi \times \omega_{\pi})}{L^S(2,\pi,{\rm Sym}^2)}.
\]
\end{proposition}

\begin{proof}
We confirm that $\Pi_v$ is good for all $v$. To this end, we draw the conclusion from \eqref{FJ-Whittaker-Symplectic-decomp}
together with the irreducibility of the global descent that $\mathcal{D}_{\psi}(\Pi_v)$ is irreducible for all $v$.

\par
When $\Pi \in {\rm Scusp}_k \, \tilde{\mathbb{M}}$ and $\sigma \in \mathcal{D}_{\psi^{-1}}(\Pi)$, we obtain $L^S(s,\sigma\times\pi)=L^S(s,\pi\times\pi)L^S(s,\omega_{\pi}\times\pi)$,
which has a pole of order $k$ at $s=\frac{1}{2}$. On the one hand, $J(s,W'_v,\tilde{W}_v,\Phi_v)$ is non-vanishing for suitable $W'_v$, $\tilde{W}_v$, and $\Phi_v$,
from which it follows that the right-hand side of \eqref{Before-Limit-Symplectic} (when $\varphi' \in \sigma$) has a pole of order at most $k$ at $s=\frac{1}{2}$
for appropriate $\varphi'$, $\tilde{\varphi}$, and $\Phi$. On the other hand, the left-hand side of \eqref{Before-Limit-Symplectic} has a pole of order at most $k$
because this is true for $\mathcal{E}(s,\tilde{\varphi})$ and $\varphi'$ is rapidly decreasing. Multiplying \eqref{Before-Limit-Symplectic} by $(s-\frac{1}{2})^k$
and then taking the limit as $s \rightarrow \frac{1}{2}$, we conclude that $J(s,W'_v,\tilde{W}_v,\Phi_v)$ is holomorphic at $s=\frac{1}{2}$ for all $v$,
and for $\varphi' \in \sigma$, $\tilde{\varphi} \in \mathcal{A}(\Pi)$, and $\Phi \in \mathcal{S}(\mathbb{A}^n)$ we get
\begin{equation}
\label{After-Limit-Symplectic}
  \langle \varphi', {\rm FJ}_{\psi_{N_{\mathbb{M}}}}(\mathcal{E}_{-k} \tilde{\varphi},\Phi) \rangle_{G'}=
  \frac{{\rm vol}(N'(\mathcal{O}_S) \backslash N'(F_S))}{\prod_{j=1}^n\zeta^S_F(2j)}\cdot n^S_{-k,1} \prod_{v \in S} J\left(\frac{1}{2},W'_v,\tilde{W}_v,\Phi_v\right).
\end{equation}

\par
We fix a place $v_0$. We assume that $\tilde{W}_{v_0}$ and $\Phi_{v_0}$ are such that $A^{\psi} \left(-\frac{1}{2},\cdot,M\left( \frac{1}{2},\Pi_{v_0} \right)\tilde{W}_{v_0},\Phi_{v_0} \right) \equiv 0$.
Thereafter by \eqref{FJ-Whittaker-Symplectic-decomp}, $\tilde{\mathcal{W}}^{\psi_{\tilde{N}'}}(\cdot,{\rm FJ}_{\psi_{N_{\mathbb{M}}}}(\mathcal{E}_{-k} \tilde{\varphi},\Phi))\equiv 0$. With this in hand, the irreducibility and genericity of the descent ensure that ${\rm FJ}_{\psi_{N_{\mathbb{M}}}}(\mathcal{E}_{-k} \tilde{\varphi},\Phi) \equiv 0$. We deduce from \eqref{After-Limit-Symplectic} that $J\left(\frac{1}{2},W'_{v_0},\tilde{W}_{v_0},\Phi_{v_0}\right)=0$. Putting all together, $\Pi_{v_0}$ is good.
\end{proof}

While we expect that any irreducible generic genuine representation of symplectic type $\Pi_v$ is good, the proof that we give under the assumption that
$\Pi_v$ is the local component of $\Pi \in {\rm Scusp}_k\, \tilde{\mathbb{M}}$ is valid at the very least for the case of our interest which pertains to the global period integrals.
Our proof here requires a local to global argument. Whereas the global approach is overkill and indirect, one may directly prove the desired result by using a local functional
 equation \cite[Proposition 5.3]{LM17}. For the sake of brevity, we only include this indirect mean in keeping with the spirit of \cite[Propositions 5.7 and 8.4]{LM16}.

\begin{theorem}
\label{Main-Result-Symplectic-Theorem}
Let $\Pi \in {\rm Scusp}_k \, \tilde{\mathbb{M}}$ and $\sigma=\mathcal{D}_{\psi}(\Pi)$.
Assume our Working Hypothesis \ref{symplectic-working-hypothesis}.
Let $S$ be a finite set of places including all the archimedean places and even places such that $\Pi$ and $\psi$
are all unramified outside $S$. Then for any $\varphi \in \sigma$ and $\varphi^{\vee} \in \sigma^{\vee}$ which are fixed under $K'_v$
for all $v \notin S$, we have
\begin{multline}
\label{Main-Result-Symplectic}
\quad \mathcal{W}^{\psi_{N'}}(\varphi) \mathcal{W}^{\psi^{-1}_{N'}}(\varphi^{\vee})=2^{-k} \left( \prod_{v\in S} c^{-1}_{\Pi_v} \right) \frac{\prod_{j=1}^n\zeta^S_F(2j)}{L^S(1,\pi,\wedge^2)L^S(1,\pi\times \omega_{\pi})} \\
\times ( {\rm vol}(N'(\mathcal{O}_S) \backslash N'(F_S)))^{-1} \int^{\rm st}_{N'(F_S)} \langle \sigma(u)\varphi,\varphi^{\vee} \rangle_{G'} \psi^{-1} _{N'}(u)\,du.  \quad
\end{multline}
\end{theorem}

\begin{proof}
Since the descent of $\Pi$ is generated by ${\rm FJ}_{\psi_{N_{\mathbb{M}}}}(\mathcal{E}_{-k}\tilde{\varphi}^{\flat},\Phi)$ for $\tilde{\varphi}^{\flat} \in \mathcal{A}(\Pi)$
and $\Phi \in \mathcal{S}(\mathbb{A}^n)$, we initially take up this for the case $\varphi^{\vee}={\rm FJ}_{\psi_{N_{\mathbb{M}}}}(\mathcal{E}_{-k}\tilde{\varphi}^{\sharp},\Phi)$ with $\tilde{\varphi}^{\sharp} \in \mathcal{A}(\Pi)$, and this identity \eqref{Main-Result-Symplectic} then extends via linearity to all $\varphi^{\vee}$. In view of Proposition \ref{FJ-Whittaker-Symplectic}, we find that
\[
 \mathcal{W}^{\psi_{N'}}(\varphi) \mathcal{W}^{\psi^{-1}_{N'}}({\rm FJ}_{\psi_{N_{\mathbb{M}}}}(\mathcal{E}_{-k}\tilde{\varphi}^{\sharp},\Phi) )
 = \mathcal{W}^{\psi_{N'}}(\varphi)  \frac{m^S_{-k,1}(\pi)}{2^k}  \prod_{v \in S}  A^{\psi^{-1}} \left( -\frac{1}{2}, e, M \left( \frac{1}{2}, \Pi_v \right) \tilde{W}^{\sharp}_v, \Phi_v \right). 
\]
Counting on the fact that $\Pi_v$ is good in Proposition \ref{symplectic-good-repn}, we employ Proposition \ref{unramified-constant-symplectic} to insert the identity \eqref{cpi-equation-symplectic} 
for all $v$ with $c_{\Pi_v}=1$ for almost all $v$. In doing so, we write
\begin{multline}
\label{Symplectic-Product-Form-one}
 \mathcal{W}^{\psi_{N'}}(\varphi) \mathcal{W}^{\psi^{-1}_{N'}}({\rm FJ}_{\psi_{N_{\mathbb{M}}}}(\mathcal{E}_{-k}\tilde{\varphi}^{\sharp},\Phi) )\\
 =\frac{m^S_{-k,1}(\pi) }{2^k} \left( \prod_{v \in S} c^{-1}_{\Pi_v} 
 \int^{\rm st}_{N'(F_v)} J\left( \frac{1}{2}, \sigma_v(u_v)W_v, \tilde{W}_v^{\sharp}, \Phi_v \right) \psi^{-1}_{N'(F_v)}(u_v)\,du_v \right).
\end{multline}
After integrating \eqref{EulerProducts-Symplectic} over $N'(F_S)$ for $v \in S$, a product of regularized integrals becomes
\begin{multline}
\label{Symplectic-Product-Form-two}
 \prod_{v \in S} \int^{\rm st}_{N'(F_v)} J\left( \frac{1}{2}, \sigma_v(u_v)W_v, \tilde{W}_v^{\sharp}, \Phi_v \right) \psi^{-1}_{N'(F_v)}(u_v)\,du_v \\
 =\frac{\prod_{j=1}^n\zeta^S_F(2j)}{ {\rm vol}(N'(\mathcal{O}_S)\backslash N'(F_S)) \cdot n^S_{-k,1} } \int^{\rm st}_{N'(F_S)} \langle \sigma(u)\varphi, {\rm FJ}_{\psi_{N_{\mathbb{M}}}}(\mathcal{E}_{-k}\tilde{\varphi}^{\sharp},\Phi) \rangle_{G'} \psi^{-1}_{N'}(u) \, du.
\end{multline}
The factorization $L^S(s,\pi \times \pi)=L^S(s,\pi \times \pi^{\vee})=L^S(s,\pi,{\rm Sym}^2)L^S(s,\pi,\wedge^2)$ allows us to simplify the fraction
\begin{equation}
\label{Symplectic-Ratio}
m^S_{-k,1}/n^S_{-k,1}=1/(L^S(1,\pi,\wedge^2)L^S(1,\pi\times \omega_{\pi}))
\end{equation}
whose partial $L$-function occurring in the denominator is nothing but the adjoint $L$-function $L^S(s,\sigma,{\rm Ad})$ for $\sigma$. Combining \eqref{Symplectic-Product-Form-one} and \eqref{Symplectic-Product-Form-two} with \eqref{Symplectic-Ratio} leads us to the result that we seek for.
\end{proof}

\section{The Case for Odd Special Orthogonal Groups} 
\label{Sec:5}

\subsection{Groups, Embeddings, and Weyl Elements} Let $G={\rm SO}_{4n}$ be the even split special orthogonal group defined by
\[
G={\rm SO}_{4n}=\{  g \in {\rm GL}_{4n} \,|\,   \prescript{t}{}{g} w_{4n} g=w_{4n} \}.
\]
It acts on the $4n$-dimensional space $F^{4n}$ with the standard basis
$e_1,\dotsm, e_{2n},e_{-2n}, \dotsm, e_{-1}$. 
We denote by $\langle \cdot,\cdot \rangle$ the symmetric form defined on the space $F^{4n}$ by $\langle u,v\rangle=\prescript{t}{}{u} w_{4n} v$  for any $u, v \in F^{4n}$. 
We define $G'={\rm SO}_{2n+1} \subseteq G$ consisting of elements fixing $e_1,\dotsm,e_{n-1},e_{1-n},\dotsm, e_{-1}$ and $e_{2n}+e_{-2n}$.
(Note that we take $\alpha=-2$ in their notation \cite[(3.40),(9.8)]{GRS11}). Let $\mathcal{V}$ be the orthogonal complement of the space spanned by 
$e_1,\dotsm,e_{n-1},e_{1-n},\dotsm, e_{-1}$ and $e_{2n}+e_{-2n}$. We choose $\mathcal{B}_{G'}=\{ e_n,\dotsm, e_{2n-1}, e_{2n}-e_{-2n},e_{1-2n},\dotsm,e_{-n} \}$
as an ordered basis for $\mathcal{V}$. The group $G'$ is given by
\[
G'={\rm SO}_{2n+1}=\{  g \in {\rm GL}_{2n+1} \,|\,   \prescript{t}{}{g} \begin{pmatrix} && w_{n} \\ &-2& \\ w_{n}& & \end{pmatrix} g=\begin{pmatrix} && w_{n} \\ &-2& \\ w_{n}& & \end{pmatrix} \}.
\]
Let $\mathcal{B}'_{G'}=\{ e_n,\dotsm, e_{2n-1},e_{2n}+e_{-2n}, e_{2n}-e_{-2n},e_{1-2n},\dotsm,e_{-n} \}$ be the basis of $\mathcal{V} \oplus \langle e_{2n}+e_{-2n} \rangle$
obtained by adding $e_{2n}+e_{-2n}$ to $\mathcal{B}_{G'}$ as the $n+1^{th}$ vector.  In coordinates, the embedding $\eta : {\rm SO}_{2n+1} \rightarrow {\rm SO}_{4n} $ is given by
\[
\eta : \begin{pmatrix} A_1& A_2 & A_3 \\ A_4 & A_5 & A_6 \\ A_7 & A_8 & A_9 \\ \end{pmatrix} \mapsto {\rm diag}(1_{n-1},\mathcal{M} \begin{pmatrix} A_1&& A_2 & A_3 \\ &1&& \\ A_4 && A_5 & A_6 \\ A_7 && A_8 & A_9 \\ \end{pmatrix} \mathcal{M}^{-1},1_{n-1}),
\]
where $A_1$ and $A_9$ are $n \times n$ matrices, $A_5$ is a $1 \times 1$ matrix,  and
\[
  \mathcal{M}={\rm diag} (1_n, \begin{pmatrix} 1&\;\;\;1 \\  1&-1 \end{pmatrix} ,1_n)
\]
is the change of coordinate matrix from the standard basis to $\mathcal{B}'_{G'}$. Specifically, the image can be explicitly computed as follows,
\[
\eta ( \begin{pmatrix} A_1& A_2 & A_3 \\ A_4 & A_5 & A_6 \\ A_7 & A_8 & A_9 \\ \end{pmatrix})
={\rm diag}(1_{n-1},\begin{pmatrix}  \;\;\;A_1 & \frac{1}{2} A_2 &  -\frac{1}{2} A_2 & \;\;\; A_3 \\  \;\;\;A_4  &  1/2\cdot (1+A_5) &1/2 \cdot (1-A_5)  & \;\;\; A_6 \\
-A_4  & 1/2 \cdot (1-A_5) &1/2 \cdot (1+A_5)  & -A_6 \\  \;\;\;A_7 & \frac{1}{2} A_8 &  -\frac{1}{2} A_8 & \;\;\;  A_9 \\
 \end{pmatrix},1_{n-1}).
\]
For instance  \cite[(9.8)]{GRS11},
a standard unipotent subgroup $g$ in ${\rm SO}_{2n+1}$ and the image $\eta(g)$ in ${\rm SO}_{4n}$, with respect to the standard basis of $F^{4n}$, are
\[
 g=\begin{pmatrix} 1_n &2x&y \\ &1&\check{x}\\ && 1_n \end{pmatrix} \in {\rm SO}_{2n+1}, \quad \eta(g)={\rm diag}(1_{n-1}, \begin{pmatrix} 1_n &(x,-x)&y \\ & 1_2& \begin{pmatrix*}[r]  \check{x} \\ -\check{x}  \end{pmatrix*}  \\ && 1_n \end{pmatrix} ,1_{n-1})  \in {\rm SO}_{4n}.
\]
Let $P'=M'\ltimes U'$ be the Seigel parabolic subgroup of $G'$ with the Levi part $M'=\varrho'(\mathbb{M}')$ and the Seigel unipotent subgroup $U'$,
where $\varrho' : m \mapsto {\rm diag}(1_{n-1}, m, 1_2, m^{\ast}, 1_{n-1})$. Let us set 
\[
\mathfrak{a}_n=\{ x \in {\rm Mat}_{n,n} \,|\, \check{x}=-x \}.
\] 
Then $U'=U'_1 \ltimes U'_0$ with
\[
  U'_0=\{{\rm diag}(1_{n-1}, \begin{pmatrix} 1_n &&u \\ &1_2& \\ && 1_n \end{pmatrix} ,1_{n-1}) \,|\, u \in \mathfrak{a}_n \}
\]
and
\[
U'_1=\{ {\rm diag}(1_{n-1},\begin{pmatrix} 1_n &(v,-v)& \\ & 1_2& \begin{pmatrix*}[r]  \check{v} \\ -\check{v}  \end{pmatrix*}  \\ && 1_n \end{pmatrix}, 1_{n-1}) \,|\, v \in F^n \}.
\]
Let $N'$ (resp. $N$) be the standard maximal unipotent subgroup of $G'$ (resp. $G$). Then $N'=N'_{M'} \ltimes U'$ with $N'_{M'}=\varrho'(N'_{\mathbb{M}'})$.
Let $V$ be the unipotent radical in $G$ of the standard parabolic subgroup with Levi ${\rm GL}_1^{n-1} \times {\rm SO}_{2n+2}$:
\[
 V=\{ \begin{pmatrix} v & u & \;\;\;y \\ &1_{2n+2}& -\check{u} \\ && \;\;\; v^{\ast} \end{pmatrix} \in {\rm SO}_{4n} \,|\, v \in Z_n, u \in {\rm Mat}_{n,2n+2} \}.
\]
Let $P^{\diamond}=M^{\diamond} \ltimes U^{\diamond}$ (resp. $P=M \ltimes U$) be the standard Seigel parabolic subgroup of ${\rm SO}_{2n+2}$ (resp. $G$) with
the Levi part $M^{\diamond}$ (resp. $M$) isomorphic to ${\rm GL}_{n+1}$ (resp. ${\rm GL}_{2n}$).
We let $\mathbb{M}={\rm GL}_{2n}$ and use the isomorphism $\varrho(h)={\rm diag}(h,h^{\ast})$ to identify $\mathbb{M}$ with $M$. Let 
\[
 w'_{U^{\diamond}}={\rm diag}(1_{n-1}, \begin{pmatrix} & 1_{n+1} \\ 1_{n+1} & \end{pmatrix}, 1_{n-1})
\]
represent the longest $M^{\diamond}$-reduced Weyl element of ${\rm SO}_{2n+2}$. Let $w_U=\begin{pmatrix} & 1_{2n} \\ 1_{2n} & \end{pmatrix}$
represent the longest $M$-reduced Weyl element of $G$. We write
\[
 \gamma=w_U(w'_{U^{\diamond}})^{-1}=\begin{pmatrix} &1_{n+1}&&\vspace{-.1cm}\\ \vspace{-.1cm}&&&1_{n-1}\\ 1_{n-1}&&\vspace{-.1cm}&\\ &&1_{n+1}& \end{pmatrix}.
\]
Let us denote
\[
 V_{\gamma}=V \cap \gamma^{-1} N \gamma=\{ \begin{pmatrix} v &&u&\vspace{-.1cm}\\ &1_{n+1}&\vspace{-.1cm}&-\check{u} \\ &&1_{n+1}&\vspace{-.1cm}\\ &&&v^{\ast} \end{pmatrix}\,|\, v \in Z_n, u \in {\rm Mat}_{n,n+1} \}.
\]
Let $N_M$ (resp. $N_{\mathbb{M}}$) be the standard maximal unipotent subgroup of $M$ (resp. $\mathbb{M}$).
Let $K=K_{{\rm GL}_{4n}} \cap {\rm SO}_{4n}$. Then $K$ is a maximal compact subgroup of ${\rm SO}_{4n}$.

\subsection{Additive Characters} We fix a non-trivial additive character $\psi$ of $F$ and a non-degenerate 
character $\psi_{N_{\mathbb{M}}}$ of $N_{\mathbb{M}}$. Let $\psi_{\circ}$ be the non-trivial character of $F$ given by
\[
  \psi_{\circ}(x)=\psi_{N_{\mathbb{M}}}(1_{2n}+x\epsilon_{n,n+1})
\]
for $x \in F$. As discussed in \cite[Remark 6.4]{LM17} (cf. \cite[\S 4.1, \S 6.1]{LM16}), the statements in the sequel will not
depend on the choice of $\psi_{N_{\mathbb{M}}}$. For convenience, we set
\[
 \psi_{N_{\mathbb{M}}}(u)=\psi(u_{1,2}+\dotsm+u_{2n-1,2n}).
\]
Therefore $\psi_{\circ}=\psi$. We define the characters of various unipotent groups. We denote by $\psi_N$ the degenerate character on $N$ given by $\psi_N(uv)=\psi_{N_M}(u)$ for any $u \in N_M$
 and $v \in U$. 
 Let $\psi_{N_M}$ be the non-degenerate character of $N_M$ given by $\psi_{N_M}(\varrho(u))=\psi_{N_{\mathbb{M}}}(u)$.
 Let $\psi_{N'_{M'}}$ be the non-degenerate character of $N'_{M'}$ such that $\psi_{N'_{M'}}(\varrho(u))=\psi_{N'_{\mathbb{M}'}}(u)$.
 With the choice of $\psi_{N_{\mathbb{M}}}$, we have
 \[
  \psi_{N'_{\mathbb{M}'}}(u')=\psi(u'_{1,2}+\dotsm+u'_{n-1,n}).
 \]
 Let $\psi_{U'}$ be the character on $U'$ given by $\psi_{U'}(u)=\psi(u_{2n-1,2n})^{-1}$. Then the non-degenerate character $\psi_{N'}$
 is given by (cf. \cite[7.1]{LM16})
 \[ 
  \psi_{N'}(uv)=\psi_{N'_{M'}}(u)\psi_{U'}(v)=\psi(u_{n,n+1}+\dotsm+u_{2n-2,2n-1}-v_{2n-1,2n})
 \]
with $u \in N'_{M'}$ and $v \in U'$. Our choice is consistent with the characters considered in  \cite[(9.9)]{GRS11} as we have specified $\lambda=-1$ in their notation (cf. \cite[Theorem 9.4]{GRS11}). An element in $V$ can be written as $vu$ where $u$ fixes $e_1,\dotsm,e_n$, $v$ fixes $e_{n+1},e_{n+2},\dotsm,e_{-n-1}$ and 
we set \cite[(3.48)]{GRS11}
\[
 \psi_V(uv)=\psi^{-1}(v_{1,2}+\dotsm+v_{n-1,n})\psi^{-1}(u_{n,n+1}+u_{n,n+2}).
\]

\subsection{Whittaker Models and the Intertwining Operator}

Let $\pi$ be an irreducible generic representation of $\mathbb{M}$
 with its Whittaker model $\mathbb{W}^{\psi_{N_{\mathbb{M}}}}(\pi)$ with respect to the character $\psi_{N_{\mathbb{M}}}$.
 Similarly we use the notation $\mathbb{W}^{\psi^{-1}_{N_{\mathbb{M}}}}$, $\mathbb{W}^{\psi_{N_M}}$, $\mathbb{W}^{\psi^{-1}_{N_M}}$.

Given $g \in G$, we define $\nu(g)$ by $\nu(u \varrho(m)k)=|\det m|$ for any $u \in U$, $m \in \mathbb{M}$, and $k \in K$
via the Iwasawa decomposition. For any $f \in C^{\infty}(G)$ and $s \in \mathbb{C}$, define $f_s(g)=f(g)\nu(g)^s$, $g \in G$.
Let ${\rm Ind}(\mathbb{W}^{\psi_{N_M}}(\pi))$ be the space of $G$-smooth left $U$-invariant function $W : G \rightarrow \mathbb{C}$
such that for all $g \in G$, the function $m \mapsto \delta^{\frac{1}{2}}_P(m)W(mg)$ on $M$ belongs to $\mathbb{W}^{\psi_{N_M}}(\pi)$.
Similarly we define ${\rm Ind}(\mathbb{W}^{\psi^{-1}_{N_M}}(\pi))$, ${\rm Ind}(\mathbb{W}^{\psi_{N_{\mathbb{M}}}}(\pi))$, ${\rm Ind}(\mathbb{W}^{\psi^{-1}_{N_{\mathbb{M}}}}(\pi))$.
For any $s \in \mathbb{C}$, we have a representation $I(s,g)$ on the space ${\rm Ind}(s,\mathbb{W}^{\psi_{N_M}}(\pi))$
given by $(I(s,g)W)_s)(x)=W_s(xg)$, $x,g \in G$.

\par
We define the intertwining operator  \cite[\S 3.1]{Kap15}  
\[
M(s,\pi) : {\rm Ind}(s,\mathbb{W}^{\psi_{N_M}}(\pi)) \rightarrow {\rm Ind}(-s,\mathbb{W}^{\psi_{N_M}}(\pi^{\vee}))
\]
by (the analytic continuation of)
\[
 M(s,\pi)W(g)=\nu(g)^s \int_U W_s(\varrho(\mathfrak{t})w_Uug) \, du,
\]
where $\mathfrak{t} \in {\rm GL}_{2n}$ is introduced in order to preserve the character $\psi_{N_M}$.

\par
In the case where $F$ is a $p$-adic and $\pi$ and $\psi$ are unramified, and there exist (necessarily unique) $K$-fixed element
$W^{\circ} \in {\rm Ind}(\mathbb{W}^{\psi_{N_M}}(\pi))$ and $\check{W}^{\circ} \in {\rm Ind}(\mathbb{W}^{\psi_{N_M}}(\pi^{\vee}))$ such 
that $W^{\circ}(e)=\check{W}^{\circ}(e)=1$, then we have \cite[\S 2]{Sha88} (assuming ${\rm vol}(U \cup K)=1$)
\begin{equation}
\label{unramified-odd-orth-intertwining}
 M(s,\pi)W^{\circ}={\rm vol}(U \cap K) \frac{L(2s,\pi,\wedge^2)}{L(2s+1,\pi,\wedge^2)} \check{W}^{\circ}.
\end{equation}
The following result is an analogue of \cite[Proposition 2.1]{LM16} and \cite[Proposition 4.1]{LM17}.

\begin{proposition}
\label{holomorphy-odd-orthgonal}
Suppose that $\pi \in {\rm Irr}_{\rm gen} \, \mathbb{M}$ such that $\pi \cong \pi^{\vee}$.
Then $M(s,\pi)$ is holomorphic at $s=\frac{1}{2}$.
\end{proposition}

\begin{proof}
We write $\pi=\sigma_1 \boxplus \dotsm \boxplus \sigma_k$ where $\sigma_1,\dotsm,\sigma_k$ are essentially square-integrable.
By decomposing $M(s,\pi)$ into a product of co-rank one intertwining operators, it is suffices to show that the following intertwining operators
\begin{enumerate}[label=$(\mathrm{\roman*})$] 
\item \label{odd-orthognal-intertwining-differ} $\sigma_i |\cdot|^s \boxplus \sigma_j^{\vee}|\cdot|^{-s} \rightarrow \sigma_j^{\vee}|\cdot|^{-s} \boxplus \sigma_i|\cdot|^s, \quad i,j=1,2,\dotsm,k, i \neq j$,
\item \label{odd-orthognal-intertwining-same} $M(s,\sigma_i), \quad i=1,2,\dotsm,k$, 
\end{enumerate}
are holomorphic at $s=\frac{1}{2}$. The holomorphy of \ref{odd-orthognal-intertwining-differ} follows from \cite[Lemma 4.2]{LM17}
and the irreducibility of $\sigma_i \boxplus \sigma_j^{\vee}$. If $F=\mathbb{R}$ or $\mathbb{C}$, the holomorphy of \ref{odd-orthognal-intertwining-same}
is an immediate consequence of \cite[Lemma 4.3]{LM17} and the irreducibility of $\sigma_i \boxplus \sigma_i^{\vee}$. Under the identical condition of \cite[Lemma 4.3]{LM17},
Luo \cite[\S 2.4 Main Theorem 1]{Luo24} recently resolves the conjecture which turns out to be the stronger conclusion that $L(2s,\sigma_i,\wedge^2)^{-1}M(s,\sigma_i)$ is holomorphic for the $p$-adic case.
\end{proof}

\subsection{Period and Distinction}

\subsubsection{Global period}
Let $\mathfrak{S}$ be the Shalika subgroup
\[
 \mathfrak{S}=\left\{ \begin{pmatrix} g & \\ & g \end{pmatrix} \begin{pmatrix} 1_n & X \\ & 1_n \end{pmatrix} \,\middle|\, g \in \mathbb{M}',  X \in {\rm Mat}_{n,n} \right\}
\]
of $\mathbb{M}$ with the Shalika character $\psi_{\mathfrak{S}} \left( \begin{pmatrix} g & \\ & g \end{pmatrix} \begin{pmatrix} 1_n & X \\ & 1_n \end{pmatrix} \right)=\psi({\rm tr} \, X)$. Let $F$ be a number field. We say that $\pi \in {\rm Cusp}\, \mathbb{M}$ is of {\it orthogonal type} if
\[
 \int_{\mathfrak{S}(F) \backslash \mathfrak{S}(\mathbb{A})} \varphi(s)\psi^{-1}_{\mathfrak{S}}(s) \, ds \neq 0
\]
for some $\varphi$ in the space of $\pi$. The following characterization is due to Jacquet--Shalika \cite{}.

\begin{proposition}
Let $\pi \in {\rm Cusp}\, \mathbb{M}$. Then $\pi$ is of orthogonal type if and only if $L^S(s,\pi,\wedge^2)$ has a pole at $s=1$.
\end{proposition}

We let ${\rm Ocusp}_k \, \mathbb{M}$ be the set of automorphic representation $\pi$ of the form $\pi_1 \times \dotsm \pi_k$,
where $\pi_i \in {\rm Cusp}\, {\rm GL}_{2n_i}(F)$, $i=1,\dotsm,k$ are pairwise inequivalent (that is, $\pi_i \not \cong \pi_j$ if $i \neq j$)
and of orthogonal type with $n_1+\dotsm+n_k=n$ and $n_i > 1$ for all $i$. In particular $\pi_1^{\vee} \cong \pi_i$
and the central character $\omega_{\pi}$ of $\pi$ is trivial on $\mathbb{A}^{\times}$. We let ${\rm Ocusp}\, \mathbb{M}=\cup_k {\rm Ocusp}_k\, \mathbb{M}$.

\subsubsection{Local distinction}

Let $F$ be a local field. We say that $\pi \in {\rm Irr} \, \mathbb{M}$ is of {\it orthogonal type} if it has a non-trivial $(\mathfrak{S}(F),\psi_{\mathfrak{S}(F)})$-invariant linear form $\lambda$.
In particular, the central character $\omega_{\pi}$ of $\pi$ is trivial on $F^{\times}$. We write ${\rm Irr}_{\rm orth}\, \mathbb{M}$ for the set of irreducible representations of orthogonal type.
Clearly, if $\pi$ is of orthogonal type in the global setting, then all its local components $\pi_v$ are irreducible and of orthogonal type as well.

\begin{lemma}
 \begin{enumerate}[label=$(\mathrm{\arabic*})$]
\item (\cite{JR96}--$p$-adic case; \cite{AGJ09}--archimedean case) Suppose that $\pi \in {\rm Irr}_{\rm orth} \, {\rm GL}_{2n}(F)$. Then the space of $(\mathfrak{S}(F),\psi_{\mathfrak{S}(F)})$-invariant linear forms on $\pi$ is one-dimensional. Moreover $\pi^{\vee} \cong \pi$.
\item (\cite[Proposition 3.12]{Jo20}--$p$-adic case) Suppose that $\pi \in {\rm Irr} \, {\rm GL}_{2n}(F)$ is square integrable and $F$ is p-adic.
Then $\pi \in {\rm Irr}_{\rm orth}(F)$ if and only if $L(s,\pi,\wedge^2)$ has a pole at $s=0$.
\item (cf. \cite[Theorem 1.1]{Mat17}--$p$-adic case) Suppose that $\pi_i \in {\rm Irr}_{\rm orth} \, {\rm GL}_{2n_i}(F)$, $i=1,2$ and the parabolic induction $\pi_1 \boxplus \pi_2$
is irreducible. Then $\pi_1 \boxplus \pi_2 \in {\rm Irr}_{\rm orth} \, {\rm GL}_{2n_1+2n_2}(F)$. 
 \end{enumerate}
\end{lemma}

\begin{proof}
Let $P^{\circ}=M^{\circ} \ltimes U^{\circ}$ be the parabolic subgroup of type $(2n_1,2n_2)$ of $\mathbb{M}$ so that $\pi:=\pi_1 \boxplus \pi_2$ is the parabolic induction of $\pi_1 \otimes \pi_2$.
Let $\beta_i$ be a non-trivial $(\mathfrak{S}(F),\psi_{\mathfrak{S}(F)})$-invariant linear functionals on $\pi_i$ for $i=1,2$. Let $\beta$ be the functional $\beta_1 \otimes \beta_2$
on $\pi_1 \otimes \pi_2$. We define a linear form $\lambda$ on $\pi$ by
\[
  \lambda(\phi)=\int_{{\rm Mat}_{n_2 \times n_1}} \beta ( \phi ( \begin{pmatrix} 1_{n_1} &\vspace{-.5mm}&& \\  &1_{n_1}&\vspace{-.5mm}&\\ &x&1_{n_2}&\vspace{-.5mm}\\ &&&1_{n_2} \end{pmatrix} ) ) \, dx.  
\]
Then $\lambda$ is a non-zero functionals on $\pi$ which is $\psi_{\mathfrak{S}}(F)$-invariant under $\mathfrak{S}(F)$ (cf. \cite[Lemma 3.1]{Mat17}). Furthermore the integral converges absolutely. The absolute convergence is asserted for archimedean cases in \cite[Lemma 3.1]{FJ93} and is settled for $p$-adic case in \cite[Lemma 3.2]{Mat17}.
\end{proof}

On the purpose of completeness, we recall the following classification theorem, due to Matringe, of the set ${\rm Irr}_{\rm gen, orth}\, {\rm GL}_n(F)$
of generic representations of orthogonal types.

\begin{proposition} (\cite[Corollary 1.1]{Mat17}) Assume that $F$ is a $p$-adic field. Then the set ${\rm Irr}_{\rm gen,orth} \, {\rm GL}_{2n}(F)$ consists of 
the irreducible representations of the form
\[
 \pi=(\sigma_1 \boxplus \sigma^{\vee}_1) \boxplus \dotsm \boxplus (\sigma_k \boxplus \sigma_k^{\vee}) \boxplus \tau_1 \boxplus \dotsm \boxplus \tau_{\ell}
\]
where $\sigma_1, \dotsm,\sigma_k$ are essentially square-integrable and $\tau_1,\dotsm,\tau_{\ell}$ are square-integrable of orthogonal type (i.e., $L(0,\tau_i,\wedge^2)=\infty$
for all $i=1,\dotsm,\ell$).
\end{proposition}

\subsection{Descent maps}

\subsubsection{Global descent}

We now turn our sight to the global setting. For any automorphic form $\varphi$ on $G(\mathbb{A})$,
let ${\rm GG}(\varphi)$ be the {\it Gelfand--Graev coefficient} (a function on $G'(F) \backslash G'(\mathbb{A})$)
\[
  {\rm GG}(\varphi)(g)=\int_{V(F)\backslash V(\mathbb{A})} \varphi(vg)\psi^{-1}_V(v) \, dv, \quad g \in G'(\mathbb{A}).
\]

\par
We study the automorphic representation $\pi$ of $\mathbb{M}(\mathbb{A})$ in ${\rm Ocusp}\, \mathbb{M}$ realized on 
the space of Eisenstein series induced from $\pi_1 \otimes \dotsm \otimes \pi_k$.
We view $\pi$ as a representation of $M(\mathbb{A})$ via $\varrho$. Let $\mathcal{A}(\pi)$
be the space of functions $\varphi : M(F)U(\mathbb{A}) \backslash G(\mathbb{A}) \rightarrow \mathbb{C}$ such that $m \mapsto \delta^{\frac{1}{2}}_P(m)\varphi(mg)$, $m \in M(\mathbb{A})$ belongs  to the space of $\pi$ for all $g \in G(\mathbb{A})$. Let $\varphi_s(g)=\nu(g)^s\varphi(g)$. We define Eisenstein series
\[
\mathcal{E}(s,\varphi)=\sum_{\gamma \in P(F) \backslash G(F)} \varphi_s(\gamma g), \quad, \varphi \in \mathcal{A}(\pi).
\]
The space of $\pi$ is preserved under the conjugation by $w_U$. Then we define intertwining operator
$M(s,\pi) : \mathcal{A}(\pi) \rightarrow \mathcal{A}(\pi)$ given by
\[
 M(s,\pi) \varphi(g)=\nu(g)^s \int_{U(\mathbb{A})} \varphi_s(w_Uug)\,du.
\]
Both $\mathcal{E}(s,\varphi)$ and $M(s,\pi)$ converge absolutely for ${\rm Re}(s) \gg 0$ and extend meromorphically to $\mathbb{C}$.
It follows from \cite[Theorem 2.1]{GRS11} that both $\mathcal{E}(s,\varphi)$ and $M(s,\pi)$ have a pole of order $k$ at $s=\frac{1}{2}$.
Let
\[
 \mathcal{E}_{-k}\varphi=\lim_{s \rightarrow \frac{1}{2}} \left(s-\frac{1}{2} \right)^k \mathcal{E}(s,\varphi)
\]
and
\[
 M_{-k}=\lim_{s \rightarrow \frac{1}{2}} \left( s-\frac{1}{2}\right)^k M(s,\pi)
\]
be the leading coefficient in the Laurent series expansion at $s=\frac{1}{2}$.
The constant term
\[
 \mathcal{E}^U_{-k} \varphi(g)=\int_{U(F) \backslash U(\mathbb{A})} \mathcal{E}_{-k}\varphi(ug) \,du
\]
of $\mathcal{E}_{-k}\varphi$ along $U$ is given by
\[
 \mathcal{E}^U_{-k}\varphi(g)=(M_{-k}(\varphi)(g))_{-\frac{1}{2}}.
\]
The {\it descent} of $\pi$ (with respect to $\psi_{N_{\mathbb{M}}}$) is the space $\mathcal{D}_{\psi}(\pi)$ spanned by ${\rm GG}(\mathcal{E}_{-k}\varphi)$
with $\varphi \in \mathcal{A}(\pi)$. It is known from \cite[Theorem 3.1]{GRS11} that $\mathcal{D}_{\psi}(\pi)$ is multiplicity free.
By \cite[Theorem 9.4]{GRS11}, $\mathcal{D}_{\psi}(\pi)$ is a non-trivial cuspidal (and globally generic) automorphic representation of $G'$.
We expect it to be irreducible.
This result will be the analogue of \cite[Theorem 2.3]{GJS12} in metaplectic group case and \cite[Theorem 1.1]{Mor18} in even unitary group case. It is likely that the method of \cite[Theorem 2.3]{GJS12} and \cite[Theorem 1.1]{Mor18} is extended to the case at hand.
Since this is beyond the scope of the current paper we simply make it a working hypothesis. 

\begin{working hypothesis}
\label{odd-orthogonal-working-hypothesis}
When $\pi \in {\rm Ocusp}\, \mathbb{M}$, $\mathcal{D}_{\psi}(\pi)$ is irreducible. 
\end{working hypothesis}

The image of the weak lift is then a consequence of \cite[Theorem 11.2]{GRS11}. In particular, \cite[Theorem 9.4 (3)]{GRS11} describes 
the image for unramified places.

\begin{proposition} With the above Working Hypothesis \ref{odd-orthogonal-working-hypothesis}, $\pi \mapsto \sigma=\mathcal{D}_{\psi}(\pi)$
defines a bijection between ${\rm Ocusp} \, \mathbb{M}$ and ${\rm Cusp}_{\psi_{N'}}\,G'$, the set of $\psi_{N'}$-generic cuspidal automorphic representation
of $G'$. Furthermore $\sigma^{\vee}$ weekly lifts to $\pi$.
\end{proposition}

\subsubsection{Local descent} We turn back to the local setup.
Suppose now that $F$ is local fields. Let $\pi \in {\rm Irr}_{\rm gen} \, M$.
For any $W \in {\rm Ind}(\mathbb{W}^{\psi_{N_M}}(\pi))$ and $s \in \mathbb{C}$, we define a function on $G'$:
\[
  A^{\psi}(s,g,W)=\int_{V_{\gamma} \backslash V} W_s(\gamma v \mathfrak{c}^{n+1} g) \psi^{-1}_V(v) \, dv, \quad g \in G'.
\]
where
\[
 \mathfrak{c}={\rm diag}(1_{n-1}, -1_n,  \begin{pmatrix} 0 & 1 \\ 1 & 0 \end{pmatrix},-1_n, 1_{n-1}).
\]
Let the sign $\pm$ denote $(-1)^{n+1}$. The (signed) outer automorphism $ \mathfrak{c}$ does not lie in ${\rm SO}_{4n}$. However $\mathfrak{c}{\rm SO}_{4n}\mathfrak{c}^{-1}
=\mathfrak{c}{\rm SO}_{4n}\mathfrak{c}
={\rm SO}_{4n}$. At first glance the presence of (signed) outer automorphism $ \mathfrak{c}$ is implicit in integrals of Ginzburg \cite{Gin90} and Soudry \cite{Sou93}, but
it can be reconciled with $\beta_{n,2n}$ defined in \cite[p. 397]{Kap15} by
\[
 \beta_{n,2n} \, \varrho ( \begin{pmatrix} &1_{n+1} \\ 1_{n-1} & \end{pmatrix} )
 =\gamma \mathfrak{c}^{n+1} \,{\rm diag}(1_{n-1}, \pm1_n,  1_2,\pm 1_n, 1_{n-1}),
\]
while the Weyl element $\varrho ( \begin{pmatrix} &1_{n+1} \\ 1_{n-1} & \end{pmatrix} )$ accounts the discrepancy of our embedding $\eta$ with $i_{n,2n}$ in \cite[\S 3.1]{Kap15}.
The basic properties of $A^{\psi}$ are summarized in the following proposition. Its proof is identical to \cite[Lemmata 4.5 and 4.9]{LM16} and \cite[Lemmata 5.1 and 8.1]{LM17}, henceforth will be omitted.

\begin{proposition} Suppose that $\pi \in {\rm Irr}_{\rm gen} \, M$. Let $W \in {\rm Ind}(\mathbb{W}^{\psi_{N_M}}(\pi))$ and $s \in \mathbb{C}$. Then
\begin{enumerate}[label=$(\mathrm{\arabic*})$]
\item The integral $A^{\psi}(s,g,W)$ is well-defined and absolutely convergent uniformly for $s \in \mathbb{C}$.
Thus $A^{\psi}(s,g,W)$ is entire in $s$ and smooth in $g$. In the non-archimedean case the integrand is compactly supported on $V_{\gamma} \backslash V$.
\item For any $W \in {\rm Ind}(\mathbb{W}^{\psi_{N_M}}(\pi))$ and $s \in \mathbb{C}$, the function $g \mapsto A^{\psi}(s,g,W)$
is smooth and $(N',\psi_{N'})$-equivariant.
\item For an $g, x \in G'$, we have
\begin{equation}
\label{odd-orthogonal-G'-equivalent} 
  A^{\psi}(s,g,I(s,x)W)=A^{\psi}(s,gx,W).
\end{equation}
\item Suppose that $\pi$ is unramified, $\psi$ is unramified, and $W^{\circ} \in {\rm Ind}(\mathbb{W}^{\psi_{N_M}}(\pi))$ is the standard unramified 
vector. Then $A^{\psi}(s,e,W^{\circ}) \equiv 1$ (assuming ${\rm vol}(V \cap K)={\rm vol}(V_{\gamma}\cap K)=1$).
\end{enumerate}
\end{proposition}

Let $\pi \in {\rm Irr}_{\rm gen} \, \mathbb{M}$, regarded also as a representation of $M$ via $\varrho$. By the same argument as in \cite[Theorem in \S 1.3]{GRS99}
together with \cite[Remark 4.3]{LM17}, for any non-zero subrepresentation $\pi'$ of ${\rm Ind}(\mathbb{W}^{\psi_{N_M}}(\pi))$
there exists $W' \in \pi'$ such that $A^{\psi}(s,\cdot,W') \not \equiv 0$. Upon twisting $\pi$, we conclude that for any given $s_0 \in \mathbb{C}$,
 there exists $W' \in \pi'$ such that $A^{\psi}(s_0,\cdot,W') \not \equiv 0$.
 
 \par
 Assume now that $\pi \in {\rm Irr}_{\rm gen,orth} \, \mathbb{M}$. By Proposition \ref{holomorphy-odd-orthgonal}, $M(s,\pi)$ is holomorphic at $s=\frac{1}{2}$. The {\it local descent} of $\pi$ is the space $\mathcal{D}_{\psi}(\pi)$
 of Whittaker functions on $G'$ generated by $A^{\psi}\left(-\frac{1}{2},\cdot,M\left( \frac{1}{2},\pi \right)W \right)$
 with $W \in {\rm Ind}(\mathbb{W}^{\psi_{N_M}}(\pi))$. By the above observation, $\mathcal{D}_{\psi}(\pi) \neq 0$.
 
 \par
 Let $\pi'$ be the image of ${\rm Ind}(\mathbb{W}^{\psi_{N_M}}(\pi))$ under $M \left( \frac{1}{2}, \pi \right)$. 
 In virtue of \eqref{odd-orthogonal-G'-equivalent}, the space $\mathcal{D}_{\psi}(\pi)$ is canonically
 a quotient of the $G'$-module $J_V(\pi')$ of the $V$-coinvariant of $\pi'$. In this regard, we view
 $J_V(\pi')$ as the {\it abstract descent}, whereas $\mathcal{D}_{\psi}(\pi)$ as the {\it explicit descent}.

\subsubsection{Whittaker function of descent}
We set \cite[(8.14)]{GRS11}
\[
 \psi_{N_{\mathbb{M}}}(u)=\psi(u_{1,2}+\dotsm+u_{n,n+1}-u_{n+1,n+2}-\dotsm-u_{2n-1,2n}).
\]
The heart of the proof lies in unfolding integrals which appears in the forthcoming work of Lapid and Mao \cite[\S 4.3, Third paragraph]{LM22}.

\begin{theorem}
\label{Whittaker-main-odd-orthogonal}
$(${\rm Lapid and Mao}$)$ Let $\pi \in {\rm Ocusp}\, \mathbb{M}$ and $\varphi \in \mathcal{A}(\pi)$. 
Then for any $g \in G'$ we have
\[
 \mathcal{W}^{\psi_{N'}}(g,{\rm GG}(\mathcal{E}_{-k} \varphi))=\int_{V_{\gamma}(\mathbb{A}) \backslash V(\mathbb{A})} \mathcal{W}^{\psi_N}(\gamma v \mathfrak{c}^{n+1}g,\mathcal{E}_{-k} \varphi)\psi^{-1}_V(v) \, dv,
\]
and the integral is absolutely convergent.
\end{theorem}

\subsection{Local Rankin--Selberg--Shimura integrals}
Let $F$ be either $p$-adic or archimedean, $\pi \in {\rm Irr}_{\rm gen}\, M$ and $\sigma \in {\rm Irr}_{{\rm gen},\psi^{-1}_{N'}}\, G'$
with Whittaker model $\mathbb{W}^{\psi^{-1}_{N'}}(\sigma)$.
For any $W' \in \mathbb{W}^{\psi^{-1}_{N'}}(\sigma)$ and $W \in {\rm Ind}(\mathbb{W}^{\psi_{N_M}})(\pi)$,
we define the local Rankin--Selberg--Shimura type integral
\[
  J(s,W',W):=\int_{N' \backslash G'} W'(g)A^{\psi}(s,g,W) \, dg.
\]
The analytic properties of the integral were initiated by Ginzburg \cite{Gin90}, extended by Soudary \cite{Sou93}, and subsequently concluded by Kaplan \cite{Kap15}.

\begin{proposition} Suppose that $\pi \in {\rm Irr}_{\rm gen}\, M$ and $\sigma \in {\rm Irr}_{{\rm gen},\psi^{-1}_{N'}}\, G'$.
Let $W' \in \mathbb{W}^{\psi^{-1}_{N'}}(\sigma)$ and $W \in {\rm Ind}(\mathbb{W}^{\psi_{N_M}}(\pi))$. Then
\begin{enumerate}[label=$(\mathrm{\arabic*})$]
\item $J(s,W',W)$ converges in some right-half plane (depending only on $\pi$ and $\sigma$)
and admits a meromorphic continuation in $s$.
\item For any $s \in \mathbb{C}$, we can choose $W'$ and $W$ such that $J(s,W',W) \neq 0$.
\item  If $\pi$, $\sigma$ and $\psi$ are unramified, $W^{\circ} \in {\rm Ind}(\mathbb{W}^{\psi_{N_M}}(\pi))$ is the standard unramified vector
and ${W'}^{\circ}$ is $K'$-invariant with ${W'}^{\circ}(e)=1$, then
\[
  J(s,{W'}^{\circ},W^{\circ})=\frac{{\rm vol}(K')}{{\rm vol}(N' \cap K')} \frac{L(s+\frac{1}{2},\sigma\times\pi)}{L(2s+1,\pi,\wedge^2)}
\]
assuming the Haar measure on $V$ and $V_{\gamma}$ are normalized so that ${\rm vol}(V \cap K)$
and ${\rm vol}(V_{\gamma} \cap K)$ are all $1$.
\item $J(s,W',W)$ satisfies local functional equations:
\[
 J(-s,W',M(s,\pi)W)=\omega^n_{\pi}(-1) \frac{\gamma(s+\frac{1}{2},\sigma\times\pi,\psi)}{\gamma(2s+1,\pi,\wedge^2,\psi)} J(s,W',W)
\]
where $\gamma(s,\pi,\wedge^2,\psi)$ is the gamma factor defined by Shahidi (pertaining to the exterior square representation).
\end{enumerate}
\end{proposition}
It is noteworthy that
$\displaystyle
{\rm vol}(K')=( \prod_{j=1}^n \zeta_F(2j) )^{-1}.
$

\subsection{Factorization of Gelfand-Graev periods} 
Define $\varphi \in \mathcal{A}(\pi)$
an automorphic form on $G(\mathbb{A})$;
\[
 A^{\psi}(s,g,\varphi):=\int_{V_{\gamma}(\mathbb{A})\backslash V(\mathbb{A})} \mathcal{W}^{\psi_{N_M}}(\varphi_s, \gamma v \mathfrak{c}^{n+1} g) \psi^{-1}_V(g) \, dv.
\]
The global descent $\mathcal{D}_{\psi}(\Pi)$ and the local descent $\mathcal{D}_{\psi}(\Pi_v)$ are related in the following way;
\begin{proposition} 
Suppose that $\varphi \in \mathcal{A}(\pi)$ is factorizable vector and
 $\mathcal{W}^{\psi_{N_M}}(\varphi,\cdot)=\prod_v W_v$ with $W_v \in {\rm Ind}(\mathbb{W}^{\psi_{N_M}}(\pi_v))$.
Then for any sufficiently large finite set of places $S$, we have
\begin{equation}
\label{GG-Whittaker-OddOrthogonal-decomp}
 \mathcal{W}^{\psi_{N'}}(g,{\rm GG}(\mathcal{E}_{-k} \varphi))=
 \frac{m^S_{-k,1}(\pi)}{2^k}  \prod_{v \in S} A^{\psi} \left( -\frac{1}{2}, g_v, M\left( \frac{1}{2}, \pi_v \right)W_v \right), 
\end{equation}
where
\[
 m^S_{-k,1}(\pi)=\frac{\lim_{s \rightarrow 1}(s-1)^kL^S(s,\pi,\wedge^2)}{L^S(2,\pi,\wedge^2)}.
\]
\end{proposition}

\begin{proof}
We notice that
\[
 \mathcal{W}^{\psi_N}(\mathcal{E}_{-k}\varphi)=\mathcal{W}^{\psi_{N_M}}(\mathcal{E}^U_{-k}\varphi)=\mathcal{W}^{\psi_{N_M}}((M_{-k}(\varphi))_{-\frac{1}{2}}).
\]
Keeping this in mind, Theorem \ref{Whittaker-main-odd-orthogonal} can be expressed in terms of the global transform $A^{\psi}$ as
\begin{equation}
\label{global-Afunction-odd-orthogoal}
 \mathcal{W}^{\psi_{N'}}(g, {\rm GG}(\mathcal{E}_{-k}\varphi))=A^{\psi} \left( -\frac{1}{2},g, M_{-k}(\varphi) \right).
\end{equation}
We can identify $\mathcal{A}(\pi)$ with ${\rm Ind}^{G(\mathbb{A})}_{P(\mathbb{A})} \, \pi=\otimes_v {\rm Ind}^{G(F_v)}_{P(F_v)} \, \pi_v$. 
For $\rm Re(s) \gg 0$, we have
\begin{multline*}
 \mathcal{W}^{\psi_{N_M}}(g,M(s,\pi)\varphi)=\int_{N_M(F)\backslash N_M(\mathbb{A})} M(s,\pi)\varphi(ng)\psi^{-1}_{N_M}(n)\,dn\\
 =\int_{N_M(F) \backslash N_M(\mathbb{A})} \int_{U(\mathbb{A})} \varphi(\varrho(\mathfrak{t}w_Uung)) \psi^{-1}_{N_M}(n)\, du \, dn
 =\int_{U(\mathbb{A})}  \mathcal{W}^{\psi_{N_M}}(\varrho(\mathfrak{t})w_Uug,\varphi)\, du.
\end{multline*}
 For $S$ sufficiently large, we are led to
\[
 \mathcal{W}^{\psi_{N_M}}(M(s,\pi)\varphi)=m^S(s,\pi) \left(\prod_{v \in S} M(s,\pi_v)W_v \right) \prod_{v \notin S} W_v
\]
as meromorphic functions in $s \in \mathbb{C}$ where
\[
 m^S(s,\pi)=\frac{L^S(2s,\pi,\wedge^2)}{L^S(2s+1,\pi,\wedge^2)}.
\]
Taking the limit as $s \rightarrow \frac{1}{2}$, we get for $S$ large enough,
\begin{equation}
\label{limit-onehalf-odd-orthogoal}
 \mathcal{W}^{\psi_{N_M}}(M_{-k}\varphi)=m^S_{-k}(\pi)\prod_{v \in S} M \left( \frac{1}{2},\pi_v \right)W_v \prod_{v \notin S} W_v
\end{equation}
for which
\[
m^S_{-k}(\pi)=\lim_{s \rightarrow \frac{1}{2}} \left( s-\frac{1}{2}\right)^km^S(s,\pi)
=2^{-k}\frac{\lim_{s \rightarrow 1}(s-1)^kL^S(s,\pi,\wedge^2)}{L(2,\pi,\wedge^2)}
=\frac{m^S_{-k,1}(\pi)}{2^k}.
\]
For any factorizable $\varphi \in \mathcal{A}(\pi)$, our desired result is immediate from \eqref{global-Afunction-odd-orthogoal} and \eqref{limit-onehalf-odd-orthogoal}.
\end{proof}

We unfold Petersson inner products of automorphic forms against the Gelfand--Graev coefficient of Eisenstein series \cite[Theorem 10.3, (10.4)]{GRS11}
and then perform unramified computations \cite[(10.59)]{GRS11} to obtain the following Euler product.

\begin{proposition}  Let $\pi \in {\rm Ocusp}_k \, \mathbb{M}$ and $\sigma=\mathcal{D}_{\psi^{-1}}(\pi)$.
Assume our Working Hypothesis \ref{odd-orthogonal-working-hypothesis}.
Let $\varphi' \in \sigma$ be an irreducible $\psi^{-1}_{N'}$-generic cuspidal automorphic representation of $G'$ and $\varphi \in \mathcal{A}(\pi)$
Suppose that $\mathcal{W}^{\psi^{-1}_{N'}}(\cdot,\varphi')=\prod_v W'_v$ and $\mathcal{W}^{\psi_{N_M}}(\cdot,\varphi)=\prod_v W_v$.
Then for any sufficiently large finite set of places $S$ we have
\label{GG-Whittaker-OddOrthogonal}
\begin{equation}
\label{Before-Limit-Odd-Orthogoal}
 \langle \varphi', {\rm GG}(\mathcal{E}(s,\varphi)) \rangle_{G'}=\frac{1}{2}\cdot \frac{{\rm vol}  (N'(\mathcal{O}_S) \backslash N'(F_S))}{\prod_{j=1}^n \zeta^S_F(2j)}
 \cdot \frac{L(s+\frac{1}{2},\sigma\times \pi)}{L(2s+1,\pi, \wedge^2)} \prod_{v \in S} J(s,W'_v,W_v),
\end{equation}
where on the right-hand side we take the unnormalized Tamagawa measure on $G'(F_S)$, $N'(F_S)$, $V(F_S)$ and $V_{\gamma}(F_S)$
(which are independent of the choice of gauge forms when $S$ is sufficiently large. Refer to \cite[\S 6]{LM17}.)
\end{proposition}

\begin{proof}
It is worthwhile to mention that the volume of $G'(F) \backslash G'(\mathbb{A})$, appearing on the right-hand side of \eqref{Petterson-Innerproduct},
is $2$ equal to the Tamagawa number when we use the Tamagawa measure as in \cite{LM16,LM17}.
We apply \cite[Theorem 10.3, (10.4)]{GRS11} in order to unfold
\begin{equation}
 \label{odd-orthgonal-unfolding} 
  \langle \varphi', {\rm GG}(\mathcal{E}(s,\varphi)) \rangle_{G'}=\frac{1}{2} \cdot {\rm vol}(N'(F) \backslash N'(\mathbb{A})) \int_{N'(\mathbb{A} \backslash G'(\mathbb{A}))} 
  \mathcal{W}^{\psi^{-1}_{N'}}(g,\varphi') A^{\psi}(s,g,\varphi)\, dg.
\end{equation}
Similar to \cite[Propositions 5.7 and 8.4]{LM16}, our desired result is a consequence of \eqref{odd-orthgonal-unfolding} in conjunction with 
the unramified computation \cite{Gin90} (cf. \cite[(10.59)]{GRS11}):
\begin{equation}
\label{odd-orthgonal-spherical-comp} 
 J(s,{W'_v}^{\circ},W^{\circ}_v)=\frac{{\rm vol}(K')}{{\rm vol}(N'(F_v)\cap K')}\frac{L(s+\frac{1}{2},\sigma_v\times \pi_v)}{L(2s+1,\pi_v,\wedge^2)}
\end{equation}
with $K'=K_{{\rm GL}_{2n+1}(F_v)}\cap G'(F_v)$. Since we use the Tamagawa measure, the factor $\Delta^S_{{\rm SO}_{2n+1}}(1)^{-1}$
arises, where $\Delta^S_{{\rm SO}_{2n+1}}(s)=L(M^{\vee}(s))$ is the partial $L$-function of the dual motive $M^{\vee}$ 
of the motive $M$ associated to ${\rm SO}_{2n+1}$ by Gross \cite{Gro97}. In light of \citelist{\cite{Xue17}*{Notation and Convention} \cite{II10}}, we have 
\[
\Delta^S_{{\rm SO}_{2n+1}}(1)=\prod_{j=1}^n \zeta^S_F(2j)
\]
for any finite set of places $S$.
\end{proof}

\subsection{Petersson inner products and Fourier coefficients}
Let $F$ be a local field. Let $\pi \in {\rm Irr}_{\rm gen,orth} \, \mathbb{M}$.
We say that $\pi$ is {\it good} if the following conditions are satisfied for all $\psi$;
\begin{enumerate}[label=$(\mathrm{\roman*})$]
\item $\mathcal{D}_{\psi}(\pi)$ is irreducible.
\item $J(s,W',W)$ is holomorphic at $s=\frac{1}{2}$ for any $W' \in \mathcal{D}_{\psi^{-1}}(\pi)$ and $W \in {\rm Ind}(\mathbb{W}^{\psi_{N_M}}(\pi))$.
\item For any $W' \in \mathcal{D}_{\psi^{-1}}(\pi)$,\\
$J(\frac{1}{2},W',W)$ factors through the map $W \mapsto (A^{\psi}(-\frac{1}{2},\cdot,M(\frac{1}{2},\pi)W))$.
\end{enumerate}

We verify from \eqref{odd-orthogonal-G'-equivalent} that for $x \in G'$ we have
\[
 J(s,\sigma(x)W',I(s,x)W)=J(s,W',W),
\]
where $\sigma=\mathcal{D}_{\psi^{-1}}(\pi)$. Therefore if $\pi$ is good, there is a non-degenerate $G'$-invariant pairing $[\cdot,\cdot]$ on
$\mathcal{D}_{\psi^{-1}}(\pi) \times \mathcal{D}_{\psi}(\pi)$ such that
\[
 J \left( \frac{1}{2},W',W \right)=\left[ W', A^{\psi}\left( -\frac{1}{2},\cdot, M\left( \frac{1}{2},\pi \right)W \right) \right]
\]
for any $W' \in \mathcal{D}_{\psi^{-1}}(\pi)$ and $W \in {\rm Ind}(\mathbb{W}^{\psi_{N_M}}(\pi))$. By \cite[\S 2]{LM15},
when $\pi$ is good, there exists $c_{\pi}$ satisfying
\begin{equation}
\label{cpi-equation-odd-orthogoal}
 \int^{\rm st}_{N'} \left( \frac{1}{2},\sigma(u)W',W \right) \psi_{N'}(u) \, du=c_{\pi}W'(e) A^{\psi} \left( -\frac{1}{2}, e, M\left( \frac{1}{2}, \pi \right)W \right).
\end{equation}

\begin{proposition}
\label{unramified-constant-odd-orthogonal}
Suppose that $F$ is $p$-adic, $\pi \in {\rm Irr}_{\rm gen,orth} \, \mathbb{M}$ and $\psi$ are unramified.
Then $\sigma=\mathcal{D}_{\psi^{-1}}(\pi)$ is irreducible and unramified. Let $W^{\circ} \in {\rm Ind}(\mathbb{W}^{\psi_{N_M}}(\pi))$ be the
standard unramified vector and let ${W'}^{\circ}$ be $K'$-invariant with ${W'}^{\circ}(e)=1$.
Then holds with $c_{\pi}=1$.
\end{proposition}

\begin{proof}
With \cite[Theorem 5.6]{GRS11} in hand, the proof of \cite[Lemma 5.4]{LM17} can be carried over verbatim to the first part of our current setting. Assuming ${\rm vol}(U \cap K)=1$ as in \eqref{unramified-odd-orth-intertwining}, we know that
\[
 {\rm vol}(K')=(\prod_{j=1}^n \zeta^S_F(2j))^{-1}
\]
Combining it and \cite[Proposition 2.14]{LM15} leads to
\[
 \int^{\rm st}_{N'} J \left( \frac{1}{2},\sigma(u){W'}^{\circ},W^{\circ} \right) \psi^{-1}_{N'}(u) \,du=\frac{{\rm vol}(N' \cap K')}{{\rm vol}(K')}\frac{J(\frac{1}{2},{W'}^{\circ},W^{\circ})}{L(1,\sigma,{\rm Ad})}.
\]
Using the description of $\sigma$ in this case \cite[Theorem 5.6]{GRS11} coupled with \cite[\S 7]{GGP12}, we infer that
$L(1,\sigma,{\rm Ad})=L(1,\pi,{\rm Sym}^2)$ and $L(s,\sigma\times \pi)=L(s,\pi\times\pi)$. 
We recall that $\pi=\pi^{\vee}$ in our case.
Therefore \eqref{unramified-odd-orth-intertwining} aligned with of \eqref{odd-orthgonal-spherical-comp} implies that 
\[
  \frac{1}{L(1,\pi,{\rm Sym}^2)}\frac{L(1,\pi \times \pi)}{L(2,\pi,\wedge^2)}=\frac{L(1,\pi,\wedge^2)}{L(2,\pi,\wedge^2)}=M\left( \frac{1}{2},\pi \right)W^{\circ}(e)
  ={W'}^{\circ}(e)A^{\psi} \left( -\frac{1}{2},e,M\left( \frac{1}{2},\pi \right)W^{\circ}\right)
\]
as requested.
\end{proof}

We note that a priori $c_{\pi}$ implicitly depends on the choice of Haar measure on $G'$ (the former used in the definition of $J(s,W',W)$) and $U$
(the latter used in the definition of the intertwining operator), but not on any other groups (refer to \cite[Remark 5.2]{LM17} for further details).
However the Lie algebras of $G'$ and $U$ are both canonically isomorphic as vector spaces to
\[
 \{ X \in {\rm Mat}_{2n,2n}(F) \,|\, {X^t}\,w_{2n}=-w_{2n}X \}.
\]
We say that the Haar measures on $G'$ and $U$ are {\it compatible} if they are defined as unnormalized Tamagawa measure with respect to
the same gauge form via the above identification of Lie algebras. We can now see that $c_{\pi}$ does not depend on any choice of measure provided
that we choose compatible measures on $G'$ and $U$.

\begin{proposition} 
\label{odd-orthogonal-good-repn}
Let $\pi \in {\rm Ocusp}_k\,\mathbb{M}$ and $\sigma=\mathcal{D}_{\psi^{-1}}(\pi)$.
Assume our Working Hypothesis \ref{odd-orthogonal-working-hypothesis}. 
Let $\varphi' \in \sigma$ be an irreducible $\psi^{-1}_{N'}$-generic cuspidal automorphic representation of $G'$ and $\varphi \in \mathcal{A}(\pi)$.
Suppose that $\mathcal{W}^{\psi^{-1}_{N'}}(\cdot,\varphi')=\prod_v W'_v$ and $\mathcal{W}^{\psi_{N_M}}(\cdot,\varphi)=\prod_v W_v$.
Then for all $v$ $\pi_v$ is good. Moreover for any sufficiently large finite set of places $S$, we have
\begin{equation}
\label{EulerProducts-Odd-Orthogonal}
 \langle \varphi', {\rm GG}(\mathcal{E}_{-k}\varphi) \rangle_{G'}=\frac{1}{2}\cdot \frac{{\rm vol}(N'(\mathcal{O}_S \backslash N'(F_S)))}{\prod_{j=1}^n \zeta^S_F(2j)} \cdot n^S_{-k,1}
 \prod_{v \in S} J\left( \frac{1}{2},W'_v,W_v \right),
\end{equation}
where
\[
 n^S_{-k,1}=\frac{\lim_{s \rightarrow 1}(s-1)^kL^S(s,\pi \times \pi)}{L^S(2,\pi,\wedge^2)}.
\]
\end{proposition}

\begin{proof}
We verify that $\pi_v$ is good for all $v$. To this end, it can be deduced from \eqref{GG-Whittaker-OddOrthogonal-decomp}
together with the irreducibility of the global descent that $\mathcal{D}_{\psi^{-1}}(\pi_v)$ is irreducible for all $v$.

\par
When $\pi \in {\rm Ocusp}_k \, \mathbb{M}$ and $\sigma=\mathcal{D}_{\psi^{-1}}(\pi)$, we obtain $L^S(s,\sigma\times \pi)=L^S(s,\pi\times\pi)$,
which has a pole of order $k$ at $s=\frac{1}{2}$. On the one hand, $J(s,W'_v,W_v)$ is non-vanishing at $s=\frac{1}{2}$
for appropriate $W'_v$ and $W_v$. Therefore the right-hand side of \eqref{Before-Limit-Odd-Orthogoal} (when $\varphi'\in \sigma$)
has a pole of order at least $k$ at $s=\frac{1}{2}$ for suitable $\varphi'$ and $\varphi$. On the other hand, the left-hand side of \eqref{Before-Limit-Odd-Orthogoal}
has a pole of order at most $k$ at $s=\frac{1}{2}$ because this is true for $\mathcal{E}(s,\varphi)$ and $\varphi'$ is rapidly decreasing.
Multiplying \eqref{Before-Limit-Odd-Orthogoal} by $(s-\frac{1}{2})^k$ and then taking the limit as $s \rightarrow \frac{1}{2}$,
we conclude that $J(s,W'_v,W_v)$ is holomorphic at $s=\frac{1}{2}$ for all $v$, and for $\varphi' \in \sigma$ and $\varphi \in \mathcal{A}(\pi)$ we get
\begin{equation}
\label{After-Limit-Odd-Orthogoal}
 \langle \varphi', {\rm GG}(\mathcal{E}_{-k}\varphi) \rangle_{G'}=\frac{1}{2}\cdot \frac{{\rm vol}(N'(\mathcal{O}_S \backslash N'(F_S)))}{\prod_{j=1}^n \zeta^S_F(2j)} \cdot n^S_{-k,1}
 \prod_{v \in S} J\left( \frac{1}{2},W'_v,W_v \right).
\end{equation}
\par
We fix a place $v_0$. We now assume that $W_{v_0}$ such that $A^{\psi}\left(-\frac{1}{2},\cdot,M\left( \frac{1}{2},\pi_{v_0}\right)W_{v_0} \right) \equiv 0$.
Then by \eqref{GG-Whittaker-OddOrthogonal-decomp}, we see that $\mathcal{W^{\psi_{N'}}}(\cdot, {\rm GG}(\mathcal{E}_{-k}\varphi))\equiv 0$, 
which in turn implies that ${\rm GG}(\mathcal{E}_{-k}\varphi)\equiv 0$ owing to the irreducibility and genericity of the descent.
We conclude from \eqref{After-Limit-Odd-Orthogoal} that $J(\frac{1}{2},W'_{v_0},W_{v_0})=0$. Henceforth $\pi_{v_0}$ is good.
\end{proof}

While we expect that any irreducible generic representation of orthogonal type $\pi_v$ is good, the proof that we give under the assumption 
that $\pi_v$ is the local component of $\pi \in {\rm Scusp}_k \,\mathbb{M}$ is valid at the very least for the case of our interest,
which pertains to the global period integrals. Our proof here requires a local to global argument. Whereas the global approach is
overkill and indirect, one may directly prove the desired result by using a local functional equation \cite[Proposition 5.3]{LM17}. For the sake of brevity,
we only include this indirect mean in keeping with the spirit of \cite[Propositions 5.7 and 8.4]{LM16}.

\begin{theorem} 
\label{Main-Result-Odd-Orthogonal-Theorem}
Let $\pi \in {\rm Ocusp}_k \, \mathbb{M}$ and $\sigma=\mathcal{D}_{\psi}(\pi)$. Assume our Working Hypothesis \ref{odd-orthogonal-working-hypothesis}. 
Let $S$ be a finite set of places including all the archimedean places such that $\pi$ and $\psi$ are all unramified outside $S$.
Then for any $\varphi \in \sigma$ and $\varphi^{\vee} \in \sigma^{\vee}$
which are fixed under $K'_v$ for all $v \not\in S$, we have
\begin{multline}
\label{Main-Result-Odd-Orthogonal}
 \quad \quad
 \mathcal{W}^{\psi_{N'}}(\varphi) \mathcal{W}^{\psi^{-1}_{N'}}(\varphi^{\vee})=2^{1-k} \left( \prod_{v \in S} c^{-1}_{\pi_v} \right) \frac{\prod_{j=1}^n \zeta^S_F(2j)}{L^S(1,\pi,{\rm Sym}^2)}\\
 \times ({\rm vol}  (N'(\mathcal{O}_S) \backslash N'(F_S))    )^{-1} \int_{N'(F_S)}^{\rm st} \langle \sigma(u)\varphi,\varphi^{\vee} \rangle_{G'} \psi^{-1}_{N'}(u) \,du. \quad \quad 
\end{multline}
\end{theorem}

\begin{proof}
Since the descent of $\pi$ is the space $\mathcal{D}_{\psi_{N_{\mathbb{M}}}}(\pi)$ spanned by ${\rm GG}(\mathcal{E}_{-k}\varphi^{\flat})$ for $\varphi^{\flat} \in \mathcal{A}(\pi)$,
we initially tackle with this for the case $\varphi^{\vee}={\rm GG}(\mathcal{E}_{-k}\varphi^{\sharp})$ with $\varphi^{\sharp} \in \mathcal{A}(\pi)$, and this identity \eqref{Main-Result-Odd-Orthogonal}
then extends via linearity to all $\varphi^{\vee}$. In view of Proposition \ref{GG-Whittaker-OddOrthogonal}, we find that
\[
 \mathcal{W}^{\psi_{N'}}(\varphi)\mathcal{W}^{\psi^{-1}_{N'}}({\rm GG}(\mathcal{E}_{-k}\varphi^{\sharp}))
 =\mathcal{W}^{\psi_{N'}}(\varphi)\frac{m^S_{-k,1}(\pi)}{2^k} \prod_{v \in S} A^{\psi^{-1}} \left(-\frac{1}{2},e,M\left( \frac{1}{2},\pi_v \right), W^{\sharp}_v \right).
\]
Counting on the fact $\pi_v$ is good in Proposition \ref{odd-orthogonal-good-repn}, we exploit Proposition \ref{unramified-constant-odd-orthogonal}
to insert the identity \eqref{cpi-equation-odd-orthogoal} for all $v$ with $c_{\pi_v}=1$ for almost all $v$. In doing so, the above can be read as
\begin{equation}
\label{Odd-Orthogonal-Product-Form-one}
 \mathcal{W}^{\psi_{N'}}(\varphi)\mathcal{W}^{\psi^{-1}_{N'}}({\rm GG}(\mathcal{E}_{-k}\varphi^{\sharp}))
 =\frac{m^S_{-k,1}(\pi)}{2^k} \left( \prod_{v \in S} c^{-1}_{\pi_v} \int^{\rm st}_{N'(F_v)} J \left( \frac{1}{2},\sigma_v(u_v)W_v,W^{\sharp}_v \right) \psi^{-1}_{N'(F_v)}(u_v) \,du_v \right).
\end{equation}
Integrating \eqref{EulerProducts-Odd-Orthogonal} over $N'(F_S)$ for $v \in S$, a product of regularized integrals is turned into
\begin{multline}
\label{Odd-Orthogonal-Product-Form-two}
 \prod_{v \in S} \int^{\rm st}_{N'(F_v)} J\left( \frac{1}{2},\sigma_v(u_v)W_v,W^{\sharp}_v \right) \psi^{-1}_{N'(F_v)}(u_v) \, du_v\\
 =\frac{2\prod_{j=1}^n \zeta^S_F(2j)}{{\rm vol}(N'(\mathcal{O}_S)\backslash N'(F_S))\cdot n^S_{-k,1}} \int^{\rm st}_{N'(F_S)} \langle
 \sigma(u)\varphi, {\rm GG}(\mathcal{E}_{-k}\varphi^{\sharp}) \rangle_{G'} \psi^{-1}_{N'(F_v)}(u) \, du.
\end{multline}
The factorization $L^S(s,\pi\times\pi)=L(s,\pi\times\pi^{\vee})=L^S(s,\pi,{\rm Sym}^2)L^S(s,\pi,\wedge^2)$ enables us to simplify the ratio
\begin{equation}
\label{Odd-Orthogonal-Ratio}
 m^S_{-k,1} / n^S_{-k,1}=1/L^S(1,\pi,{\rm Sym}^2)
\end{equation}
whose partial $L$-function appearing in the denominator is actually the adjoint $L$-function $L^S(s,\sigma,{\rm Ad})$ for $\sigma$. 
Combining \eqref{Odd-Orthogonal-Product-Form-one} and \eqref{Odd-Orthogonal-Product-Form-two} with \eqref{Odd-Orthogonal-Ratio}, we are led to the conclusion.
\end{proof}

\begin{acknowledgements}
This paper is inspired from \cite{LM16} and owes much to \cite{LM16}, when the author learned it in the workshop proceeding in honor of Prof. James Cogdell's 60th birthday. It is pleasure to revisit remaining cases and dedicate this paper to 
Prof. James Cogdell on the occasion of his 70th birthday. We wish Prof. James Cogdell many more happy and healthy years. 
We would like to thank to Prof. Sug Woo Shin for drawing author's attention to residues of Rankin--Selberg type integrals.
This work was supported by the National Research Foundation of Korea (NRF) grant 
funded by the Korea government (No. RS-2023-00209992). 
\end{acknowledgements}

\begin{conflicts of interest}
The authors state that there is no conflict of interest.
\end{conflicts of interest}

\begin{data availability}
This manuscript has no associated data.
\end{data availability}

 \bibliographystyle{amsplain}
 \bibliography{references}

\begin{bibdiv}
\begin{biblist}

\bib{AGJ09}{article}{
      author={Aizenbud, Avraham},
      author={Gourevitch, Dmitry},
      author={Jacquet, Herv\'{e}},
       title={Uniqueness of {S}halika functionals: the {A}rchimedean case},
        date={2009},
     journal={Pacific J. Math.},
      volume={243},
      number={2},
       pages={201\ndash 212},
}

\bib{BLX24}{article}{
      author={Boisseau, Paul},
      author={Lu, Weixiao},
      author={Xue, Hang},
       title={The global gan--gross--prasad conjecture for fourier-jacobi
  periods on unitary groups},
        date={2024},
        note={Available at \url{https://arxiv.org/abs/2404.07342}},
}

\bib{BG92}{article}{
      author={Bump, Daniel},
      author={Ginzburg, David},
       title={Symmetric square {$L$}-functions on {${\rm GL}(r)$}},
        date={1992},
     journal={Ann. of Math. (2)},
      volume={136},
      number={1},
       pages={137\ndash 205},
}

\bib{FJ93}{article}{
      author={Friedberg, Solomon},
      author={Jacquet, Herv\'{e}},
       title={Linear periods},
        date={1993},
     journal={J. Reine Angew. Math.},
      volume={443},
       pages={91\ndash 139},
}

\bib{GGP12}{incollection}{
      author={Gan, Wee~Teck},
      author={Gross, Benedict~H.},
      author={Prasad, Dipendra},
       title={Symplectic local root numbers, central critical {$L$} values, and
  restriction problems in the representation theory of classical groups},
        date={2012},
       pages={1\ndash 109},
        note={Sur les conjectures de Gross et Prasad. I},
}

\bib{Gin90}{article}{
      author={Ginzburg, David},
       title={{$L$}-functions for {${\rm SO}_n\times {\rm GL}_k$}},
        date={1990},
     journal={J. Reine Angew. Math.},
      volume={405},
       pages={156\ndash 180},
}

\bib{GJS12}{article}{
      author={Ginzburg, David},
      author={Jiang, Dihua},
      author={Soudry, David},
       title={On correspondences between certain automorphic forms on {${\rm
  Sp}_{4n}(\Bbb A)$} and {$\widetilde{\rm Sp}_{2n}(\Bbb A)$}},
        date={2012},
     journal={Israel J. Math.},
      volume={192},
      number={2},
       pages={951\ndash 1008},
}

\bib{GRS98}{article}{
      author={Ginzburg, David},
      author={Rallis, Stephen},
      author={Soudry, David},
       title={{$L$}-functions for symplectic groups},
        date={1998},
     journal={Bull. Soc. Math. France},
      volume={126},
      number={2},
       pages={181\ndash 244},
}

\bib{GRS99}{article}{
      author={Ginzburg, David},
      author={Rallis, Stephen},
      author={Soudry, David},
       title={On a correspondence between cuspidal representations of {${\rm
  GL}_{2n}$} and {$\widetilde{\rm Sp}_{2n}$}},
        date={1999},
     journal={J. Amer. Math. Soc.},
      volume={12},
      number={3},
       pages={849\ndash 907},
}

\bib{GRS11}{book}{
      author={Ginzburg, David},
      author={Rallis, Stephen},
      author={Soudry, David},
       title={The descent map from automorphic representations of {${\rm
  GL}(n)$} to classical groups},
   publisher={World Scientific Publishing Co. Pte. Ltd., Hackensack, NJ},
        date={2011},
}

\bib{Gro97}{article}{
      author={Gross, Benedict~H.},
       title={On the motive of a reductive group},
        date={1997},
     journal={Invent. Math.},
      volume={130},
      number={2},
       pages={287\ndash 313},
}

\bib{GP92}{article}{
      author={Gross, Benedict~H.},
      author={Prasad, Dipendra},
       title={On the decomposition of a representation of {${\rm SO}_n$} when
  restricted to {${\rm SO}_{n-1}$}},
        date={1992},
     journal={Canad. J. Math.},
      volume={44},
      number={5},
       pages={974\ndash 1002},
}

\bib{Har14}{article}{
      author={Harris, R.~Neal},
       title={The refined {G}ross-{P}rasad conjecture for unitary groups},
        date={2014},
     journal={Int. Math. Res. Not. IMRN},
      number={2},
       pages={303\ndash 389},
}

\bib{HII08}{article}{
      author={Hiraga, Kaoru},
      author={Ichino, Atsushi},
      author={Ikeda, Tamotsu},
       title={Formal degrees and adjoint {$\gamma$}-factors},
        date={2008},
     journal={J. Amer. Math. Soc.},
      volume={21},
      number={1},
       pages={283\ndash 304},
}

\bib{HJ24}{article}{
      author={Humphries, Peter},
      author={Jo, Yeongseong},
       title={Test vectors for archimedean period integrals},
        date={2024},
     journal={Publ. Mat.},
      volume={68},
      number={1},
       pages={139\ndash 185},
}

\bib{II10}{article}{
      author={Ichino, Atsushi},
      author={Ikeda, Tamutsu},
       title={On the periods of automorphic forms on special orthogonal groups
  and the {G}ross-{P}rasad conjecture},
        date={2010},
     journal={Geom. Funct. Anal.},
      volume={19},
      number={5},
       pages={1378\ndash 1425},
}

\bib{ILM17}{article}{
      author={Ichino, Atsushi},
      author={Lapid, Erez},
      author={Mao, Zhengyu},
       title={On the formal degrees of square-integrable representations of odd
  special orthogonal and metaplectic groups},
        date={2017},
     journal={Duke Math. J.},
      volume={166},
      number={7},
       pages={1301\ndash 1348},
}

\bib{Jac09}{incollection}{
      author={Jacquet, H.},
       title={Archimedean {R}ankin-{S}elberg integrals},
        date={2009},
   booktitle={Automorphic forms and {$L$}-functions {II}. {L}ocal aspects},
      series={Contemp. Math.},
      volume={489},
   publisher={Amer. Math. Soc., Providence, RI},
       pages={57\ndash 172},
}

\bib{JPSS83}{article}{
      author={Jacquet, H.},
      author={Piatetskii-Shapiro, I.~I.},
      author={Shalika, J.~A.},
       title={Rankin-{S}elberg convolutions},
        date={1983},
     journal={Amer. J. Math.},
      volume={105},
      number={2},
       pages={367\ndash 464},
}

\bib{JR96}{article}{
      author={Jacquet, H.},
      author={Rallis, S.},
       title={Uniqueness of linear periods},
        date={1996},
     journal={Compositio Math.},
      volume={102},
      number={1},
       pages={65\ndash 123},
}

\bib{JS81a}{article}{
      author={Jacquet, H.},
      author={Shalika, J.~A.},
       title={On {E}uler products and the classification of automorphic forms.
  {I}},
        date={1981},
     journal={Amer. J. Math.},
      volume={103},
      number={3},
       pages={499\ndash 558},
}

\bib{JS81b}{article}{
      author={Jacquet, H.},
      author={Shalika, J.~A.},
       title={On {E}uler products and the classification of automorphic forms.
  {II}},
        date={1981},
     journal={Amer. J. Math.},
      volume={103},
      number={4},
       pages={777\ndash 815},
}

\bib{Jac89}{article}{
      author={Jacquet, Herv\'{e}},
       title={On the nonvanishing of some {$L$}-functions},
        date={1987},
     journal={Proc. Indian Acad. Sci. Math. Sci.},
      volume={97},
      number={1-3},
       pages={117\ndash 155},
}

\bib{JS07}{article}{
      author={Jiang, Dihua},
      author={Soudry, David},
       title={On the genericity of cuspidal automorphic forms of {${\rm
  SO}(2n+1)$}. {II}},
        date={2007},
     journal={Compos. Math.},
      volume={143},
      number={3},
       pages={721\ndash 748},
}

\bib{Jo20}{article}{
      author={Jo, Yeongseong},
       title={Derivatives and exceptional poles of the local exterior square
  {$L$}-function for {$GL_m$}},
        date={2020},
     journal={Math. Z.},
      volume={294},
      number={3-4},
       pages={1687\ndash 1725},
}

\bib{Jo23}{article}{
      author={Jo, Yeongseong},
       title={The local period integrals and essential vectors},
        date={2023},
     journal={Math. Nachr.},
      volume={296},
      number={1},
       pages={339\ndash 367},
}

\bib{Kap10}{article}{
      author={Kaplan, Eyal},
       title={An invariant theory approach for the unramified computation of
  {R}ankin-{S}elberg integrals for quasi-split {${\rm SO}_{2n}\times {\rm
  GL}_n$}},
        date={2010},
     journal={J. Number Theory},
      volume={130},
      number={8},
       pages={1801\ndash 1817},
}

\bib{Kap12}{article}{
      author={Kaplan, Eyal},
       title={The unramified computation of {R}ankin-{S}elberg integrals for
  {$SO_{2l}\times GL_n$}},
        date={2012},
     journal={Israel J. Math.},
      volume={191},
      number={1},
       pages={137\ndash 184},
}

\bib{Kap15}{article}{
      author={Kaplan, Eyal},
       title={Complementary results on the {R}ankin-{S}elberg gamma factors of
  classical groups},
        date={2015},
     journal={J. Number Theory},
      volume={146},
       pages={390\ndash 447},
}

\bib{Kap16}{article}{
      author={Kaplan, Eyal},
       title={Representations distinguished by pairs of exceptional
  representations and a conjecture of {S}avin},
        date={2016},
     journal={Int. Math. Res. Not. IMRN},
      number={2},
       pages={604\ndash 643},
}

\bib{Kap17}{article}{
      author={Kaplan, Eyal},
       title={The characterization of theta-distinguished representations of
  {${\rm GL}(n)$}},
        date={2017},
     journal={Israel J. Math.},
      volume={222},
      number={2},
       pages={551\ndash 598},
}

\bib{KY17}{article}{
      author={Kaplan, Eyal},
      author={Yamana, Shunsuke},
       title={Twisted symmetric square {$L$}-functions for {${\rm GL}_n$} and
  invariant trilinear forms},
        date={2017},
     journal={Math. Z.},
      volume={285},
      number={3-4},
       pages={739\ndash 793},
}

\bib{LM22}{book}{
      author={Lapid, Erez},
       title={Some perspectives on eisenstein series},
        date={(2024)},
        note={To appear in Proceedings of Symposia in Pure Mathematics : The
  2022 IH{\'E}S Summer School, available at
  \url{https://arxiv.org/abs/2204.02914}},
}

\bib{LM15}{article}{
      author={Lapid, Erez},
      author={Mao, Zhengyu},
       title={A conjecture on {W}hittaker-{F}ourier coefficients of cusp
  forms},
        date={2015},
     journal={J. Number Theory},
      volume={146},
       pages={448\ndash 505},
}

\bib{LM16}{incollection}{
      author={Lapid, Erez},
      author={Mao, Zhengyu},
       title={On {W}hittaker-{F}ourier coefficients of automorphic forms on
  unitary groups: reduction to a local identity},
        date={2016},
   booktitle={Advances in the theory of automorphic forms and their
  {$L$}-functions},
      series={Contemp. Math.},
      volume={664},
   publisher={Amer. Math. Soc., Providence, RI},
       pages={295\ndash 320},
}

\bib{LM17-2}{article}{
      author={Lapid, Erez},
      author={Mao, Zhengyu},
       title={On an analogue of the {I}chino-{I}keda conjecture for {W}hittaker
  coefficients on the metaplectic group},
        date={2017},
     journal={Algebra Number Theory},
      volume={11},
      number={3},
       pages={713\ndash 765},
}

\bib{LM17}{article}{
      author={Lapid, Erez},
      author={Mao, Zhengyu},
       title={Whittaker-{F}ourier coefficients of cusp forms on
  {$\widetilde{\rm Sp}_n$}: reduction to a local statement},
        date={2017},
     journal={Amer. J. Math.},
      volume={139},
      number={1},
       pages={1\ndash 55},
}

\bib{Liu16}{article}{
      author={Liu, Yifeng},
       title={Refined global {G}an-{G}ross-{P}rasad conjecture for {B}essel
  periods},
        date={2016},
     journal={J. Reine Angew. Math.},
      volume={717},
       pages={133\ndash 194},
}

\bib{Luo24}{article}{
      author={Luo, Caihua},
       title={Holomorphy of normalized intertwining operators for certain
  induced representations},
        date={2024},
        note={To appear in Mathematische Annalen, available at
  \url{https://arxiv.org/abs/2112.03531}},
}

\bib{Mat17}{article}{
      author={Matringe, Nadir},
       title={Shalika periods and parabolic induction for {$GL(n)$} over a
  non-archimedean local field},
        date={2017},
     journal={Bull. Lond. Math. Soc.},
      volume={49},
      number={3},
       pages={417\ndash 427},
}

\bib{Mor18}{article}{
      author={Morimoto, Kazuki},
       title={On the irreducibility of global descents for even unitary groups
  and its applications},
        date={2018},
     journal={Trans. Amer. Math. Soc.},
      volume={370},
      number={9},
       pages={6245\ndash 6295},
}

\bib{Mor22}{article}{
      author={Morimoto, Kazuki},
       title={On a certain local identity for {L}apid-{M}ao's conjecture and
  formal degree conjecture: even unitary group case},
        date={2022},
     journal={J. Inst. Math. Jussieu},
      volume={21},
      number={4},
       pages={1107\ndash 1161},
}

\bib{Mor24}{article}{
      author={Morimoto, Kazuki},
       title={On ichino-ikeda type formula of whittaker periods for unitary
  groups},
        date={2024},
        note={Available at \url{https://arxiv.org/abs/2403.19166}},
}

\bib{Qiu19}{article}{
      author={Qiu, Yannan},
       title={The {W}hittaker period formula on metaplectic {$\rm SL_2$}},
        date={2019},
     journal={Trans. Amer. Math. Soc.},
      volume={371},
      number={2},
       pages={1083\ndash 1117},
}

\bib{Rao93}{article}{
      author={Ranga~Rao, R.},
       title={On some explicit formulas in the theory of {W}eil
  representation},
        date={1993},
     journal={Pacific J. Math.},
      volume={157},
      number={2},
       pages={335\ndash 371},
}

\bib{SV17}{article}{
      author={Sakellaridis, Yiannis},
      author={Venkatesh, Akshay},
       title={Periods and harmonic analysis on spherical varieties},
        date={2017},
     journal={Ast\'{e}risque},
      number={396},
       pages={viii+360},
}

\bib{Sha88}{article}{
      author={Shahidi, Freydoon},
       title={On the {R}amanujan conjecture and finiteness of poles for certain
  {$L$}-functions},
        date={1988},
     journal={Ann. of Math. (2)},
      volume={127},
      number={3},
       pages={547\ndash 584},
}

\bib{Sou93}{article}{
      author={Soudry, David},
       title={Rankin-{S}elberg convolutions for {${\rm SO}_{2l+1}\times{\rm
  GL}_n$}: local theory},
        date={1993},
     journal={Mem. Amer. Math. Soc.},
      volume={105},
      number={500},
       pages={vi+100},
}

\bib{Szp13}{article}{
      author={Szpruch, Dani},
       title={Some irreducibility theorems of parabolic induction on the
  metaplectic group via the {L}anglands-{S}hahidi method},
        date={2013},
     journal={Israel J. Math.},
      volume={195},
      number={2},
       pages={897\ndash 971},
}

\bib{Tak14}{article}{
      author={Takeda, Shuichiro},
       title={The twisted symmetric square {$L$}-function of {${\rm GL}(r)$}},
        date={2014},
     journal={Duke Math. J.},
      volume={163},
      number={1},
       pages={175\ndash 266},
}

\bib{Wal81}{article}{
      author={Waldspurger, J.-L.},
       title={Sur les coefficients de {F}ourier des formes modulaires de poids
  demi-entier},
        date={1981},
     journal={J. Math. Pures Appl. (9)},
      volume={60},
      number={4},
       pages={375\ndash 484},
}

\bib{Xue17}{article}{
      author={Xue, Hang},
       title={Refined global {G}an-{G}ross-{P}rasad conjecture for
  {F}ourier-{J}acobi periods on symplectic groups},
        date={2017},
     journal={Compos. Math.},
      volume={153},
      number={1},
       pages={68\ndash 131},
}

\bib{Zha14}{article}{
      author={Zhang, Wei},
       title={Automorphic period and the central value of {R}ankin-{S}elberg
  {L}-function},
        date={2014},
     journal={J. Amer. Math. Soc.},
      volume={27},
      number={2},
       pages={541\ndash 612},
}

\end{biblist}
\end{bibdiv}

\end{document}